\documentclass[10pt,reqno,final]{amsart}

\usepackage[utf8]{inputenc}
\usepackage{tgtermes}

\usepackage[margin=30mm]{geometry}
\usepackage{mathtools,amsmath,mathabx,tikz-cd}
\usepackage{amssymb,mathrsfs,stmaryrd}

\usepackage[backend=biber,bibencoding=utf8,
style=numeric,sortcites,maxbibnames=50,
url=true,doi=false,eprint=true,isbn=false,
hyperref=auto,backref=false]{biblatex}
\usepackage{setspace} 
\usepackage{indentfirst}
\usepackage{pdfpages}

\usepackage[notcite,notref,color,draft]{showkeys}
\definecolor{refkey}{gray}{.50}
\definecolor{labelkey}{rgb}{1,0,0}

\usepackage[right,mathlines]{lineno}

\usepackage[hidelinks]{hyperref}

\newtheorem{theorem}{Theorem}[section]
\newtheorem{lemma}[theorem]{Lemma}
\newtheorem{proposition}[theorem]{Proposition}
\newtheorem{corollary}[theorem]{Corollary}
\newtheorem{definition}[theorem]{Definition}
\newtheorem{remark}[theorem]{Remark}
\newtheorem{example}[theorem]{Example}

\newtheoremstyle{intro}% name
  {12pt}%      Space above, empty = `usual value'
  {12pt}%      Space below
  {\itshape}% Body font
  {}%         Indent amount (empty = no indent, \parindent = para indent)
  {\bfseries}% Thm head font
  {.}%        Punctuation after thm head
  {.5em}%     Space after thm head: " " = normal interword space;
  {}% Thm head spec

\theoremstyle{intro}
\newtheorem{introthm}{Theorem}

\usepackage{macros}

\title[Formal geometry and Tamarkin--Tsygan calculi of dg manifolds]
{Formal geometry and Tamarkin--Tsygan calculi of dg manifolds}

\thanks{Research partially supported by
the National Science Foundation (award DMS-2302447),
the Simons Foundation (award MP-TSM-00002272)
and Taiwan's Ministry of Science and Technology (grants 110-2115-M-007-001-MY2 and 112-2115-M-007-016-MY3).
This material is based upon work supported by
the Swedish Research Council under grant no.\ 2021-06594
while Sti\'enon and Xu were in residence
at Institut Mittag-Leffler in Djursholm, Sweden
during the spring trimester of 2025.
}

\author{Hsuan-Yi Liao}
\address{Department of Mathematics, National Tsing Hua University}
\email{hyliao@math.nthu.edu.tw}

\author{Mathieu Stiénon}
\address{Department of Mathematics, Pennsylvania State University}
\email{stienon@psu.edu}

\author{Ping Xu}
\address{Department of Mathematics, Pennsylvania State University}
\email{ping@math.psu.edu}

\dedicatory{Dedicated to the memory of Murray Gerstenhaber}

\begin{document}

\begin{abstract}
The main goal of this paper is to study the formal geometry of dg manifolds
\`a la Fedosov.
For any dg manifold $(\cM, Q)$, we construct a Fedosov dg foliation
(or dg Lie algebroid) $\cFQ \to \cNQ$. We establish homotopy contractions
between their respective spaces of polyvector fields, differential forms,
polydifferential operators, and polyjets. As a consequence,
we prove that their respective Cartan calculi and noncommutative calculi,
in the sense of Tamarkin--Tsygan, are isomorphic.
\end{abstract}

\maketitle

\tableofcontents

% \linenumbers

\section*{Introduction}

Formal differential geometry has a long history and has been extended
to various contexts \cite{MR0287566,MR2119140}. Recall that,
in the formal differential geometry of Gel'fand--Kazhdan \cite{MR0287566},
for any smooth manifold $M$ of dimension $n$, one considers
an infinite-dimensional manifold $\tilde{M}$ consisting of formal parametrizations of $M$,
i.e.\ the $\infty$-jets at $0$ of local diffeomorphisms $\RR^n\to M$.
The Lie algebra \( W_n \) of formal vector fields acts freely and transitively on \( \tilde{M} \)
via infinitesimal reparametrizations. The Lie subalgebra of linear vector fields
can be identified with \( \gl_n \), which integrates to an action of \( \GL_n \),
such that the induced quotient space \( \tilde{M}/\GL_n \) is homotopy 
equivalent to \( M \). 
In fact, there is a section of the canonical projection \( \tilde{M}/\GL_n \to M \)
which can be constructed explicitly using an affine connection.
Specifically, an affine connection \( \nabla \) determines, for any \( x \in M \),
a local diffeomorphism \( \RR^n \to M \) mapping \( 0 \) to \( x \)
via the exponential map \( T_x M \to M \). The \( \infty \)-jets at \( 0 \)
of this exponential map yield the desired section.
Thus, instead of the infinite-dimensional manifold \( \tilde{M} \),
one can equivalently consider a finite-dimensional differential graded (dg) manifold
of the form \( (\cN, \fedosov) \), where \( \cN := T_M[1] \oplus T_M \),
and the homological vector field \( \fedosov \) can be constructed
using the formal exponent map associated with \( \nabla \).

This dg manifold $(\cN,\fedosov)$ encodes the formal neighborhood of each point of \( M \)
in the sense that the dg algebra of its functions provides a resolution of \( C^\infty(M) \).
This perspective was first developed by Fedosov in his construction of deformation quantization
for symplectic manifolds, known as the Fedosov method \cite{MR1293654,MR1376365}.
It was later adapted by Emmrich--Weinstein \cite{MR1327535} to the above setting
and further developed by Dolgushev \cite{MR2102846}, where it played a key role
in the proof of Kontsevich's formality theorem for ordinary smooth manifolds \cite{MR2062626}. 
Dolgushev’s method \cite{MR2102846, MR2199629}, formulated in terms of Fedosov dg manifolds,
has proven to be a highly effective procedure
for assembling global objects from local building blocks.
He employed it to prove both Kontsevich's formality theorem
and Tsygan's formality theorem for smooth manifolds \cite{MR2102846,MR2199629}.

Dg manifolds have recently become increasingly important.
%Dg manifolds have recently gained significant importance
%in various mathematical and physical contexts.
A dg manifold is defined as a $\ZZ$-graded manifold
equipped with a homological vector field, i.e.\ a vector field $Q$ of degree $+1$
satisfying $\schouten{Q}{Q}=0$. Originally appearing in physics as BRST operators
for describing gauge symmetries, dg manifolds (also known as $Q$-manifolds)
have since become prominent in mathematical physics, notably in the AKSZ formalism
\cite{MR1432574,MR1230027}, and also arise naturally
in geometry, Lie theory, and mathematical physics. 
Standard examples include $L_\infty$-algebras, foliations,
complex manifolds, and derived intersections.
(1).\ \textsl{$L_\infty$ algebras}
Given a finite-dimensional Lie algebra $\frakg$,
we write $\frakg[1]$ to denote the dg manifold
having $C^\infty(\frakg[1]):=\Lambda^\bullet\frakg^\vee$
as its algebra of functions and the Chevalley--Eilenberg differential
$Q:=d_{\CE}$ as its homological vector field.
This construction admits an `up to homotopy' version:
Given a $\ZZ$-graded finite-dimensional vector space
$\frakg=\bigoplus_{i\in\ZZ}\frakg_i$,
the graded manifold $\frakg[1]$ is a dg manifold,
i.e.\ admits a homological vector field,
if and only if the graded vector space $\frakg$
admits a structure of curved $L_\infty$ algebra.
(2).\ \textsl{Foliations and complex manifolds}
%\label{example-two}
Given an integrable distribution $F\subseteq T_M^{\KK}:=T_M\otimes\KK$
($\KK$ being either $\RR$ or $\CC$),
we write $F[1]$ to denote the dg manifold having
$C^\infty(F[1])=\Omega_F^\bullet:=\sections{\Lambda F^\vee}$
as its algebra of functions and the leafwise de Rham differential
$Q:=d_{\DR}^F$ as its homological vector field.
In particular, a complex manifold $X$ determines an integrable distribution
$T^{0,1}_X\subset T_X^\CC$. Then $(T^{0,1}_X[1],\bar{\partial})$
is a dg manifold: the algebra of functions is
$C^\infty(T^{0,1}_X[1])=\Omega^{0,\bullet}(X)$
and the homological vector field is the Dolbeault operator $\bar{\partial}$.
%
%(2) \textit{Complex manifolds} --- Given a complex manifold $X$,
%we write $T^{0,1}_X[1]$ to denote the dg manifold having
%$C^\infty (T^{0,1}_X[1])=\Omega^{0,\bullet}(X)$ as its algebra of functions
%and the Dolbeault operator $Q =\bar{\partial}$ as its homological vector field.
%
(3) \textsl{Derived intersections}
Given a smooth section $s$ of a vector bundle $E\to M$,
we write $E[-1]$ to denote the dg manifold having
$C^\infty (E[-1])=\sections{\Lambda^{-\bullet}E\dual}$
as algebra of functions and $Q=\iota_s$, the interior product with $s$, as
homological vector field. This dg manifold can be thought of
as a smooth model for the (possibly singular) intersection of $s$
with the zero section of the vector bundle $E$,
and is often called a `derived intersection', or
a \emph{quasi-smooth derived manifold} \cite{MR4735657}.

The notion of dg manifolds provides a highly effective framework
for studying the differential geometry of manifolds with singularities.
Dg manifolds of amplitude $[-n,-1]$ are equivalent to \emph{Lie $n$-algebroids}
\cite{MR2768006,MR4016029,MR2223155,MR3893501,MR3090103,MR4007376,arXiv:2001.01101}
and~\cite[Letters~7 and~8]{arXiv:1707.00265},
which can be regarded as the infinitesimal counterparts of higher groupoids
(see~\cite{MR2521116,MR4172669,MR2441255,arXiv:1705.02880,
MR2970717,arXiv:math/0612349,MR3774337}).
Furthermore, dg manifolds of amplitude $[1,n]$ correspond
to derived manifolds \cite{MR4735657,arXiv:2307.08179}.
As a result, general dg manifolds with amplitude $[-m,n]$ have the capacity to encode
both stacky and derived singularities in differential geometry. 
Notably, dg manifolds also serve as a powerful tool
for analyzing singular foliations \cite{MR4164730}.

However, despite its significance, the formal geometry of dg manifolds remains underexplored.
This paper aims to address this gap in terms of 
Fedosov dg manifolds.
Our main motivation is to establish
a Duflo--Kontsevich-type theorem for
the calculi of dg manifolds \cite{paper-two}.
It is well-known that formal geometry
\`a la Fedosov plays a crucial role
in the globalization of local Kontsevich
and Tsygan formality results \cite{MR2102846,MR2199629,MR2304327,MR1914788}.
Following Dolgushev \cite{MR2102846}, we construct a Fedosov dg manifold
$\cN_Q=(\cN,\fedosov+\tauTpolyF(Q))$ for any dg manifold $(\cM,Q)$ such that $\cN_Q$
is quasi-isomorphic to $(\cM,Q)$ as a dg manifold. Our construction builds
upon the Fedosov dg manifolds of graded manifolds introduced by two of the authors
in~\cite{MR3910470}, with an additional twisting induced by $Q$.
The underlying $\ZZ$-graded manifold of $\cN_Q$ is $\cN:=T_\cM[1]\oplus T_\cM$,
and the homological vector field is $\fedosov+\tauTpolyF(Q)$,
where $\fedosov$ is the Fedosov differential corresponding to an affine connection $\nabla$ on $\cM$,
constructed via the formal exponential map, and $\tau:\XX(\cM)\to\XX(\cN)$ is a Lie algebra morphism.
Furthermore, we construct a dg foliation on $\cN_Q$,
referred to as the \emph{Fedosov dg Lie algebroid}
$\cFQ\to\cNQ$. As a vector bundle, denoted $\cF\to\cN$, it is the pullback of the vector bundle
$T_\cM\to\cM$ via the surjective submersion $\cN\onto\cM$.
We frequently represent the dg Lie algebroid $\cFQ\to\cNQ$
as $\big(\cF\to\cN,\iL_{\fedosov+\tauTpolyF(Q)}\big)$.

We expect that the Fedosov dg Lie algebroid $\cFQ\to\cNQ$ is `homotopy equivalent'
to the tangent dg Lie algebroid $T_\cM\to\cM$ in the sense that
their `associated algebraic structures' are homotopy equivalent. 
More precisely, by `associated algebraic structures,' we mean Tamarkin--Tsygan calculi.
A calculus, as defined by Tamarkin and Tsygan \cite{MR1783778, MR2986860},
refers to a pair of $\ZZ$-graded vector spaces $(\aTheta^{\bullet},\Xi_{\bullet})$
such that: $\aTheta^{\bullet}$ is a Gerstenhaber algebra,
$\Xi_{\bullet}$ is simultaneously an associative algebra module (with action $\iI$)
and a Lie algebra module (with action $\iL$) over $\aTheta^{\bullet}$,
and is equipped with a degree~$-1$ operator $d$ satisfying $d^2=0$
and a set of identities similar to those of the Cartan calculus.

For a given dg manifold $(\cM,Q)$, there are two naturally associated calculi:
the Cartan calculus $\calculus_C(\cM,Q)$ and the noncommutative calculus $\calculus_H(\cM,Q)$.
Extending Cartan calculus to `noncommutative manifolds'
has been a central problem in noncommutative geometry,
explored by Connes \cite{MR1303779,MR823176} and Gelfand--Daletski--Tsygan \cite{MR1039918}.
More precisely, $\calculus_C(\cM,Q)$ consists of the pair: \[ \big(\HTT,\HOM\big) .\]
Here, $\totTpolyM{\bullet}$ and $\totApolyM{\bullet}$ denote the spaces of polyvector fields
and differential forms on $\cM$, respectively, which are naturally cochain complexes
with the differentials $\cQ$ induced by the homological
vector field $Q$ on $\cM$. The space of polyvector fields
$\totTpolyM{\bullet}$ is naturally a Gerstenhaber algebra, with the wedge product $\wedge$
and the Schouten bracket $\schouten{\argument}{\argument}$.
The space of differential forms $\totApolyM{\bullet}$ admits a differential
and module structures over $\totTpolyM{\bullet}$, where the operators $d$, $\iI$, and $\iL$
correspond to the de Rham differential $d_{\DR}$ \eqref{DR0}, the interior product \eqref{eq:IX0},
and the Lie derivative \eqref{eq:LX10}, respectively.

On the other hand, the noncommutative calculus $\calculus_H(\cM,Q)$ consists of the pair:
\[ \big(\HDD, \quad \HCC\big) .\]
Here, $\totDpolyM{\bullet}$ and $\totCpolyM{\bullet}$ 
denote the spaces of polydifferential
operators and polyjets on $\cM$, respectively. These spaces have induced differentials $\cQ$,
and the Hochschild (co)homology differentials $\hochschild$ and $\hochschildb$.
The space $\totDpolyM{\bullet}$ admits a Gerstenhaber bracket
$\gerstenhaber{\argument}{\argument}$ and a cup product $\cup$, which, up to homotopy,
form a Gerstenhaber algebra \cite{MR161898}.
The differential $d$ on the space of polyjets $\totCpolyM{\bullet}$
is the \RinehartConnes operator $B$ \eqref{eq:Connes1}. The action operators $\iI$ and $\iL$
of $\totDpolyM{\bullet}$ on $\totCpolyM{\bullet}$ are explicitly defined
in Equations~\eqref{eq:iI} and~\eqref{eq:iL}. 
Indeed, $\calculus_H(\cM,Q)$ is the noncommutative calculus
of smooth Hochschild (co)homologies \cite{MR1783778,MR2986860}:
\[ \big(H_{\mathrm{smooth}}^{\bullet}(A, A), H_{-\bullet}^{\mathrm{smooth}}(A, A)\big) ,\]
where $A$ is the dg algebra $(C^\infty(\cM),Q)$.

The calculi \(\calculus_C(\cM, Q)\) and \(\calculus_H(\cM, Q)\) can be explicitly expressed
in terms of the dg Lie algebroid \(T_\cM \to \cM\). In fact, for any dg Lie algebroid,
one can associate calculi of these two types. 
In particular, for a Fedosov dg Lie algebroid \(\cFQ \to \cNQ\)
associated with a dg manifold \((\cM, Q)\), we have two calculi:
\(\calculus_C(\cF,\fedosov+\tauTpolyF(Q))\) and \(\calculus_H(\cF,\fedosov +\tauTpolyF(Q))\),
where \(\calculus_C(\cF, \fedosov + \tauTpolyF(Q))\) denotes the pair:
\[ \Big(\cohomology{} \big(\totTpolyF{\bullet}, \schouten{\fedosov + \tauTpolyF(Q)}{\argument} \big), 
\quad \cohomology{} \big(\totcApolyF{\bullet}, \iL_{\fedosov + \tauTpolyF(Q)}\big)\Big) ,\]
and \(\calculus_H(\cF, \fedosov + \tauTpolyF(Q))\) denotes the pair:

\[ \Big(\cohomology{} \big(\totDpolyF{\bullet},
\gerstenhaber{\fedosov + \tauTpolyF(Q)}{\argument} + \hochschild \big),
\quad\cohomology{} \big(\totcCpolyF{\bullet},
\iL_{\fedosov + \tauTpolyF(Q)} + \hochschildb \big)\Big) .\]
Here, \(\totTpolyF{\bullet}\), \(\totcApolyF{\bullet}\), \(\totDpolyF{\bullet}\),
and \(\totcCpolyF{\bullet}\) denote the spaces of polyvector fields, differential forms,
polydifferential operators, and polyjets of the Fedosov dg Lie algebroid
\(\cFQ\to\cNQ\), respectively.

The main result of this paper is as follows:

\begin{introthm}\label{introthm:mainT}
Given a dg manifold \((\cM, Q)\) and a torsion-free affine connection \(\nabla\) on \(\cM\),
let \(\cFQ \to \cNQ := \big(\cF \to \cN, \iL_{\fedosov + \tauTpolyF(Q)}\big)\)
be its corresponding Fedosov dg Lie algebroid. Then: 
\begin{itemize}
\item There exists a pair of contractions
\[ \begin{tikzcd}[cramped]
\Big(\totTpolyM{\bullet}, \schouten{Q}{\argument}\Big)
\arrow[r, "\tauTpolyQ", shift left] &
\Big(\totTpolyF{\bullet}, \schouten{\fedosov + \tauTpolyF(Q)}{\argument} \Big)
\arrow[l, "\sigmaTpolyQ", shift left]
\arrow[loop, "\hTpolyQ", out=5, in=-5, looseness=3]
\end{tikzcd} \]
and
\[ \begin{tikzcd}[cramped]
\Big(\totApolyM{\bullet}, \iL_{Q}\Big)
\arrow[r, "\tauTpolyQ", shift left] &
\Big(\totcApolyF{\bullet}, \iL_{\fedosov + \tauTpolyF(Q)}\Big)
\arrow[l, "\sigmaTpolyQ", shift left]
\arrow[loop, "\hTpolyQ", out=5, in=-5, looseness=3]
\end{tikzcd} \]
such that the pair of injections \(\tauTpolyQ\) respects the operations
\(\wedge\), \(\schouten{\argument}{\argument}\), \(\iI\), \(\iL\), and \(d\).
\item There exists a pair of contractions
\[ \begin{tikzcd}[cramped]
\Big(\totDpolyM{\bullet}, \gerstenhaber{Q}{\argument} + \hochschild\Big)
\arrow[r, "\tauDpolyQ", shift left] &
\Big(\totDpolyF{\bullet}, \gerstenhaber{\fedosov + \tauTpolyF(Q)}{\argument} + \hochschild \Big)
\arrow[l, "\sigmaDpolyQ", shift left]
\arrow[loop, "\hDpolyQ", out=5, in=-5, looseness=3]
\end{tikzcd} \]
and
\[ \begin{tikzcd}[cramped]
\Big(\totCpolyM{\bullet}, \schoutenc{Q} + \hochschildb\Big)
\arrow[r, "\tauDpolyQ", shift left] &
\Big(\totcCpolyF{\bullet}, \schoutenc{\fedosov + \tauTpolyF(Q)} + \hochschildb\Big)
\arrow[l, "\sigmaDpolyQ", shift left]
\arrow[loop, "\hDpolyQ", out=5, in=-5, looseness=3]
\end{tikzcd} \]
such that the pair of injections \(\tauDpolyQ\) respects the operations
\(\cup\), \(\gerstenhaber{\argument}{\argument}\), \(\iI\), \(\iL\), and \(B\).
\end{itemize}
\end{introthm}

As an immediate consequence, we prove the following:
\begin{introthm}
\label{introthm:mainT1}
Under the same hypothesis as in Theorem \ref{thm:mainT}, we have:
\begin{itemize}
\item
The pair of maps \(\tauTpolyQ\) (and hence \(\sigmaTpolyQ\)) below:
\[ \begin{tikzcd}[cramped]
\cohomology{} \Big(\totTpolyM{\bullet}, \schouten{Q}{\argument}\Big)
\arrow[r, "\tauTpolyQ", shift left] &
\cohomology{} \Big(\totTpolyF{\bullet}, \schouten{\fedosov + \tauTpolyF(Q)}{\argument} \Big)
\arrow[l, "\sigmaTpolyQ", shift left]
\end{tikzcd} \]
and
\[ \begin{tikzcd}[cramped]
\cohomology{} \Big(\totApolyM{\bullet}, \iL_{Q}\Big)
\arrow[r, "\tauTpolyQ", shift left] &
\cohomology{} \Big(\totcApolyF{\bullet}, \iL_{\fedosov + \tauTpolyF(Q)}\Big)
\arrow[l, "\sigmaTpolyQ", shift left]
\end{tikzcd} \]
is an isomorphism of Tamarkin--Tsygan calculi between \(\calculus_C(\cM, Q)\)
and \(\calculus_C(\cF, \fedosov + \tauTpolyF(Q))\).
\item
The pair of maps \(\tauDpolyQ\) (and hence \(\sigmaDpolyQ\)) below:
\[ \begin{tikzcd}[cramped]
\cohomology{} \Big(\totDpolyM{\bullet},
\gerstenhaber{Q}{\argument} + \hochschild \Big)
\arrow[r, "\tauDpolyQ", shift left] &
\cohomology{} \Big(\totDpolyF{\bullet},
\gerstenhaber{\fedosov + \tauTpolyF(Q)}{\argument} + \hochschild \Big)
\arrow[l, "\sigmaDpolyQ", shift left]
\end{tikzcd} \]
and
\[ \begin{tikzcd}[cramped]
\cohomology{} \Big(\totCpolyM{\bullet}, \iL_{Q} + \hochschildb \Big)
\arrow[r, "\tauDpolyQ", shift left] &
\cohomology{} \Big(\totcCpolyF{\bullet}, \iL_{\fedosov + \tauTpolyF(Q)} + \hochschildb \Big)
\arrow[l, "\sigmaDpolyQ", shift left]
\end{tikzcd} \]
is an isomorphism of Tamarkin--Tsygan calculi between \(\calculus_H(\cM, Q)\)
and \(\calculus_H(\cF, \fedosov + \tauTpolyF(Q))\).
\end{itemize}
\end{introthm}

One of the main applications of Theorem~\ref{introthm:mainT}
is the study of the relationship between the calculi $\calculus_C(\cM,Q)$
and $\calculus_H(\cM,Q)$, which leads to the establishment of a Duflo--Kontsevich-type theorem --- see~\cite{paper-two}. Furthermore, we believe that Theorem~\ref{introthm:mainT}
is of independent interest, for instance in the study of QFT
along the line of \cite{MR3586504,MR2827826},
where the formal geometry of the target space often plays a crucial role in constructing the action.

Our approach is first to establish Fedosov contractions for graded manifolds, i.e. to prove Theorem \ref{introthm:mainT} for the case that $Q=0$. Then we use homotopy perturbation to treat the general case of dg manifolds. Note that in the ordinary approach to the globalization of Kontsevich's formality theorem and Tsygan's formality theorem for $\RR^d$ to arbitrary manifolds, Dolgushev obtained Fedosov-type quasi-isomorphisms in the setting of ordinary smooth manifolds \cite{MR2102846, MR2199629}. In the graded manifolds case, Cattaneo-Felder obtained Fedosov-type quasi-isomorphisms for polyvector fields and polydifferential operators \cite{MR2304327}. Here, rather than quasi-isomorphisms, we construct contractions, which are not only stronger as results, but also crucial to upgrading the Fedosov quasi-isomorphisms to dg manifolds, i.e., graded manifolds \emph{endowed with a homological vector field}.

%\subsection*{Acknowledgments}
%\addcontentsline{toc}{section}{Acknowledgments}

\subsection*{Notations and conventions}

In this paper, ``graded'' means $\ZZ$-graded unless otherwise stated. 

Given a bigraded $R$-module $C=\bigoplus_{p,q\in\ZZ}C^{p,q}$, we set 
\[ \tot_\oplus^n(C) := \bigoplus_{\substack{p+q=n \\ p,q \in \ZZ}} C^{p,q}
\qquad \text{and} \qquad \tot_\Pi^n(C) := \prod_{\substack{p+q=n \\ p,q \in \ZZ}} C^{p,q} ,\]
which are called the \emph{direct-sum total space}
and the \emph{direct-product total space}, respectively. 
The notations $\tot_\oplus$ and $\tot_\Pi$ are inspired
by the two types of total complexes of a double complex.
Explicitly, in this paper, a double complex $(C,d_1,d_2)$
is a $\ZZ$-bigraded $R$-module $C=C^{\bullet,\bullet}$
together with operators $d_1:C^{p,q}\to C^{p+1,q}$
and $d_2:C^{p,q}\to C^{p,q+1}$ such that both of the operators
square to zero and they anti-commute. Given a double complex $(C,d_1,d_2)$,
one has two types of total complexes:
the direct-sum total complex $(\tot_\oplus^\bullet(C),d_1+d_2)$
and the direct-product total complex $(\tot_\Pi^\bullet(C),d_1+d_2)$.

Let $V$ be a graded $R$-module and $i,j\in\ZZ$.
Let $V[i]$ be the graded $R$-module whose homogeneous components are 
$ (V[i])^j = V^{i+j} .$
We denote by $\nshift: V \to V[1] $ the degree-shifting map of degree $-1$, and denote by $ \pshift: V \to V[-1] $ the degree-shifting map of degree $+1$. Thus, we have $V[1] = \nshift V$ and $V[-1] = \pshift V$. 
% \[ V[1] = \nshift V \qquad \text{and} \qquad V[-1] = \pshift V .\]
The $R$-action on $V[i]$ is given by $r \cdot \nshift^i(v) = (-1)^{i |r|} \nshift^i (r \cdot v)$ for $r \in R$ and $v \in V$, where $\nshift^i:V \to V[i]$ is the degree-shifting map.

Throughout the paper, we also use the following notations: 
$\NN=\{1,2,3,\cdots\}$ and $\NO=\{0\}\cup\NN$.

\section{Tamarkin--Tsygan calculi of dg manifolds}

\subsection{Tamarkin--Tsygan calculi}
\label{set:1.1}

Recall that a Gerstenhaber algebra is a graded vector space
$\aTheta^\bullet=\oplus_{k\in\ZZ}\aTheta^k$ together with
\begin{enumerate}
\item a graded commutative associative algebra structure on $\aTheta^\bullet$
(with multiplication $\argument\cdot\argument:\aTheta^p\otimes\aTheta^q\to\aTheta^{p+q}$);
\item a graded Lie algebra structure on $(\aTheta[1])^\bullet$
(with bracket $\lie{\argument}{\argument}:\aTheta^p\otimes\aTheta^q\to\aTheta^{p+q-1}$)
\end{enumerate}
such that
the multiplication and bracket are compatible in the sense that,
for all homogeneous $X\in\aTheta$, $Y\in\aTheta$, and $Z\in\aTheta$, we have
\[ \lie{X}{Y\cdot Z}=\lie{X}{Y}\cdot Z+(-1)^{(\degree{X}-1)\degree{Y}} Y\cdot\lie{X}{Z} .\]
%\purple{\tiny
%\[ \pshift \lie{\nshift X}{\nshift(Y\cdot Z)}
%=\pshift \lie{\nshift X}{\nshift Y}\cdot Z
%+(-1)^{(\degree{X}-1)\degree{Y}} Y\cdot
%\pshift \lie{\nshift X}{\nshift Z} .\]
%}
%
%{\tiny
%In particular, we note that,
%for all $X\in\Theta^x$, $Y\in\Theta^y$, and $Z\in\Theta$, we have
%\begin{gather*}
%(X\wedge Y)\wedge Z=X\wedge(Y\wedge Z) \\
%Y\wedge X=(-1)^{xy}X\wedge Y \\
%\lie{Y}{X}=-(-1)^{(x-1)(y-1)}\lie{X}{Y} \\
%\lie{X}{\lie{Y}{Z}}=\lie{\lie{X}{Y}}{Z}+(-1)^{(x-1)(y-1)}\lie{Y}{\lie{X}{Z}}
%\end{gather*}
%}
The following definition is due to Tamarkin--Tsygan~\cite{MR1783778,MR2986860}.

\begin{definition}[{\cite[Definitions~3.6.3 and~3.6.4]{MR2986860}}]
\label{def:calc}
A \emph{calculus} is a pair of graded spaces $\aTheta^{\bullet}$ and $\Xi_{\bullet}$ such that
\begin{itemize}
\item[(i)] $\aTheta^{\bullet}$ carries a \emph{structure of Gerstenhaber algebra};
\item[(ii)] $\Xi_{\bullet}$ carries a \emph{structure of graded module
over the graded commutative algebra} $\aTheta^{\bullet}$
(with corresponding action denoted by
$\iI_X:\Xi_{\bullet} \to \Xi_{\bullet+|X|}, \forall X\in\aTheta^{\bullet}$);
\item[(iii)] $\Xi_{\bullet}$ carries a \emph{ structure of graded module
over the graded Lie algebra} $(\aTheta[1])^\bullet$
(with corresponding action denoted by
$\iL_X:\Xi_{\bullet} \to \Xi_{\bullet+|X|-1} , \forall X\in\aTheta^{\bullet}$);
\item[(iv)] there exists an operator $d$ of degree~$-1$ on $\Xi_{\bullet}$
satisfying $d^2=0$;
\item[(v)] the following identities hold
\begin{gather}
\iL_X=\commutator{\iI_X}{d} = (-1)^{\degree{X}-1} \commutator{d}{\iI_X}, \label{brick} \\
\iI_{\lie{X}{Y}}=\commutator{\iL_X}{ \iI_Y }. \label{straw}
\end{gather}
\end{itemize}
\end{definition}

Equation~\eqref{brick} is called the Cartan formula.
It is simple to see that (iii) is in fact redundant,
which follows automatically from (iv) and (v).		
Furthermore, the following identity holds automatically
\begin{equation}
\iL_{X\cdot Y}= (-1)^\degree{Y}
\iL_X\circ \iI_Y+ \iI_X\circ\iL_Y. \label{wood}
\end{equation}

Alternatively, Equation~\eqref{straw} can be written as
\begin{equation}
\iI_{\lie{X}{Y}}=\commutator{\commutator{\iI_X}{d}}{\iI_Y},
\end{equation}
which means that the Lie bracket on $(\aTheta[1])^{\bullet}$ can be
realized as a derived bracket \cite{MR2163405, MR1427124, MR2104437}.

%\begin{liao}
%Change $\aTheta^{\bullet}[1]$ to $(\aTheta[1])^{\bullet}$.
%\end{liao}

The following are two basic examples, which are of fundamental importance in this paper.

\begin{example}\label{ex:1}
%For a given smooth manifold $M$,
Let $M$ be a smooth manifold.
Let $\aTheta^{\bullet}:=T_{\poly}^{\bullet}(M)=\Gamma(\Lambda^\bullet T_M)$
be its space of polyvector
fields, and $\Xi_{\bullet}:=\Omega^{-\bullet}(M)$
its space of differential forms.
The pair $\big(\aTheta^{\bullet}, \Xi_{\bullet})$ is a
Tamarkin--Tsygan calculus, which will be denoted by $\calculus_C(M)$:
the operators $\iI_X$ and $\iL_X$ are simply the usual interior product
and Lie derivative with respect to the polyvector field $X$
while the operator $d$ is the de~Rham differential. 

Note that a vector field $X$ on $M$ is an element of degree $(+1)$
in $\aTheta^{\bullet}=T_{\poly}^{\bullet}(M)$.
Thus, when $X$ is a vector field rather than a polyvector field,
Equation~\eqref{brick} reduces to the classical Cartan formula:
\[ \iL_X= \commutator{\iI_X}{d}= \commutator{d}{\iI_X}=\iI_X \circ d+d\circ \iI_X .\] 
Assigning degree $(+1)$ to vector fields has the advantage of making the associative
multiplication in $T_{\poly}^{\bullet}(M)$ a degree preserving map. 
\end{example}

\begin{example}\label{ex:2}
For any associative algebra $A$, we denote by $H^{\bullet}(A,A)$
and $H_{\bullet}(A,A)$ its Hochschild cohomology
and Hochschild homology, respectively. 
It is a classical result 
that there exists a calculus structure --- expressible by explicit formulas --- on the pair $\big(H^{\bullet}(A,A),H_{-\bullet}(A,A)\big)$ \cite{MR161898,
MR154906, MR1783778, MR1039918, MR823176, MR2986860}.

\end{example}

When $A$ is the algebra $C^\infty(M)$ of smooth functions
on a manifold $M$, instead of purely algebraic (co)\-ho\-mol\-o\-gies,
one often needs to consider certain smooth Hochschild (co)homologies.
The calculus obtained in this way will be denoted by $\calculus_H(M)$.
See \cite{MR2986860} for details.

\begin{definition}
\label{def:calc-mor}
Let $\big(\aTheta^{\bullet}_i, \Xi_{\bullet}^i\big), \quad i=1, 2$, be calculi. 
A \emph{morphism} from $\big(\aTheta^{\bullet}_1, \Xi_{\bullet}^1\big)$ to
$\big(\aTheta^{\bullet}_2, \Xi_{\bullet}^2\big)$
consists of a pair of degree $0$ morphisms, denoted by
the same symbol by abuse of notation,
$$
\tau: \aTheta^{\bullet}_1 \to \aTheta^{\bullet}_2 \qquad \text{and} \qquad \tau: \Xi_{\bullet}^1 \to \Xi_{\bullet}^2
$$
% \begin{eqnarray}
% &\tau:& \aTheta^{\bullet}_1 \to \aTheta^{\bullet}_2 ;\label{eq:tau1}\\
% &\tau:& \Xi_{\bullet}^1 \to \Xi_{\bullet}^2 \label{eq:tau2}
% \end{eqnarray}
such that
\begin{itemize}
\item[(i)] the first map $\tau: \aTheta^{\bullet}_1 \to \aTheta^{\bullet}_2$ is a morphism
of Gerstenhaber algebras;
\item[(ii)] the second map $ \tau: \Xi_{\bullet}^1 \to \Xi_{\bullet}^2$ satisfies
the condition $d\circ\tau=\tau\circ d$;
% \begin{equation}
% d\circ\tau=\tau\circ d
% \end{equation}
\item[(iii)] the maps preserves the module structures: for any $ X\in\aTheta_1^{\bullet} $ and $\xi\in\Xi_{\bullet}^1$, we have $\tau(\iI_X \xi) = \iI_{\tau(X)} \tau (\xi)$.
% \begin{align}
% \tau(\iI_X \xi) &= \iI_{\tau(X)} \tau (\xi) \label{eq:tau3} \\
% \tau(\iL_X \xi) &= \iL_{\tau(X)} \tau (\xi) \label{eq:tau4}
% \end{align}
\end{itemize}
Here we used the same notations for the structure maps of both
calculi $\big(\aTheta^{\bullet}_i, \Xi_{\bullet}^i\big), \quad i=1, 2$.
\end{definition}

Although it is not required by Definition~\ref{def:calc-mor}, the Cartan formula implies that a morphism of calculi also preserves the Lie module structures: $\tau(\iL_X \xi) = \iL_{\tau(X)} \tau (\xi)$.

% \eqref{eq:tau4} is redundant.
% It follows from \eqref{eq:tau3} and the Cartan formula.

\subsection{Cartan calculi of dg manifolds}\label{corniche}

There are several definitions of graded manifolds in the literature. 
For us, a \emph{($\ZZ$-)graded manifold} $\cM$ over $\KK$ consists of a smooth manifold $M$
%(called support)
together with a sheaf $\fA_{\cM}$ of $\ZZ$-graded commutative $\cO_M$-algebras over $M$
such that there exists a covering $\{U_i\}_{i\in I}$ of $M$
by open submanifolds $U_i\subset M$ and a family of isomorphisms
$\fA_{\cM}(U_i)\cong \prod_{k=0}^\infty \cO(U_i) \otimes_\KK S^k(V\dual)$,
where $V$ is a fixed finite-dimensional $\ZZ$-graded vector space over the field $\KK$,
$V\dual$ denotes the $\KK$-dual of $V$.
Essentially $\fA_{\cM}(U_i)$ is the $\KK$-algebra of formal series
on $V$ with coefficients in $\cO(U_i)$.
The manifold $M$ is called the support of the graded manifold $\cM$.
In the sequel, we denote the graded $\KK$-algebra $\fA_{\cM}(M)$
of global sections of the sheaf $\fA_{\cM}$ either by $C^\infty(\cM)$ or by $\cR$.

In particular, any graded vector bundle $E = \bigoplus_{i \in \ZZ} E^i$
over a smooth manifold $M$ determines a graded manifold $\cM$:
%Its sheaf $\fA_{E}$ of topological commutative graded $\cO_M$-algebras over $M$:
for any open subset $U$ of $M$, 
\[ \fA_{\cM}(U) =\prod_{k=0}^\infty\Gamma(U, S^k E\dual)
=\varprojlim_p \frac{\Gamma(U, S E\dual)}{I^p} \] 
with $I=\Gamma(U,S^{\geq 1}E\dual)$. 
Here the algebra $\Gamma(U,\hat{S}E\dual)$ is equipped with the \emph{$I$-adic topology}. 
%
%A \emph{$\ZZ$-graded manifold} $\cM$ over $\KK$ ($\KK = \RR$ or $\CC$) consists of a smooth manifold $M$ (called {\em support}) together with a sheaf $\fA_{\cM}$ of topological commutative graded $\cO_M$-algebras over $M$ such that there exist a $\ZZ$-graded vector bundle $E \to M$ and an isomorphism of sheaves of topological graded $\cO_M$-algebras: $\fA_{\cM} \cong \fA_{E}\otimes_\RR \KK$. 
%Here an isomorphism $\phi^\ast:\fA_{\cM} \xto\cong \fA_{E}$ is an isomorphism of sheaves of topological $\cO_M$-algebras that preserves the homogeneous elements: $\phi^\ast(\fA_{\cM}^n ) = \fA_{E}^n$. 
See \cite{MR2709144,MR2275685,MR2819233,MR4514383}. 

Let $\cM$ be a graded manifold with support $M$.
Let $(U;x_1,\cdots,x_m)$ be a local chat of $M$. 
We can think of the coordinate functions $x_1,\cdots,x_m$
as elements of the $\KK$-algebra $\cO_M(U)$. 
Choose a homogeneous local frame of $E$ which induces
homogeneous coordinate functions $(x_{m+1},\cdots,x_{m+r})$ a fiber of $E$.
Together we obtain a local chat $(x_1,\cdots,x_m; x_{m+1},\cdots,x_{m+r})$ of $\cM$.
The $\KK$-algebra $\fA(U)$ is generated by smooth functions
of $(x_1,\cdots,x_m)$ and formal power series on $(x_{m+1},\cdots,x_{m+r})$.

%Choose a basis of $V$ which induces
%linear homogeneous coordinate functions $(x_{m+1},\cdots,x_{m+r})$
%on the graded vector space $V$. Together we obtain a
%local chat $(x_1,\cdots,x_m; x_{m+1},\cdots,x_{m+r})$ of $\cM$.
%The $\KK$-algebra $\fA(U)$ is generated by smooth functions
%of $(x_1,\cdots,x_m)$ and polynomial functions on $(x_{m+1},\cdots,x_{m+r})$.

\begin{definition}
\label{def:dg}
A \emph{dg manifold} is a graded manifold $\cM$ endowed
with a homological vector field, i.e.\ a derivation $Q$ of degree~$+1$
on $C^\infty(\cM)$ satisfying $\commutator{Q}{Q}=0$.
\end{definition}

Let $(\cM,Q)$ be a dg manifold over $\KK$. 
For all $p\geq 0$, let $\Tpolym{p}$ denote the space
$\sections{S^p(T_\cM[-1])}$, which is isomorphic to $\sections{\Lambda^{p}T_{\cM}}[-p]$,
of $p$-vector fields on $\cM$. In particular, we have
$\Tpolym{0}\cong C^\infty(\cM)=\cR$
and $\Tpolym{1}\cong \pshift \sections{T_{\cM}}=\pshift \Der(\cR) $.
We use the symbol $\TpolyM{0}{n}$ to denote the
space of smooth functions of degree~$n$ on $\cM$,
and the symbol $\TpolyM{p}{n}$ to denote the degree~$n$ subspace
of $\Tpolym{p}$.\footnote{Note that the notations $\TpolyM{p}{n}$
and $\DpolyM{p}{n}$ in this paper have slightly different meaning
from those in~\cite{MR3754617}. Essentially, there is a degree shift
between the conventions used in these two papers.}
%In other words, an element in $\TpolyM{p}{q}$ is a finite sum
%$\sum X_1 \wedge \cdots \wedge X_p$,
%where $X_1,\cdots,X_p\in\sections{T_\cM}$ are homogeneous vector fields on $\cM$
%with $\degree{X_1}+\cdots+\degree{X_p}=q$.
%The bigraded left $\cR$-module
%\[ \TpolyM{\bullet}{\diamond} = \bigoplus_{\substack{p,n\in\ZZ \\ p\geq 0}}\TpolyM{p}{n} \]
%is called the \emph{space of polyvector fields on $\cM$}. 
The space $\Tpolym{}:= \bigoplus_{p =0}^\infty \Tpolym{p}$ 
% \[ \Tpolym{}:= \bigoplus_{p =0}^\infty \Tpolym{p} \]
is called the \emph{space of polyvector fields on $\cM$},
which is equipped with the bigrading: 
$ \big(\Tpolym{}\big)^{p,q} = \TpolyM{p}{p+q} .$ 
With this bigrading, we have the direct-sum total space 
\begin{equation}\label{eq:GaredeNord}
\totTpolyM{n} = \bigoplus_{p+q =n} \big(\Tpolym{}\big)^{p,q}
=\bigoplus_{p=0}^\infty \TpolyM{p}{n}
.\end{equation}

When endowed with the graded commutator $\schouten{\argument}{\argument}$,
the space $\TpolyM{1}{\bullet} = \big(\pshift \Der(\cR) \big)^\bullet$
of graded derivations of $\cR$ is a graded Lie algebra.
This Lie bracket can be extended to the space $\totTpolyM{n}$
of graded polyvector fields on $\cM$ in such a way that the triple $\big(\totTpolyM{\bullet},\schouten{\argument}{\argument},\wedge\big)$ 
% \[ \big(\totTpolyM{\bullet},\schouten{\argument}{\argument},\wedge\big) \]
becomes a Gerstenhaber algebra.
%\[ \schouten{\xi}{\eta_1\wedge\eta_2}
%=\schouten{\xi}{\eta_1}\wedge\eta_2
%+(-1)^{(\degree{\xi}-1)\degree{\eta_1}}
%\eta_1\wedge\schouten{\xi}{\eta_2} ,\]
%for $\xi\in\TpolyM{p_0}{q_0}$, $\eta_1\in\TpolyM{p_1}{q_1}$,
%$\eta_2\in\TpolyM{\bullet}{\bullet}$
%so that $\degree{\xi}=p_0+q_0$ and $\degree{\eta_1}=p_1+q_1$.
This extended bracket is called Schouten bracket.
The homological vector field $Q$ induces a degree~$+1$ differential 
$ \tQ: \TpolyM{p}{n} \to \TpolyM{p}{n+1} ,$
namely the Lie derivative $\tQ:
%=\iL_{Q} 
=\schouten{Q}{\argument}$ w.r.t.\ the homological vector field $Q$.
%\[ \tQ: \totTpolyM{\bullet}\to\totTpolyM{\bullet+1} ,\]
Thus we obtain a dg Gerstenhaber algebra
$ \big(\totTpolyM{\bullet}, \schouten{Q}{\argument} ,\schouten{\argument}{\argument},\wedge\big) . $ 
Its cohomology $\cohomology{\bullet}\big(\totTpolyM{},\tQ\big)$ is
a Gerstenhaber algebra.

Similarly, let $\Apolym{-p}=\sections{S^{p}(T^\vee_\cM [1])}
= \Omega^p (\cM)[p]$ be the space of differential $p$-forms on $\cM$ with negative degree-shifting. 
For example, $dx^i$ is of degree $-1 + |x^i|$ in $\Apolym{-1}$. 
We denote the degree~$n$ subspace of $\Apolym{-p}$ by $\ApolyM{-p}{n}$.
%The bigraded left $\cR$-module
%\[ \ApolyM{\bullet}{\diamond} = \bigoplus_{\substack{p,n\in\ZZ \\ p\geq 0}}\ApolyM{-p}{n} \]
%is called the \emph{space of differential forms on $\cM$}.
The space $\Apolym{}:= \bigoplus_{p =0}^\infty \Apolym{-p}$ is referred to as the \emph{space of differential forms on $\cM$},
which is equipped with the bigrading: $\big(\Apolym{}\big)^{-p,q}=\ApolyM{-p}{-p+q} .$ 
With this bigrading, we have the direct-product total space
\begin{equation}\label{eq:GaredeEst} 
\totApolyM{n}=\prod_{-p+q=n}\big(\Apolym{}\big)^{-p,q}
=\prod_{p = 0}^\infty \ApolyM{-p}{n} .
\end{equation}
Here we consider the direct-sum total space of $\Tpolym{}$
and the direct-product total space of $\Apolym{}$,
which are relevant to the Duflo--Kontsevich type theorems for dg manifolds \cite{paper-two}.
We will use an analogous setting for $\Dpolym{}$ and $\Cpolym{}$ in the next section. 

The homological vector field $Q$ induces a degree~$+1$ differential
$\aQ$ via the Lie derivative $\aQ=\liederivative{Q}: \ApolyM{-p}{n}\to \ApolyM{-p}{n+1} .$ 
In the sequel, for the simplicity of notations,
we use $\aQ$ to denote the Lie derivative $\liederivative{Q}$ 
w.r.t.\ the homological vector field $Q$,
whose meaning should be adapted to the context by the objects it refers to.

By $d_{\DR}$, we denote the de Rham differential
\begin{equation}
\label{DR0}
d_{\DR}: \Apolym{-p} \to \Apolym{-p-1},
\end{equation}
For any $X\in\Tpolym{d}$, denote by
\begin{equation}
\label{eq:IX0}
\iI_X: \Apolym{-p} \to \Apolym{-p+d}
\end{equation}
the natural contraction operator;
and
\begin{equation}
\label{eq:LX10}
\iL_{X}: \Apolym{-p} \to \Apolym{-p+d-1}, \qquad \iL_{X}=(-1)^{|X|-1} \commutator{d_{\DR}}{\iI_X},
\end{equation}
the Lie derivative defined by the Cartan formula. 
% \begin{equation}
% \label{eq:LX20}
% \iL_{X}=(-1)^{|X|-1} \commutator{d_{\DR}}{\iI_X}.
% \end{equation}
It is simple to see that all the operations $d_{\DR}$, $\iI$ and $\iL$ are compatible with the
dg structures on $\totTpolyM{\bullet}$
and $\totApolyM{\bullet}$, and therefore descend to
well-defined operations, denoted by the same
symbols, to their respective ${\aQ}$-(co)homologies
$\HTT$ and $\HOM$. 
Thus we have

\begin{proposition}
Let $(\cM,Q)$ be a dg manifold.
The pair $(\aTheta^{\bullet}, \Xi_{\bullet})$, where
\[ \aTheta^{\bullet}:=\HTT,
\quad\quad\text{and}\quad\quad
\Xi_{\bullet}=\HOM ,\]
admit a Tamarkin--Tsygan calculus,
whose operators $d$, $\iI_X$, and $\iL_X$ are respectively
the de Rham differential $d_{DR}$ \eqref{DR0},
the interior product w.r.t.\
the polyvector field $X$ \eqref{eq:IX0}, and
the Lie derivative w.r.t.\ $X$ \eqref{eq:LX10}.
\end{proposition}

Such a calculus is denoted $\calculus_C(\cM,Q)$, which is a dg extension
of $\calculus_C(M)$ in Section~\ref{set:1.1}.

\subsection{Noncommutative calculi of dg manifolds}\label{sec:HtpCalDGMfd}

Next, we discuss the dg analogue of $\calculus_H(M)$ in Section~\ref{set:1.1}.
For $\aTheta^{\bullet}$, we should take ``smooth
Hochschild cohomology'' of the dg manifold $(\cM,Q)$. 
More precisely, for Hochschild cochain complex, we
take polydifferential operators on $(C^\infty(\cM),Q)$,
which is a subcomplex of the full Hochschild cochain complex of the dg algebra $(C^\infty(\cM),Q)$. 
Also, since the dg algebra $(C^\infty(\cM),Q)$ are $\ZZ$-graded,
there are two choices of Hochschild cochain complexes with
respect to the grading: direct sum and direct product.
Throughout the paper, we will choose ``direct sum'' following
\cite{MR2986860,MR4584414}, since it is more relevant to the formality theorem
\cite{MR2062626,MR2699544}, on which we rely in order to prove 
our Duflo-Kontsevich type theorems \cite{MR3754617,paper-two}. 
%even though it is ill behaved from the category viewpoint.
Note that the direct sum Hochschild cohomology of
a differential graded algebra behaves significantly
differently from the ordinary Hochschild
cohomology, i.e. the direct product Hochschild cohomology \cite{MR2202177,keller2003derived},
and are not well behaved under the homotopy equivalence.

A linear differential operator on $\cM$ is a $\KK$-linear endomorphism
of $\cR=C^\infty(\cM)$ that can be written as a finite sum of compositions
of derivations (i.e.\ vector fields) and multiplications
by specified functions (i.e.\ differential operators of order zero).
% of compositions
%of graded derivations $X_1,\dots,X_k$ of $\cR$ with $\degree{X_1}+\cdots+\degree{X_k}=q$.
We use the symbol $\sD(\cM)$ to denote the space of linear differential
operators on $\cM$, and the symbol $\big(\sD(\cM)\big)^n$ to denote
the subspace consisting of linear differential operators of degree~$n$.
We also consider polydifferential operators on $\cM$.

Let $\pshift\sD(\cM):=\sD(\cM)[-1]$
be the suspended space of linear differential operators on $\cM$,
which is naturally an $\cR$-module.
It can be naturally embedded into $\Hom_\KK(\cR[1],\cR)$ as a graded vector subspace.
%\diffM{}[-1]\otimes_\cR\cdots\otimes_\cR\diffM{}[-1]$,
The space $\sD^p_{\poly}(\cM)$ of $p$-differential operators on $\cM$ 
is defined as follows:
$\sD^p_{\poly}(\cM)=\bigotimes_\cR^p\big(\pshift\sD(\cM)\big)$ 
for $p\geq 1$ and $\sD^0_{\poly}(\cM)=\cR$. 
By $\DpolyM{p}{n}$, we denote the space
of $p$-differential operators on $\cM$ of total degree~$n$.
%We use the symbol $\DpolyM{p}{q}$ to denote the space
% of $p$-differential operators of degree~$q$ on $\cM$,
Then $\DpolyM{0}{n}$ is
the space of smooth functions of degree~$n$ on $\cM$.

%The bigraded left $\cR$-module
%$\DpolyM{\bullet}{\bullet} = \bigoplus_{\substack{p,q\in\ZZ \\ p\geq 0}}\DpolyM{p}{q}$
%is called the \emph{space of polydifferential operators on $\cM$}.
%We are most interested in the graded left $\cR$-module $\totDpolyM{\bullet}$
%defined by
%The bigraded left $\cR$-module
%\[ \DpolyM{\bullet}{\bullet} = \bigoplus_{\substack{ p\geq 0}}\DpolyM{p}{q} \]
%is called the \emph{space of polydifferential operators on $\cM$}.
The left $\cR$-module $\Dpolym{}:= \bigoplus_{p =0}^\infty \Dpolym{p}$ is referred to as the \emph{space of polydifferential operators on $\cM$}, which is equipped with the bigrading: $\big(\Dpolym{}\big)^{p,q}=\DpolyM{p}{p+q}.$ 
% \[ \big(\Dpolym{}\big)^{p,q}=\DpolyM{p}{p+q} .\]
With this bigrading, we have the direct-sum total space 
\begin{equation}\label{eq:totDm}
\totDpolyM{n}=\bigoplus_{p+q=n}\big(\Dpolym{}\big)^{p,q}
=\bigoplus_{p=0}^\infty\DpolyM{p}{n}
.\end{equation}

%As in the classical case,
Endowing the space of polydifferential operators
% $\totDpolyM{\bullet}$
with the standard Gerstenhaber bracket
\begin{equation}\label{eq:Gbracketdg}
\gerstenhaber{\argument}{\argument}:
\DpolyM{p}{n} \times\DpolyM{p'}{n'}\to\DpolyM{p+p'-1}{n+n'-1} 
\end{equation}
and the Hochschild differential
\[ \hochschild:=\gerstenhaber{m}{\argument}:\DpolyM{p}{n}\to\DpolyM{p+1}{n+1} , \]
makes $\totDpolyM{\bullet}[1]$ into a dgla \cite{MR161898},
where $m\in\DpolyM{2}{2}$ is the shifted multiplication defined by 
%= - \nshift (\pshift 1 \otimes \pshift 1) \in \DpolyM{2}{0}[1]$
%is identified with the map $m: \cR[1] \otimes_\KK \cR[1] \to \cR[1]$,
\begin{equation}\label{eq:ShiftedMultiplication}
m(\nshift f,\nshift g)=(-1)^\degree{f} fg, \qquad\forall f,g\in\cR.
\end{equation}
Here $\nshift:\cR\to\cR[1]$ is the degree $(-1)$-shifting map,
and $\Dpolym{2} = \pshift \sD(\cM) \otimes_\cR \pshift \sD(\cM)$ is identified
with a subspace of $\Hom(\cR[1] \otimes_\KK \cR[1], \cR)$.
See, for example, \cite[Appendix~B]{MR4584414}.

The tensor product of left $\cR$-modules determines a cup product
\begin{equation}\label{eq:cupdg}
\cup:\DpolyM{p}{n} \times\DpolyM{p'}{n'}\to\DpolyM{p+p'}{n+n'}
.\end{equation}
The homological vector field $Q$ induces a degree~$(+1)$ differential
%$\dQ=\gerstenhaber{Q}{\argument}$
\[ \dQ=\gerstenhaber{Q}{\argument}:\quad \totDpolyM{\bullet}\to\totDpolyM{\bullet+1} .\]

Since $(\cM,Q)$ is a dg manifold,
it follows that $\gerstenhaber{Q+m}{Q+m}=0$.
Hence $\big(\dQ+\hochschild\big)^2=0$.
The Hochschild cohomology of the dg manifold $(\cM,Q)$,
denoted $\HDD$, is the cohomology of the direct-sum total cochain complex $\CDD$ of the double complex $(\Dpolym{},\hochschild, \dQ)$.

It is standard that the Gerstenhaber bracket \eqref{eq:Gbracketdg}
$\gerstenhaber{\argument}{\argument}$ and the cup product \eqref{eq:cupdg}
$\cup$ descend to the Hochschild cohomology. 
Endowed with the cup product and the Gerstenhaber bracket,
$\HDD$ becomes a Gerstenhaber algebra.
See \cite[Appendix]{MR2304327} and \cite[Section~2.2]{MR2986860}.
For more details, the reader might wish to consult
\cite{MR3754617,MR2275685,MR2986860, PolishSurvey}.

Denote by $\jm:=\Hom_\cR(\sD(\cM),\cR)$ the space of $\infty$-jets of $\cM$, which is naturally an $\cR$-module.
For every $p\geq 1$, let $\Cpolym{-p}:=\big(\nshift \jm \big)^{\cotimes_\cR \, p}
=\big(\jm [1] \big)^{\cotimes_\cR \, p}$,
called the space of \emph{$p$-polyjets} of $\cM$.
For $p=0$, let $\Cpolym{0}:=\cR$. Then, for every $p\geq 0$,
$\Cpolym{p}$ can be identified with $\cJ^\infty_\Delta\big( \cM^{\times (p+1)} \big)[p]$,
the space of $\infty$-jets along the diagonal $\Delta$ of the product
$\cM \times \cdots \times \cM$ of $(p+1)$-copies of the graded manifold $\cM$.
The isomorphism is induced by the map $\Phi: \big(C^\infty(\cM)
\otimes_\KK (C^\infty(\cM))^ {\otimes_\KK p}\big)[p] 
\to \Cpolym{-p}$ 
% \begin{equation}\label{eq:Phi}
% \Phi: \big(C^\infty(\cM)
% \otimes_\KK (C^\infty(\cM))^ {\otimes_\KK p}\big)[p], 
% \to \Cpolym{-p}
% \end{equation}
which is defined by
\begin{equation}
\label{eq:Phi1}
\duality{\Phi(a_0\otimes a_1\otimes\cdots\otimes a_p)}
{u_1\otimes\cdots\otimes u_p}
= \pm\, a_0 \cdot u_1(a_1) \cdots u_p(a_p),
\end{equation}
for $a_0, \cdots, a_p \in C^\infty(\cM)$.
See~\cite{MR2986860} for more details. 
By $\CpolyM{-p}{n}$, we denote the subspace
of $\Cpolym{-p}$ consisting of elements of total degree~$n$. 
Then $\CpolyM{0}{n}$ is the space of smooth functions of degree~$n$ on $\cM$.
%The bigraded left $\cR$-module
%\[ \CpolyM{\bullet}{\bullet} = \bigoplus_{\substack{p,q\in\ZZ \\ p\geq 0}}\CpolyM{p}{q} \]
%is called the \emph{space of poly-$\infty$-jets on $\cM$}.

The space $\Cpolym{}:= \bigoplus_{p =0}^\infty \Cpolym{-p}$ is called the \emph{space of polyjets on $\cM$} and is equipped with the bigrading: 
$\big(\Cpolym{}\big)^{-p,q}=\CpolyM{-p}{-p+q}.$ 
With this bigrading, we have the direct-product total space 
\begin{equation}\label{eq:totCm}
\totCpolyM{n}=\prod_{-p+q=n}\big(\Cpolym{}\big)^{-p,q}
=\prod_{p=0}^\infty\CpolyM{-p}{n}
.\end{equation}
Following \cite{MR2986860}, we introduce the Hochschild boundary differential $\hochschildb$,
\RinehartConnes operator $B$, and the pairings $\iI$ and $\iL$ as follows.
Here we implicitly use the identification induced by \eqref{eq:Phi1} and consider polyjets on $\cM$ as elements in
$\prod_{p\geq 0}\cJ^\infty_\Delta [p]$. 
The advantage of such considerations is that we can use the exact same
formulas for dg algebras as in \cite{MR2986860} without any modifications. 

%However, for Hochschild (co)homology of $\ZZ$-graded Lie algebroids $\cL$,
%the formulas below do not have an obvious extension.
%Such structures and explicit formulas depend on the ``formal groupoid"
%structure of the jet space $\jet{\cL}$ \cite{MR3456700,paper-zero}.
%
Let $a_0, \cdots, a_p \in C^\infty(\cM)$. Set $B:\CpolyM{-p}{n}\to\CpolyM{-p-1}{n-1}$ 
% \begin{equation}\label{eq:Connes}
% B:\CpolyM{-p}{n}\to\CpolyM{-p-1}{n-1}
% \end{equation}
to be the \RinehartConnes operator, which is induced by the formula 
%\begin{mathieu}
%I have the correct formulae
%for the calculus operations in the $\ZZ$-graded case
%in my handwritten notes from Stockholm.
%\end{mathieu}
\begin{multline}\label{eq:Connes1}
B(a_0\otimes a_1\otimes\cdots\otimes a_p)
=\sum_{i=0}^p \pm\,
\big(1\otimes a_i\otimes a_{i+1}\otimes\cdots\otimes a_p\otimes
a_0\otimes\cdots\otimes a_{i-1} \\
\pm\, a_i\otimes 1\otimes a_{i+1}\otimes\cdots\otimes a_p\otimes
a_0\otimes\cdots\otimes a_{i-1}\big).
\end{multline}
%
%\begin{equation}
%\label{eq:Connes1}
%\begin{split}
%B(a_0 \otimes a_1 & \otimes\cdots\otimes a_p) = 
%\sum_{i=0}^p \epsilon_i \cdot 1 \otimes a_i \otimes \cdots
%\otimes a_p \otimes a_0 \cdots \otimes a_{i-1} \\
%& - (-1)^{| a_0| + \cdots + | a_p| -(p+1)} \sum_{i=0}^p \epsilon_i \cdot a_i \otimes a_{i+1}
%\otimes \cdots \otimes a_p \otimes a_0 \cdots \otimes a_{i-1} \otimes 1,
%\end{split}
%\end{equation}
%where $\epsilon_i = (-1)^{(| a_0| + \cdots + | a_{i-1}|-i)(| a_i| + \cdots + | a_p|-p+i-1)}$. 
For any $\iD\in\Dpolym{d}$, let $\iI_\iD:\Cpolym{-p}\to\Cpolym{-p+d}$ be the map 
\begin{equation}\label{eq:iI}
\iI_\iD(a_0 \otimes a_1 \cdots \otimes a_p) = (-1)^{|\iD||a_0|} a_0
\iD ( a_1 , \cdots , a_d) \otimes a_{d+1} \cdots \otimes a_p, 
\end{equation}
and $\iL_\iD:\Cpolym{-p}\to\Cpolym{-p+d-1}$ the map
\begin{multline}\label{eq:iL}
    \iL_\iD (a_0 \otimes \ldots \otimes a_p) = \sum _{k=1}^{p-d} \epsilon_k \cdot  a_0 \otimes \ldots \otimes \iD(a_{k+1}, \ldots, a_{k+d}) \otimes \ldots \otimes a_p  \\
    + \sum _{k=p+1 -d}^{n} \eta_k \cdot  \iD (a_{k+1}, \ldots, a_p, a_0, \ldots ) \otimes \ldots \otimes a_k,  
\end{multline}
% \begin{eqnarray}
% \iL_\iD (a_0 \otimes \ldots \otimes a_p)&=&\sum _{k=1}^{p-d} \epsilon _k a_0
% \otimes \ldots \otimes \iD(a_{k+1}, \ldots, a_{k+d}) \otimes \ldots \otimes
% a_p \nonumber
% \\
% &+&
% \sum _{k=p+1 -d}^{n} \eta _k \iD (a_{k+1}, \ldots, a_p, a_0, \ldots )
% \otimes \ldots \otimes a_k,
% \label{eq:iL}
% \end{eqnarray}
where the second sum in \eqref{eq:iL} is taken over all cyclic permutations
such that $a_0$ is inside $\iD$.
Here the signs are $\epsilon_k = (-1)^{(|\iD|+1) \sum _{i=0}^{k} (|a_i|+1)}$ and $\eta _k = (-1)^{|\iD|+ 1 + \sum_{i \leq k}(|a_i|+1)\sum_{i \geq k}(|a_i|+1)} .$

The Hochschild boundary differential is defined by $\hochschildb = \iL_m :\CpolyM{-p}{n}\to\CpolyM{-p+1}{n+1},$ 
% \begin{equation}
% \label{eq:hochschildbp0}
% \hochschildb = \iL_m :\CpolyM{-p}{n}\to\CpolyM{-p+1}{n+1},
% \end{equation}
where $m$ is the shifted multiplication defined in Equation~\eqref{eq:ShiftedMultiplication}. 
The homological vector field $Q$ induces a square zero differential $\ccQ$ by Lie derivative $\ccQ = \liederivative{Q}:\CpolyM{-p}{n}\to\CpolyM{-p}{n+1} .$
% \[ \ccQ = \liederivative{Q}:\CpolyM{-p}{n}\to\CpolyM{-p}{n+1} .\]
Since $Q$ is a homological vector field, it follows that $[\hochschildb, \ccQ ]=0$. 
Hence $\big(\ccQ+\hochschildb\big)^2=0$. The Hochschild 
homology of the dg manifold $(\cM,Q)$,
denoted $\HCC$, is the cohomology of the direct-product total
cochain complex $\CCC$ of the double complex $(\Cpolym{},\hochschildb, \ccQ)$.

%\begin{mathieu}One term missing in the formula for $L$?\end{mathieu}
%\begin{multline}\label{eq:lL}
%\iL_\iD(a_0\otimes\cdots\otimes a_n)
%=\sum_{k=0}^{n-d}\epsilon_k\
%a_0\otimes\cdots\otimes a_k\otimes\iD(a_{k+1},\cdots,a_{k+d})
%\otimes a_{k+d+1}\otimes\cdots\otimes a_n \\
%+\sum_{k=n-d+1}^{n}\eta_k\
%\iD(a_{k+1},\cdots,a_n,a_0,\cdots,a_{k+d-n-1})
%\otimes a_{k+d-n}\cdots\otimes a_k
%\end{multline}
%(The second sum in the above formula is taken over all cyclic permutations
%such that $a_0$ is inside $\iD$). The signs are given by
%\[ \epsilon_k = (\degree{\iD}+1)\sum _{i=0}^{k}(\degree{a_i}+1) \]
%and
%\[ \eta_k = {\color{red}\degree{\iD}+1+}\sum_{i\leq k}(\degree{a_i}+1)
%\sum_{i{\color{red}\geq} k}(\degree{a_i}+1) \]

It is standard that all the operations $B$, $\iI_\iD$ and $\iL_\iD$ descend to well defined operators on cohomology groups, and the following proposition holds \cite{MR2986860}.

\begin{proposition}
Let $(\cM,Q)$ be a dg manifold.
The pair \[ \Theta^{\bullet}=\HDD \qquad\text{and}\qquad \Xi_{\bullet}=\HCC \]
admits a Tamarkin--Tsygan calculus structure,
whose differential $d$ is the \RinehartConnes operator $B$ \eqref{eq:Connes1}
and whose action operators $\iI_\iD$ and $\iL_\iD$ are defined
explicitly as in Equations~\eqref{eq:iI} and~\eqref{eq:iL}.
\end{proposition}

Such a calculus is denoted $\calculus_H(\cM,Q)$, which is a dg extension
of $\calculus_H(M)$ in Section~\ref{set:1.1}.

\section{Tamarkin--Tsygan calculi of dg Lie algebroids}
\label{sec:CalculiDGLieAbd}

In this section, we will extend Tamarkin--Tsygan calculi to any dg Lie algebroids.
%In fact,
% for Hochschild (co)homology of $\ZZ$-graded Lie algebroids $\cL$,
%the formulas in the previous section
% do not have an obvious extension.
%Such structures and explicit formulas depend on the ``formal groupoid"
%structure of the jet space $\jet{\cL}$ \cite{MR3456700,paper-zero}.

\subsection{Differential graded Lie algebroids}

In this subsection, we recall some basic notations concerning 
dg vector bundles and dg Lie algebroids following Mehta \cite{MR2534186, MR2709144}.
Let $\cM$ be a graded manifold, and $\cE\xto{\pi}\cM$ a graded vector bundle.
%in the category of $\ZZ$-graded manifolds,
Its space of sections $\sections{\cE}$ is defined to be the direct sum
$\bigoplus_{j\in\ZZ}\Gamma (\cE)^j$,
where $\Gamma (\cE)^j$ denotes the space of degree preserving maps
$s\in\Hom(\cM,\cE[-j])$ such that $(\pi[-j])\circ s=\id_\cM$.
Here $\pi[-j]:\cE[-j]\to\cM$ is the natural map induced by $\pi:\cE\to\cM$
--- see~\cite{MR2534186} for more details.
When $\cE\to\cM$ is a dg vector bundle, the homological vector fields on $\cE$ and $\cM$
naturally induce an operator $\cQ$ of degree $(+1)$ on $\sections{\cE}$,
making $\sections{\cE}$ into a dg module over $(C^\infty(\cM) , Q)$.
Since $C^\infty(\cM)$ and the space $\sections{\cE^\vee}$ of linear functions on $\cE$
together generate the algebra $C^\infty(\cE)$, the converse is also true:
the homological vector field on $\cM$ and the operator $\cQ$ on $\sections{\cE}$
determine a dg structure on $\cE$.

%Now assume that $\cE$ is a dg vector bundle over the dg manifold $(\cM,Q)$.
%By assumption, $\sections{\cE}$ is a dg module over
%$\big(C^\infty(\cM),Q\big)$ with degree $(+1)$ differential
%$\cQ:\sections{\cE}\to\sections{\cE}$.
Define a degree $(+1)$ operator
$\cQ:\sections{\cE^\vee}\to\sections{\cE^\vee}$ by
\[ \duality{\cQ(\xi)}{l}=
Q\duality{\xi}{l}-(-1)^{\degree{\xi}}\duality{\xi}{\cQ(l)} \]
for all homogeneous $\xi\in\sections{\cE^\vee}$
and $l\in \sections{\cE}$.
It is simple to see that this operator makes $\sections{\cE^\vee}$
into a dg module over $\big(C^\infty(\cM),Q\big)$.
Thus $\cE^\vee \to \cM$ is also a dg vector bundle.

Set, $\forall k\geq 0$, the kth-exterior product of $\cE\to \cM$ as $\Lambda^{k}\cE:=S^k (\cE[-1])[k].$  
It is simple to see that $\Lambda^{k}\cE\to \cM$ is also a dg vector bundle, where the dg module structure $\cQ:\sections{\Lambda^{k}\cE}
\to\sections{\Lambda^{k}\cE}$ over $\big(C^\infty(\cM),Q\big)$
is given by 
\begin{equation}\label{eq:wedge1}
\cQ(l_1\wedge\cdots\wedge l_k)
=\sum_{i=1}^k (-1)^{\degree{l_1}+\cdots+\degree{l_{i-1}}}
l_1\wedge\cdots\wedge\cQ(l_i)\wedge\cdots\wedge l_k
\end{equation}
for all homogeneous $l_1,\dots,l_k\in\sections{\cE}$.

We summarize the discussion above in the following

\begin{lemma}
\label{lem:dgvector}
Let $\cE$ be a dg vector bundle over $(\cM,Q)$. Then
\begin{enumerate}
\item the dual bundle $\cE^\vee$ is a dg vector bundle over
$(\cM,Q)$; and
\item for all $k\geq 1$,
both the exterior tensor power vector bundle $\Lambda^k\cE$ and
 $\Lambda^k\cE^\vee$ are dg vector bundles over $(\cM,Q)$.
\end{enumerate}
\end{lemma}

Recall that a ($\ZZ$-)graded Lie algebroid \cite{MR2709144}
% object in the category of $\ZZ$-graded manifolds
consists of a $\ZZ$-graded vector bundle $\cL\to\cM$
% object $\cL\to\cM$ in the category of $\ZZ$-graded manifolds
together with a bundle map $\rho:\cL\to T_{\cM}$ of degree $0$, called anchor,
and a structure of graded Lie algebra on $\sections{\cL}$
with Lie bracket satisfying the following condition:
\[ \lie{X}{fY} = \rho_X (f) Y + (-1)^{\degree{X}\degree{f}} f \lie{X}{Y} \]
for all homogeneous $X,Y\in\sections{\cL}$ and $f\in C^\infty(\cM)$.

The Chevalley--Eilenberg differential $d_\cL:\sections{\Lambda^k\cL\dual}\to\sections{\Lambda^{k+1}\cL\dual}$ 
is defined by
\begin{multline}
\label{eq:CE2}
\big(d_\cL\omega\big)(X_0,X_1,\cdots,X_k)
=\sum_{i=0}^{n} \pm
\rho_{X_i}\big(\omega(X_0,\cdots,\widehat{X_i},\cdots,X_k)\big) \\
+\sum_{i<j} \pm \omega(\lie{X_i}{X_j},X_0,\cdots,\widehat{X_i},\cdots,\widehat{X_j},\cdots,X_k) .
\end{multline}
%and the exterior product make $\bigoplus_{k\geq 0}\sections{\Lambda^k L\dual}$
%into a differential graded commutative $R$-algebra.

A dg Lie algebroid is a Lie algebroid object in the category of dg manifolds.
Equivalently, a dg Lie algebroid is a dg vector bundle $\cL\to\cM$
endowed with a graded Lie algebroid structure
satisfying the compatibility condition
\begin{equation}\label{eq:compatibility}
\schouten{d_{\cL}}{\cQ}=0,
\end{equation}
where $d_{\cL}$ is the Chevalley--Eilenberg differential \eqref{eq:CE2}
%\begin{equation}\label{eq:CE}
%d_{\cL}:\sections{\Lambda^\bullet\cL^\vee}
%\to\sections{\Lambda^{\bullet+1}\cL^\vee}
%\end{equation}
of the Lie algebroid $\cL$, while $\cQ$ is the differential (of internal degree $(+1)$)
\begin{equation}\label{eq:intQ}
\cQ:\sections{\Lambda^k\cL^\vee} \to\sections{\Lambda^k \cL^\vee}
\quad \quad \quad \forall k\geq 0
\end{equation}
induced by the dg vector bundle structure on $\cL\to\cM$
(see Lemma~\ref{lem:dgvector}).
For more details, we refer the reader to~\cite{MR2534186,MR2709144},
where dg Lie algebroids are called $Q$-algebroids.

\begin{example}\label{exp:dgtangent2}
%As in Example~\ref{exp:dgtangent1}, let $(\cM,Q)$ be a dg manifold.
Let $(\cM, Q)$ be a dg manifold.
The space $\XX(\cM)$ of vector fields
on $\cM$ (i.e.\ graded derivations of $C^\infty(\cM)$),
which can be regarded as the space of sections $\Gamma(T_\cM)$,
is naturally a dg module over the dg algebra $(C^\infty(\cM),Q)$
with the Lie derivative $\liederivative{Q}:\XX(\cM)\to\XX(\cM)$
playing the role of the operator $\cQ$.
As a consequence, $T_\cM $ is a dg manifold
--- the corresponding homological vector field on $T_\cM$ is called the
\emph{complete lift} of $Q$ as well as tangent lift
in~\cite{MR3319134} --- and $T_\cM\to\cM$ is a dg vector bundle.

In addition to being a dg vector bundle,
$T_{\cM}$ is also a Lie algebroid.
In this case, the Chevalley--Eilenberg differential \eqref{eq:CE2}
is the de Rham differential $d_{\DR}: \Omega^\bullet (\cM) \to \Omega^{\bullet+1} (\cM),$ while the internal differential \eqref{eq:intQ} is the Lie derivative $\iL_Q : \Omega^\bullet (\cM) \to \Omega^\bullet (\cM).$ Since $\schouten{d_{\DR}}{L_Q}=0$, it follows that $T_{\cM}$ is indeed a dg Lie algebroid.
\end{example}

%Denote by $\totTpolyL{\bullet}$ the graded left $\cR$-module defined by
%\[ \totTpolyL{n}=\bigoplus_{p+q=n}\TpolyL{p}{q} ,\]
%where $\TpolyL{p}{q}$ is the space of $p$-vector fields on $\cL$ of degree $q$.
%Similarly, denote by $\totDpolyL{\bullet}$ the graded left $\cR$-module
%defined by
%\[ \totDpolyL{n}=\bigoplus_{p+q=n}\DpolyL{p}{q} ,\]
%where $\DpolyL{p}{q}$ is understood as the space of
%$p$--differential operators on $\cL$ of degree $q$.

\subsection{Tamarkin--Tsygan calculus structure $\calculus_C(\cL,\cQ)$}\label{sec:CalCdgAbd}

Given a dg Lie algebroid $\cL$, denote by $\cQ:\sections{\Lambda^{k}\cL}\to\sections{\Lambda^{k}\cL}, 
\,  \forall k\geq 0$, the degree $(+1)$ differential as in \eqref{eq:intQ}.

The Lie algebroid structure on $\cL$ yields a Schouten bracket $\schouten{\argument}{\argument}: \sections{\Lambda^{p}\cL}\otimes
\sections{\Lambda^{q}\cL}\to \sections{\Lambda^{p+q-1}\cL}.$ 
Denote by $\wedge: \sections{\Lambda^{p}\cL}\otimes\sections{\Lambda^{q}\cL}\to \sections{\Lambda^{p+q}\cL}$ the wedge product. 

For any $X\in \sections{\Lambda^d \cL}$, denote, respectively, by
\begin{align}
    & \iI_X:\sections{\Lambda^n\cL^\vee}\to\sections{\Lambda^{n-d}\cL^\vee}, \label{eq:IX1} \\
    & \iL_X:\sections{\Lambda^n\cL^\vee}\to\sections{\Lambda^{n-d+1}\cL^\vee}, \label{eq:LX1}
\end{align}
the contraction operator and the Lie derivative operator, where the second is defined by the Cartan formula:
\begin{equation}\label{eq:LX2}
\iL_X=(-1)^{|X|-1} \schouten{d_{\cL}}{\iI_X}.
\end{equation}

%satisfy the identities of calculi \eqref{!!!!!}.

For any $p\geq 0$, let $\Tpolyl{p}$ denote the space
$\sections{\cM;S^p(\cL[-1])}=\sections{\Lambda^{p}\cL}[-p]$.
We use
% the symbol $\TpolyL{0}{q}$ to denote the space of smooth functions of degree $q$ on $\cM$ and
the symbol $\TpolyL{p}{n}$ to denote the degree $n$ subspace of $\Tpolyl{p}$.
Set $\Tpolyl{} = \bigoplus_{p \geq 0} \Tpolyl{p},$ which is equipped with the bigrading: $\big(\Tpolyl{}\big)^{p,q} = \TpolyL{p}{p+q} .$ 
We have the associated direct-sum total space: 
\begin{equation}\label{eq:GaredeNord1}
\totTpolyL{n}=\bigoplus_{p+q=n}\big(\Tpolyl{}\big)^{p,q}
=\bigoplus_{p=0}^\infty\TpolyL{p}{n}
.\end{equation}

Similarly, for any $p\geq 0$, 
let $\Apolyl{-p}:=\sections{\cM;S^p(\cL^\vee[1])}
=\sections{\Lambda^{p}\cL^\vee}[p]$.
We use 
%the symbol $\ApolyL{0}{q}$ to denote the space of smooth functions of degree $q$ on $\cM$ and
the symbol $\ApolyL{-p}{n}$ to denote the degree $n$
subspace of $\Apolyl{-p}$.
Set $\Apolyl{} = \bigoplus_{p\geq 0} \Apolyl{-p},$ which is equipped with the bigrading: $\big(\Apolyl{}\big)^{-p,q} = \ApolyL{-p}{-p+q} .$ 
We have the direct-product total complex:
\begin{equation}\label{eq:GaredeEst1}
\totApolyL{n}=\prod_{-p+q=n}\big(\Apolyl{}\big)^{-p,q}
=\prod_{p=0}^\infty\ApolyL{-p}{n}.
\end{equation}
Similarly to the dg manifold case, to fit the Duflo--Kontsevich type theorems,
we consider the direct-sum total space for $\Tpolyl{}$
and the direct-product total space for $\Apolyl{}$.
We will use a parallel setting for $\Dpolyl{}$ and $\Cpolyl{}$ in the next section.

The operations $\iI$ and $\iL$ in~\eqref{eq:IX1}
and ~\eqref{eq:LX1} naturally induces operations on $\totTpolyL{\bullet}$ and $\totApolyL{\bullet}$. 
Furthermore, the compatibility condition \eqref{eq:compatibility} implies that the dg structure $\cQ$, i.e.\ the differentials
as in \eqref{eq:intQ}, is compatible with the operations $\wedge$, $\schouten{\argument}{\argument}$, $\iI$ and $\iL$ in a natural sense. Thus, we have the following

\begin{proposition}\label{prop:CalC_DGLieAbd}
Let $\cL$ be a dg Lie algebroid.
The pair 
\[ \Theta^{\bullet}=\cohomology{}\big(\totTpolyL{\bullet},\cQ\big)
\quad\quad\text{and}\quad\quad
\Xi_{\bullet}=\cohomology{}\big(\totApolyL{\bullet},\cQ\big) \]
admits a Tamarkin--Tsygan calculus structure, 
whose operators $d$, $\iI_X$, and $\iL_X$ are respectively
the Chevalley--Eilenberg Lie algebroid differential $d_{\cL}$ \eqref{eq:CE2},
the interior product w.r.t.\ the polyvector field $X$ \eqref{eq:IX1},
and the Lie derivative w.r.t.\ $X$ \eqref{eq:LX1}.
\end{proposition}

Such a calculus is denoted by $\calculus_C(\cL,\cQ)$.
In the case $\cL = T_\cM$, the calculus $\calculus_C(\cL,\cQ)$ reduces to the Cartan calculus
$\calculus_C(\cM,\cQ)$ of the dg manifold $(\cM,Q)$ in Section~\ref{corniche}.

\subsection{Universal enveloping algebras of graded Lie algebroids}

In this section, we recall some basic material regarding universal enveloping
algebras of graded Lie algebroids----see~\cite{MR1815717} for the
general theory of the universal enveloping algebra of a Lie algebroid
and~\cite[Appendix]{MR3964152} for an instance of the extension
of the concept to \emph{dg} Lie algebroids.

For the moment, we consider a graded Lie algebroid $\cL $ over $\cM$
with anchor $\rho:\cL\to T_\cM$.
Recall that the universal enveloping algebra $\enveloping{\cL}$
of $\cL$ is the quotient of the (reduced) tensor algebra
\begin{equation}\label{eq:tensor}
\bigoplus_{n=1}^{\infty}\big(\cR\oplus\sections{\cL}\big)^{\otimes n}
\end{equation}
of the $\KK$-module $\cR\oplus\sections{\cL}$ by the two-sided ideal
generated by the elements of the following four types:
\begin{align}
& X\otimes Y-(-1)^{\degree{X}\degree{Y}}Y\otimes X-\lie{X}{Y} 
&& f\otimes X-fX \nonumber \\
& X\otimes g-(-1)^{\degree{g}\degree{X}}g\otimes X-\rho_X(g) 
&& f\otimes g-fg \label{eq:four}
\end{align}
for all homogeneous $X,Y\in\sections{\cL}$ and $f,g\in\cR$.
As earlier, the symbol $\cR$ denotes $C^\infty(\cM)$.
%, denoting $C^\infty(\cM)$ by $\cR$,

%There is a natural filtration \cite{MR154906}:
%\begin{equation}
%\label{Jakarta}
%\cR\into \mathcal{U}^{\leqslant 1}(\cL) \into \mathcal{U}^{\leqslant 2}(\cL)
%\into \mathcal{U}^{\leqslant 3}(\cL) \into \cdots
%\end{equation}
%where $\mathcal{U}^{\leqslant p}(\cL)=
%\text{Span}_\cR \big\{ X_1X_2 \cdots X_d|X_i\in \sections{\cL}, d\leq p
%\big\}$.

%Denote by $\totDpolyL{\bullet}$ the graded left $\cR$-module
%defined by
%\[ \totDpolyL{n}=\bigoplus_{p+q=n}\DpolyL{p}{q} ,\]
%where $\DpolyL{p}{q}$ is understood as the space of
%$p$--differential operators on $\cL$ of degree $q$.

It is clear that the universal enveloping algebra $\enveloping{\cL}$
is a graded associative algebra under the multiplication
$m: \enveloping{\cL}\times \enveloping{\cL}\to \enveloping{\cL}$
induced by the tensor product over $\KK$. Let $\alpha, \beta: \cR\to \enveloping{\cL}$
be the inclusion maps, called the source and target maps. Their
induced $\cR$-module structures on $\enveloping{\cL}$
coincide with the canonical left $\cR$-module structure. 
% \[ r \cdot u=ru, \ \forall r\in \cR, \ u \in \enveloping{\cL} .\]
In the mean time,
% of a dg Lie algebroid $\cL\to \cM$
$\enveloping{\cL}$ is also a graded cocommutative
coalgebra over $\cR$ \cite{MR1815717}.
%Note that $\enveloping{\cL}$ is a left $\cR$-module in a natural fashion.
Its graded cocommutative comultiplication is
the $\cR$-module map
\begin{equation}\label{eq:SCE}
\Delta:\enveloping{\cL}\to\enveloping{\cL}\otimes_\cR\enveloping{\cL} 
\end{equation}
completely determined by the identities below:
\begin{gather*}
\Delta(1)=1\otimes 1; \\
\Delta(b)=1\otimes b+b\otimes 1, \quad\forall b\in\sections{\cL}; \\
\Delta(u\cdot v)=\Delta(u)\cdot\Delta(v), \quad\forall u,v\in\enveloping{\cL} ,
\end{gather*}
where the symbol $\cdot$ denotes the multiplication in $\enveloping{\cL}$.
%We refer the reader to~\cite{MR1815717}
See~\cite{MR1815717} for the precise meaning of
(the r.h.s.\ of) the last equation above.
More explicitly, we have 
\begin{multline}
\label{eq:PAR}
\Delta(b_1 b_2\cdots b_n) =1\otimes(b_1 b_2\cdots b_n)
+(b_1 b_2\cdots b_n)\otimes 1 \\ 
+\sum_{\substack{p+q=n \\ p,q\in\NN}}\sum_{\sigma\in\shuffle{p}{q}}
\epsilon(\sigma;b_1,\cdots,b_n)\
(b_{\sigma(1)}\cdots b_{\sigma(p)})\otimes
(b_{\sigma(p+1)}\cdots b_{\sigma(n)}) 
,\end{multline}
%where $\epsilon(\sigma,b_1,\dots,b_n)$ denotes the Kozul sign
%of the permutation $\sigma\in\shuffle{p}{q}$
%of the $n$-tuple of homogeneous elements
%$b_1,\dots,b_n$ of $\sections{\cL}$.
where $\shuffle{p}{q}$ denotes the set of $(p,q)$-shuffles
and $\epsilon(\sigma;b_1,\cdots,b_n)$ denotes the Koszul sign of the permutation
$\sigma$ of the homogeneous elements $b_1,\cdots, b_n\in\Gamma(\cL)$.

There is a natural left $\enveloping{\cL}$-action on $\cR$ given by
\[ u(f):=\rho_{u}(f), \qquad \forall u\in \enveloping{\cL}, f\in\cR \]
where $\rho_{u}$ is a differential operator on $\cM$ induced
by the anchor map $\rho$.
By $\epsilon: \enveloping{\cL}\to \cR$, we denote the counit map defined by 
\[ \epsilon(u)=u(1)=\rho_{u}(1), \qquad
\forall u\in\enveloping{\cL} .\]
Equivantly, $\epsilon$ is the projection map to the component of $\cR$.
In fact, $\enveloping{\cL}$ together with the structure maps $(m,\alpha,\beta,\Delta,\epsilon)$
becomes a (graded) Hopf algebroid over $\cR$~\cite{MR1815717}.\footnote{This structure
is often called a bialgebroid structure, instead.}
 
\subsection{Tamarkin--Tsygan calculus structure $\calculus_H(\cL,\cQ)$}\label{sec:CalHdgAbd}

Now we will describe the noncommutative
calculus of a dg Lie algebroid $\cL$ over $\cM$.
Note that for Hochschild (co)homology of graded Lie algebroids $\cL$,
the formulas in Section~\ref{sec:HtpCalDGMfd}
do not have an obvious extension.
Such structures and explicit formulas depend on the ``formal groupoid"
structure of the jet space $\jet{\cL}$ \cite{MR3456700,paper-zero}.

Let $\enveloping{\cL}$ be the universal enveloping algebra of $\cL$.
The differential $\cQ:\sections{\cL}\to\sections{\cL}$
and the homological vector field $Q:C^\infty(\cM)\to C^\infty(\cM)$
induce a differential of degree $(+1)$ on the (reduced) tensor algebra
\eqref{eq:tensor} by Leibniz rule. From the compatibility
condition \eqref{eq:compatibility}, it is simple to see
that the two-sided ideal generated by the elements
\eqref{eq:four} is stable under this induced differential. 
Hence we obtain a degree $(+1)$ differential on the
universal enveloping algebra, denoted by the same symbol
$\cQ$ by abuse of notation:
\begin{equation}\label{eq:UL}
\cQ:\enveloping{\cL}\to\enveloping{\cL}
,\end{equation}
which makes $\enveloping{\cL}$ into a differential graded algebra.
%we denote by the same symbol $\cQ$ by abuse of notation. 

By $\pshift\, \enveloping{\cL}$, we denote the suspended
universal enveloping algebra
$\enveloping{\cL}[-1]$, which is naturally
an $\cR$-module as well.
For any $p\geq 1$, let $\Dpolyl{p}$ denote
$(\pshift\,\enveloping{\cL})^{\otimes p}$,
the tensor product 
%%%%FFFFFFFFFFFFFFFFF
% \[ \underset{\text{$p$-factors}}{\underbrace{\pshift\,\enveloping{\cL} \otimes\cdots\otimes \pshift\,\enveloping{\cL}}} \]
%$\pshift\,\enveloping{\cL}\otimes_\cR\cdots\otimes_\cR \pshift\,\enveloping{\cL}$
of $p$-copies of the left $\cR$-module $\pshift\,\enveloping{\cL}$.
For $p=0$, let $\Dpolyl{0}:=\cR$. 

The tensor product of left $\cR$-modules determines a cup product
\begin{equation}\label{eq:cup}
\cup:\DpolyL{p}{n}\times\DpolyL{p'}{n'}\to\DpolyL{p+p'}{n+n'}
.\end{equation}

The differential \eqref{eq:UL} induces a degree $(+1)$ differential $\cQ:\pshift\,\enveloping{\cL}\to \pshift\,\enveloping{\cL},$
% \begin{equation}\label{eq:sUL}
% \cQ:\pshift\,\enveloping{\cL}\to \pshift\,\enveloping{\cL},
% \end{equation}
which makes $\pshift\,\enveloping{\cL}$ into a dg module over the 
\dga $(\cR,Q)$.
%where $\cR=C^\infty(\cM)$. 
As a consequence, we obtain a degree $(+1)$ differential
for each $p\geq 0$
\begin{equation}\label{eq:QUL}
\cQ: \DpolyL{p}{n} \to \DpolyL{p}{n+1}
.\end{equation}

It is known that $\Dpolyl{}$ admits a Gerstenhaber bracket, 
$$\gerstenhaber{\argument}{\argument}: \Dpolyl{p+1} [1]
\otimes \Dpolyl{q+1} [1] \to \Dpolyl{p+q+1}[1],$$
which can be explicitly expressed in terms
of the Hopf algebroid structure on $\enveloping{\cL}$, 
\begin{equation}\label{eq:Gbraket}
    \gerstenhaber{\phi}{\psi} = \phi\star\psi
- (-1)^{|\phi||\psi|} \psi\star\phi.
\end{equation}

% \begin{equation}\label{eq:Gbraket}
% \begin{split}
% \gerstenhaber{\argument}{\argument}:& \Dpolyl{p+1} [1]
% \otimes \Dpolyl{q+1} [1] \to \Dpolyl{p+q+1}[1], \\
% %and
% %\begin{equation}\label{hazmat}
% & \gerstenhaber{\phi}{\psi} = \phi\star\psi
% - (-1)^{|\phi||\psi|} \psi\star\phi.
% %\in\UcL{u+v}
% \end{split}
% \end{equation}
To express $ \phi\star\psi$ explicitly, we identify $\Dpolyl{p+1}[1]$
with $\big((\pshift)^{\otimes p+1}\enveloping{\cL}^{\otimes p+1}\big)[1]$
by the natural isomorphism
\begin{equation}\label{eq:Shift_2_Identification}
\begin{split}
\big( (\pshift)^{\otimes p+1} \enveloping{\cL}^{\otimes p+1}\big)[1]
&\cong \big(\pshift\,\enveloping{\cL}\big)^{\otimes p+1}[1] = \Dpolyl{p+1} [1] \\
\nshift \big((\pshift)^{\otimes p+1}(u_0 \otimes\cdots\otimes u_p)\big)
& \mapsto (-1)^{\sum_{i=0}^p |u_i|(p-i)} \nshift( \pshift u_0 \otimes \cdots \otimes \pshift u_p)
.\end{split}
\end{equation}
We define $\phi\star\psi\in\big(\shift^{-1} \, \enveloping{\cL}\big)^{\otimes p+q+1}[1]$
by the formulas
%\begin{equation}\label{eq:pre-Lie}
\begin{multline*}
\nshift \big((\pshift)^{\otimes p+1}(u_0 \otimes \cdots\otimes u_p)\big) \star
\nshift \big((\pshift)^{\otimes q+1}(v_0 \otimes \cdots \otimes v_q)\big) \\
=\sum_{k=0}^{p} \epsilon_{k} \, \nshift \big((\pshift)^{\otimes p+q+1}
\big((u_0\otimes\cdots u_{k-1}\otimes\Delta^q(u_k)\otimes u_{k+1}\cdots\otimes u_p)
\cdot(1^{\otimes k}\otimes v_0\otimes\cdots\otimes v_q\otimes 1^{\otimes p-k})\big)\big)
\end{multline*}
and
\[ \nshift \big((\pshift)^{\otimes p+1}(u_0 \otimes \cdots\otimes u_p)\big) \star \nshift f
= \sum_{k=0}^p \epsilon_{k}' \, \nshift \big((\pshift)^{\otimes p} \big(u_k(f) \cdot u_0
\otimes \cdots \widehat{u_k} \cdots \otimes u_p\big) \big) \]
for $f\in\cR$ and $u_0, \dots,u_p, v_0, \cdots, v_q\in \enveloping{\cL}$.
The signs $\epsilon_{k}$ and $\epsilon_{k}'$ are determined by the Koszul rule: 
\begin{align*}
    \epsilon_{k}& =(-1)^{q(p-k)+q(|u_0|+\cdots +|u_p|)}, \\
    \epsilon_{k}'&=(-1)^{(|f|-1)(|u_0|+\cdots\widehat{|u_k|}\cdots+|u_p|)
+p-k+|u_k|(1+|u_0|+\cdots+|u_{k-1}|)}.
\end{align*}
% $\epsilon_{k}=(-1)^{q(p-k)+q(|u_0|+\cdots +|u_p|)}$
% and $\epsilon_{k}'=(-1)^{(|f|-1)(|u_0|+\cdots\widehat{|u_k|}\cdots+|u_p|)
% +p-k+|u_k|(1+|u_0|+\cdots+|u_{k-1}|)}$.

%Here, we identify $\Dpolyl{p+1}[1]$ with
%$\big((\pshift)^{\otimes p+1}\enveloping{\cL}^{\otimes p+1}\big)[1]$
%by the natural isomorphism
%\begin{equation}\label{eq:Shift_2_Identification}
%\begin{split}
%\big( (\pshift)^{\otimes p+1} \enveloping{\cL}^{\otimes p+1}\big)[1]
%&\cong \big(\pshift\, \enveloping{\cL}\big)^{\otimes p+1}[1] = \Dpolyl{p+1} [1] \\
%\nshift \big((\pshift)^{\otimes p+1}(u_0 \otimes \cdots\otimes u_p)\big)
%& \mapsto (-1)^{\sum_{i=0}^p |u_i|(p-i)} \nshift( \pshift u_0 \otimes \cdots \otimes \pshift u_p).
%\end{split}
%\end{equation}
%\blue{For simplicity, we will use the notation
%$u_0 \otimes \cdots\otimes u_p$ to denote
%$\nshift \big((\pshift)^{\otimes p+1}(u_0 \otimes \cdots\otimes u_p)\big)$.}

By $\hochschild$, we denote the Hochschild differential
\begin{equation}\label{eq:hochschild}
\hochschild:=\gerstenhaber{m}{\argument}:\quad
\DpolyL{p}{n}\to\DpolyL{p+1}{n+1}
,\end{equation}
where
\begin{equation}\label{eq:m}
m=-\nshift(\pshift 1 \otimes_\cR \pshift 1)\in\DpolyL{2}{2}
.\end{equation} This is compatible with Equation~\eqref{eq:ShiftedMultiplication} in the following way: An element $\pshift u \in \pshift \, \enveloping{\cL}$ induces a map $\nshift \cR \to \cR$ by sending $\nshift f$ to $(-1)^{|u|} u(f)$. Accordingly, a tensor $\pshift u \otimes \pshift v$, with $u,v \in \enveloping{\cL}$, induces the map $\nshift \cR \otimes \nshift \cR \to \cR$ given by $ \nshift f \otimes \nshift g \mapsto (-1)^{|\pshift v||\nshift f|+|u| + |v|} u(f) v(g)$. In particular, the map induced by $\pshift 1 \otimes \pshift 1$ sends $ \nshift f \otimes \nshift g$ to $-(-1)^{|f|} f g = -m(\nshift f,\nshift g)$.

With the previous identification, we can express $\hochschild$ explicitly by $\hochschild=\sum_{i=0}^{n+1}\partial^i$, where 
\[ \partial^i\big(\nshift\big((\pshift)^{\otimes n}
(u_1\otimes\cdots\otimes u_n)\big)\big)
= (-1)^{i+1}\, \nshift \big((\pshift)^{\otimes n+1}
(u_1\otimes\cdots\otimes(\Delta u_i)\otimes\cdots\otimes u_n)\big) \]
for $1 \leq i\leq n$, and 
\begin{align*}
\partial^0 \big( \nshift \big((\pshift)^{\otimes n}( u_1 \otimes \cdots \otimes u_n )\big)\big)
& = - \nshift \big((\pshift)^{\otimes n+1}( 1 \otimes u_1 \otimes \cdots \otimes u_n)\big) ,\\
\partial^{n+1} \big( \nshift \big((\pshift)^{\otimes n}( u_1 \otimes \cdots \otimes u_n )\big)\big)
& = (-1)^{n}\, \nshift \big((\pshift)^{\otimes n+1}(u_1 \otimes \cdots \otimes u_n \otimes 1)\big)
.\end{align*}

%\gray{
%$$\partial^i \big(\pshift u_1 \otimes \cdots \otimes \pshift u_n \big)
%=\pshift u_1 \otimes \cdots \otimes (\pshift \otimes \pshift) (\Delta u_i)
%\otimes \cdots \otimes \pshift u_n$$
%for $0\leq i\leq n$ and $\partial^0 \big(\pshift u_1 \otimes \cdots \pshift u_n \big)
%=\pshift 1 \otimes \pshift u_1 \otimes \cdots \pshift u_n$, 
%$\partial^{n+1} \big(\pshift u_1 \otimes \cdots \pshift u_n \big)
%=\pshift u_1 \otimes \cdots \pshift u_n \otimes \pshift 1$.}
%
%\ping{ANY sign?}
%We can express $\hochschild$ explicitly by
%\[ \hochschild=\sum_{i=0}^{n+1}\partial^i \]
%where 
%$$\partial^i \big(\pshift u_1 \otimes \cdots \otimes \pshift u_n \big)
%=\pshift u_1 \otimes \cdots \otimes (\pshift \otimes \pshift) (\Delta u_i)
%\otimes \cdots \otimes \pshift u_n$$
%for $0\leq i\leq n$ and $\partial^0 \big(\pshift u_1 \otimes \cdots \pshift u_n \big)
%=\pshift 1 \otimes \pshift u_1 \otimes \cdots \pshift u_n$, 
%$\partial^{n+1} \big(\pshift u_1 \otimes \cdots \pshift u_n \big)
%=\pshift u_1 \otimes \cdots \pshift u_n \otimes \pshift 1$.

The differential \eqref{eq:UL}
is a derivation with respect to the product and a 
coderivation with respect to the coproduct so that it
is compatible with both the algebra and coalgebra structures
on $\enveloping{\cL}$. In other words, $\enveloping{\cL}$ is a dg
Hopf algebroid over the \dga $(\cR,Q)$. See \cite{MR1815717,paper-zero} 
for the Hopf algebroid structure.
As a consequence, we have $[\cQ,\hochschild]=0$, 
and therefore $\big(\dQ+\hochschild\big)^2=0$.
% where $\hochschild: \ \pshift\,\enveloping{\cL}^{\otimes p}
%\to \pshift\,\enveloping{\cL}^{\otimes p+1}$
%is the Hochschild cohomology differential \eqref{eq:hochschild}.

%The tensor product of left $\cR$-modules determines a cup product
%\begin{equation}
%\label{eq:cup}
%\cup:\DpolyL{p}{n}\times\DpolyL{p'}{n'}\to\DpolyL{p+p'}{n+n'}
%.\end{equation}

%For any $p\geq 0$, let $\Dpolyl{p}$ denote
%$(\pshift\,\enveloping{\cL})^{\otimes p}$.
% We use 
%the symbol $\DpolyL{0}{q}$ to denote the
%space of smooth functions of degree $q$ on $\cM$
%and 
% The symbol $\DpolyL{p}{n}$ denote the degree $n$ subspace of $\Dpolyl{p}$. 
Set $\Dpolyl{} = \bigoplus_{p \geq 0} \Dpolyl{p}$ which is equipped with the bigrading: $\big(\Dpolyl{}\big)^{p,q} = \DpolyL{p}{p+q}.$ 
We have the associated direct-sum total space
\begin{equation}\label{eq:GaredeSud}
\totDpolyL{n}=\bigoplus_{p +q=n}\big(\Dpolyl{}\big)^{p,q}
=\bigoplus_{p=0}^\infty\DpolyL{p}{n} 
\end{equation}
The Hochschild cohomology of the dg Lie algebroid $\cL$,
denoted $\cohomology{\bullet}(\totDpolyL{}, \dQ+\hochschild)$,
is the cohomology of the direct-sum total cochain complex $\big(\totDpolyL{\bullet}, \dQ+\hochschild\big)$ of the double complex $(\Dpolyl{},\hochschild,\dQ)$. 

It is simple to see that the Gerstenhaber bracket
$\gerstenhaber{\argument}{\argument}$ \eqref{eq:Gbraket}
and the cup product $\cup$ \eqref{eq:cup}
descend to the Hochschild cohomology.
When endowed with the Gerstenhaber bracket and the cup product,
$\cohomology{\bullet}(\totDpolyL{},\hochschild+\dQ)$
becomes a Gerstenhaber algebra.

Following~\cite[Definition~2.6]{MR1913813}, by $\jet{\cL}$, we denote the space of $\cL$-jets on $\cM$. 
That is, \[ \jet{\cL}:=\Hom_\cR(\enveloping{\cL},\cR) .\]
It is well known that $\jet{\cL}$
is equipped with a formal groupoid structure
$(m,\alpha,\beta,\Delta,\varepsilon,S)$.
See Section~\ref{sec:FormalGpdJet}.

Since $\enveloping{\cL}$, being equipped with the
differential \eqref{eq:UL},
is a dg module over the \dga $(\cR,Q)$, it induces 
a degree $(+1)$-differential
\begin{equation}\label{eq:jetL}
\cQ:\jet{\cL}\to\jet{\cL}.
\end{equation}
Indeed this differential is compatible with all the structure maps $(m,\alpha,\beta,\Delta,\varepsilon,S,\cQ)$ so that $\jet{\cL}$ is a dg formal groupoid.

There is a natural filtration \cite{MR154906} on the universal enveloping algebra $\enveloping{\cL}$ by the order:
\begin{equation}
\label{Jakarta}
\cR\into \mathcal{U}^{\leqslant 1}(\cL) \into \mathcal{U}^{\leqslant 2}(\cL)
\into \mathcal{U}^{\leqslant 3}(\cL) \into \cdots
\end{equation}
where, for any $p\geq 1$, $\mathcal{U}^{\leqslant p}(\cL)=
\text{Span}_\cR \big\{ X_1X_2 \cdots X_d|X_i\in \sections{\cL}, d\leq p
\big\}$.
This filtration \eqref{Jakarta} induces a descending filtration
on $\jet{\cL}$:
\begin{equation}
\label{Jakarta1}
\jet{\cL}\supset \jet{\cL}_{\geq 1} \supset \jet{\cL}_{\geq 2} \supset
\cdots
\end{equation}
where, for any $p\geq 1$,
\[ \jet{\cL}_{\geq p} =\{\xi\in \jet{\cL}|\duality{\xi}{u}=0, \ \forall u
\in\mathcal{U}^{\leqslant p}(\cL)\} . \]
It is clear that $\jet{\cL}$ is complete with respect to the topology
defined by this filtration, i.e.,
$\jet{\cL}=\varprojlim_p \frac{\jet{\cL}}{ \jet{\cL}_{\geq p}}$.
Note that the filtration \eqref{Jakarta1} conicides with the
adic filtration for the ideal $\jet{\cL}_{\geq 1}\subset \jet{\cL}$.
In the sequel, when we talk about tensor product of $\jet{\cL}$,
we always mean the complete tensor product, denoted by $\cotimes$,
with respect to this filtration.

By $\nshift \jet{\cL}$, we denote the suspended jet space
$\jet{\cL}[1]$, which is naturally an $\cR$-module as well.
For any $p\geq 1$,
$(\nshift\jet{\cL})^{\cotimes p}$ denotes the complete tensor product 
\[ \underset{\text{$p$-factors}}{\underbrace{ \nshift\jet{\cL}
\cotimes\cdots \cotimes \nshift\jet{\cL}}} \]
%$$
%(\nshift\jet{\cL})^{\cotimes p} = \nshift\jet{\cL}\cotimes\, \nshift\jet{\cL}
%\cotimes\cdots \cotimes \nshift\jet{\cL}
%$$
as left $\cR$-modules, of $p$-copies
of the suspended jet space $\nshift\jet{\cL}= \jet{\cL}[1]$,
called \emph{the space of $p$-polyjets} on $\cL$.
%For $p=0$, let $(\nshift\jet{\cL})^{\cotimes 0}=\cR$.
For any $p\geq 1$, by $\Cpolyl{-p}$ we denote
$(\nshift \jet{\cL})^{\cotimes p}$, and for $p=0$,
we let $\Cpolyl{0}=\cR$.
We have a degree $(+1)$ differential
\begin{equation}\label{eq:sjetL}
\cQ: \nshift\jet{\cL} \to \nshift\jet{\cL},
\end{equation}
which makes $\nshift\jet{\cL}$ into a dg module over the \dga $(\cR,Q)$. 
As a consequence, we obtain a degree $(+1)$ differential for each $p\geq 0$
\begin{equation}
\label{eq:QJ}
\cQ:\CpolyL{-p}{n}\to\CpolyL{-p}{n+1}
\end{equation}

Similarly to \eqref{eq:Shift_2_Identification}, we can identify $\Cpolyl{-p}$ with $\nshift^{\otimes p} (\jet{\cL}^{\cotimes p})$
by the isomorphism
\begin{equation}\label{eq:Shift_Cpoly_Identification}
\begin{split}
\nshift^{\otimes p} (\jet{\cL}^{\cotimes p})
&\cong \big(\nshift\, \jet{\cL}\big)^{\cotimes p} = \Cpolyl{-p} \\
\nshift^{\otimes p}(\xi_1 \otimes \cdots\otimes \xi_p) 
& \mapsto (-1)^{\sum_{i=1}^p |\xi_i|(p-i)} \,
\nshift \xi_1 \otimes \cdots\otimes \nshift \xi_p.
\end{split}
\end{equation}
 
%{\tiny
%It is clear that
%\[ (\nshift\jet{\cL})^{\cotimes p}\xto{\cong}\Hom_\cR(\pshift\, \enveloping{\cL}^{\otimes p},\cR) \]
%Thus, for any $p\geq 0$, we have a bilinear map
%\begin{gather*}
%(\nshift \jet{\cL})^{\cotimes p} \times (\pshift\,\enveloping{\cL})^{\otimes p}\to\cR \\
%(D,a_1\otimes\cdots a_k)\mapsto D(a_1\otimes\cdots a_k)
%(\xi_1 \otimes \cdots \otimes \xi_p, u_1\otimes \cdots \otimes u_p)
%\mapsto \pm \duality{\xi_1}{u_1}\cdots \duality{\xi_p}{u_p}
%\end{gather*}
%for all homogeneous $\xi_1,\ldots, \xi_p\in \nshift\jet{\cL}$, \ping{what is sign}
%and $u_1, \ldots , u_p \in \pshift\, \enveloping{\cL}$.
%}

We have the \RinehartConnes operator
\begin{equation}\label{eq:Connes3}
\connes:\CpolyL{-p}{n}\to\CpolyL{-p-1}{n-1}, 
\end{equation}
defined by
%
%%\begin{mathieu}
%%I have inserted the corrected formula valid for the
%%non-graded case in blue below. For the $\ZZ$-graded case,
%%there would be some additional Koszul sign which is too crazy
%%to figure out. In the $\ZZ$-graded case, it would make more
%%sense to write the formula on the level of Grothendieck-flat
%%elements of the $(n+1)$-th tensor power of the jet space.
%%The Koszul signs involved in that formula would be manageable.
%%\end{mathieu}
%%\red{Double check with Mathieu}
%
\begin{multline*}
\connes\big(\nshift^{\otimes p}(\xi_1\otimes\xi_2\otimes\cdots\otimes\xi_p) \big) \\
= \sum_{k=0}^p \pm \, \nshift^{\otimes p +1} \Big\{ 
\xi_k^{(2)}\otimes\cdots\otimes\xi_n^{(2)}\otimes
\antipode\big(\xi_1^{(1)}*\cdots*\xi_p^{(1)}\big)
\otimes\xi_1^{(2)}\otimes\cdots\otimes\xi_{k-1}^{(2)} \\
+\alpha\circ\counit(\xi_k^{(2)})\otimes
\xi_{k+1}^{(2)}\otimes\cdots\otimes\xi_n^{(2)}\otimes
\antipode\big(\xi_1^{(1)}*\cdots*\xi_p^{(1)}\big)
\otimes\xi_1^{(2)}\otimes\cdots\otimes\xi_{k-1}^{(2)}
\Big\}
\end{multline*}
% where $\xi_0^{(2)}=\antipode\big(\xi_1^{(1)}*\cdots*\xi_p^{(1)}\big)$
%\begin{multline*}
%\connes(\xi_1\otimes\xi_2\otimes\cdots\otimes\xi_p)
%\\ = \sum_{k=0}^p \red{\pm}\, \Big\{
%\xi_k^{(2)}\otimes\xi_{k+1}^{(2)}\otimes\cdots\otimes\xi_p^{(2)}
%\otimes\antipode(\xi_1^{(1)}*\cdots*\xi_p^{(1)})\otimes
%\xi_1^{(2)}\otimes\cdots\otimes\xi_{k-1}^{(2)}
%\\ \red{\pm}\, \xi_{k+1}^{(2)}\otimes\cdots\otimes\xi_p^{(2)}
%\otimes\antipode(\xi_1^{(1)}*\cdots*\xi_p^{(1)})\otimes
%\xi_1^{(2)}\otimes\cdots\otimes\xi_{k-1}^{(2)}\otimes\alpha\circ\counit(\xi_k^{(2)}) \Big\}
%\end{multline*}
for all $\xi_1,\xi_2,\dots,\xi_p\in\jet{\cL}$.
Here, to simplify notation, $\xi^{(2)}_{-1}$ is understood to mean $\xi^{(2)}_p$
while $\xi^{(2)}_0=\xi^{(2)}_{p+1}$ is understood to mean
$\antipode(\xi_1^{(1)}*\xi_2^{(1)}*\cdots*\xi_p^{(1)})$.
Recall that $\xi^{(1)}$ and $\xi^{(2)}$ are components of
the coproduct $\Delta(\xi)=\xi^{(1)}\otimes\xi^{(2)}$,
for all $\xi\in\jet{\cL}$,
$S: \jet{\cL}\to \jet{\cL}$ is the antipode, and $*$ denotes the
multiplication on $\jet{\cL}$. See Appendix \ref{sec:FormalGpdJet}. 

For any $\iD\in\Dpolyl{d}$, the contraction operator
\begin{equation}\label{eq:ILiealgebroid}
\iI_\iD:\Cpolyl{-p}\to\Cpolyl{-p+d}, 
%\quad\quad\forall\iD\in\DpolyL{d}{}
\end{equation}
is defined by
\[ \iI_{\nshift \big((\pshift)^{\otimes d}(u_1 \otimes \cdots \otimes u_d)\big)}
\big(\nshift^{\otimes p}(\xi_1\otimes\xi_2\otimes\cdots\otimes\xi_p) \big)
= \pm\, \nshift^{\otimes p-d}\big(\alpha(\pairing{\xi_1}{u_1}\cdots\pairing{\xi_p}{u_p})*\xi_{p+1}
\otimes\xi_{p+2}\otimes\cdots\otimes\xi_{p} \big)\]
for all $u_1,u_2,\dots,u_d\in\uea(\algebroid)$
and $\xi_1,\xi_2,\dots,\xi_p\in\jet{\cL}$.

Now we introduce the Lie derivative
\begin{equation}\label{eq:L3}
\iL_\iD:\Cpolyl{-p}\to\Cpolyl{-p-1+d} 
\end{equation}

Recall from \cite{paper-zero} that there are two canonical representations
$\dmb$ and $\gro$ of the Lie algebroid $\cL$ on its jet-space $\jet{\cL}$:
\[ \pairing{\dmb_X\xi}{v}=\pairing{\xi}{v\cdot X}
\qquad\text{and}\qquad
\pairing{\gro_X\xi}{v}=\rho_X\big(\pairing{\xi}{v}\big)-\pairing{\xi}{X\cdot v} \]
for all $X\in\Gamma(\cL)$, $\xi\in\jet{\cL}$, and $v\in\mathcal{U}(\cL)$.
In other words, the jet-space $\jet{\cL}$ admits to distinct natural structures
of module over the algebra $\mathcal{U}(\cL)$:
\[ \pairing{\dmb_u\xi}{v}=\pairing{\xi}{v\cdot u}
\qquad\text{and}\qquad
\pairing{\gro_u\xi}{v}=\rho_{u_+}\big(\pairing{\xi}{u_-\cdot v}\big) \]
for all $u,v\in\Gamma(\cL)$ and $\xi\in\jet{\cL}$.
The latter representation, which is known as the \emph{Grothendieck connection},
satisfies $\gro_X\circ\alpha=0$ for all $X\in\Gamma(\cL)$ and therefore extends
to a representation $\org$ of th Lie algebroid $\cL$ on the tensor product
\[ \jet{\cL}\cotimes\underset{\text{$p$ factors}}{\underbrace{
\nshift\jet{\cL}\cotimes\cdots\cotimes\nshift\jet{\cL}}} \]
in the obvious way: 
\[ \org_X(\xi_0\otimes\nshift\xi_1\otimes\cdots\otimes\nshift\xi_p)
= (\gro_X\xi_0)\otimes\nshift\xi_1\otimes\cdots\otimes\nshift\xi_p
+ \sum_{k=1}^p {\pm} \xi_0\otimes\nshift\xi_1\otimes\cdots\otimes
\nshift(\gro_X\xi_k)\otimes\cdots\nshift\xi_p \]
for all $X\in\Gamma(\cL)$ and $\xi_0,\xi_1,\dots,\xi_p\in\jet{\cL}$.

For $p\geq 1$, we set
\[ \mathrm{C}_p(\algebroid)
=\underset{\text{$p$ factors}}{\underbrace{
\nshift\jet{\cL}\cotimes\cdots\cotimes
\nshift\jet{\cL}}} \]
and
\[ \mathrm{K}_p(\algebroid)
=\Big(\jet{\cL}\cotimes\underset{\text{$p$ factors}}{\underbrace{
\nshift\jet{\cL}
\cotimes\cdots\cotimes\nshift\jet{\cL}}}
\Big)^{\org\text{-flat}} ,\]
the subspace of $\org\text{-flat}$ elements.
For $p=0$, we set
$\mathrm{C}_0(\algebroid)=\mathcal{R}$
and $\mathrm{K}_0(\algebroid)=\big(
\jet{\cL}\big)^{\org\text{-flat}}$.
For $p\leq -1$, we set $\mathrm{C}_p(\algebroid)=0$ and $\mathrm{K}_p(\algebroid)=0$.

Recall from \cite{paper-zero} that there exists
a pair of mutually inverse algebra isomorphisms
\[\begin{tikzcd}[column sep=huge]
{\mathrm{C}_p(\cL)} & {\mathrm{K}_p(\cL)}
\arrow["{\pi^*}"', shift right=2, from=1-1, to=1-2]
\arrow["\cong"{description}, draw=none, from=1-1, to=1-2]
\arrow["{\iota^*}"', shift right=2, from=1-2, to=1-1]
\end{tikzcd}\]
The morphism $\pi^*$ is defined by
\[ \pi^*(\nshift\xi_1\otimes\nshift\xi_2
\otimes\cdots\otimes\nshift\xi_p)
=\sum_{(\xi_1)}\sum_{(\xi_2)}\cdots
\sum_{(\xi_p)}S\big(\xi_1^{(1)}
*\xi_2^{(1)}*\cdots*\xi_p^{(1)}\big)
\otimes\nshift\xi_1^{(2)}
\otimes\nshift\xi_2^{(2)}\otimes
\cdots\otimes\nshift\xi_p^{(2)} \]
while $\iota^*$ is the morphism 
\[ \iota^*(\xi_0\otimes\nshift\xi_1\otimes
\nshift\xi_2\otimes\cdots\otimes\nshift\xi_p)
=\big(\alpha\circ\counit(\xi_0)*\nshift\xi_1
\big)\otimes\nshift\xi_2\otimes\cdots
\otimes\nshift\xi_p \]
or more precisely its restriction
to $\org$-flat
elements.

The Lie derivative \eqref{eq:L3}
%\begin{equation}
%\label{eq:L3}
%\iL_\iD:\CpolyL{-p}{}\to\CpolyL{-p+d-1}{} 
%\end{equation}
%is defined by the Cartan formula:
%\[ \iL_\iD = (-1)^{|\iD|-1} \commutator{B}{\iI_\iD} .\]
is the composition
\[ \iL_\iD=\iota^*\circ\tilde{L}_\iD\circ\pi^* ,\]
where $\tilde{L}$ denotes the restriction to
$\org$-flat elements of the operator defined by
\begin{multline*}
\tilde{L}_{\pshift u_1\otimes\pshift u_2\otimes
\cdots\otimes\pshift u_d}(\xi_0\otimes\nshift\xi_1\otimes
\nshift\xi_2\otimes\cdots\otimes\nshift\xi_p)
\\
=\sum_{j=d}^p(-1)^{(d-1)(1+j-d)}{\pm}
\xi_0\otimes\nshift\xi_1\otimes\cdots\otimes\nshift\xi_{j-d}\otimes
\nshift\big((\dmb_{u_1}\xi_{j-d+1})*\cdots*(\dmb_{u_d}\xi_{j})\big)
\otimes\nshift\xi_{j+1}\otimes\cdots\otimes\nshift\xi_p
\\
+\sum_{j=1}^d(-1)^{p(p-j+2)}{\pm}
\big((\dmb_{u_1}\xi_{p-j+2})*\cdots*(\dmb_{u_{j-1}}\xi_{p})*
(\dmb_{u_j}\xi_{0})*\cdots*(\dmb_{u_d}\xi_{d-j})\big)
\otimes\nshift\xi_{d-j+1}\otimes\cdots\otimes\nshift\xi_{p-j+1}
\end{multline*}
for all $u_1,u_2,\dots,u_d\in\mathcal{U}(\cL)$
and $\xi_0,\xi_1,\xi_2,\dots,\xi_p\in\jet{\cL}$.

The Hochschild boundary differential is defined by
\begin{equation}\label{eq:hochschildb}
\hochschildb = \iL_m : \CpolyL{-p}{n}\to \CpolyL{-p+1}{n+1},
\end{equation}
where $m$ is given by Equation~\eqref{eq:m}.
It can be expressed explicitly by
\[ \begin{aligned}
\hochschildb \big(\nshift^{\otimes p}(\xi_1\otimes\xi_2\otimes\cdots\otimes\xi_p) \big)
= {} & {\pm}\, \nshift^{\otimes p-1}\big(\big((\alpha\circ\counit)
(\xi_1)*\xi_2\big)\otimes\xi_3\otimes\cdots\otimes\xi_p \big)
\\ & {\pm}\, \nshift^{\otimes p-1}\big((\xi_1*\xi_2)\otimes\xi_3\otimes\cdots\otimes\xi_p\big)
\\ & {\pm}\, \nshift^{\otimes p-1}\big(\xi_1\otimes(\xi_2*\xi_3)\otimes\xi_4\otimes
\cdots\otimes\xi_p \big)
\\ & \qquad\vdots
\\ & {\pm}\, \nshift^{\otimes p-1}
\big(\xi_1\otimes\xi_2\otimes\cdots\otimes\xi_{p-2}\otimes(\xi_{n-1}*\xi_p)\big)
\\ & {\pm}\, \nshift^{\otimes p-1}\big(\xi_1\otimes\xi_2\otimes\cdots\otimes\xi_{p-2}\otimes
\big(\xi_{p-1}*(\alpha\circ\counit)(\xi_p)\big)\big)
\end{aligned} \]
for all $\xi_1,\xi_2,\dots,\xi_p\in\jet{\cL}$.

%\ping{shift is needed}

% We use
% % the symbol $\CpolyL{0}{q}$ to denote the space of smooth functions of degree $q$ on $\cM$ and 
% the symbol $\CpolyL{-p}{n}$ to denote the degree $n$ subspace of $\Cpolyl{-p}$. 
As before, we set $\Cpolyl{} = \bigoplus_{p \geq 0} \Cpolyl{-p}$ which is equipped with the bigrading: $\big(\Cpolyl{}\big)^{-p,q}=\CpolyL{-p}{-p+q} .$ 
This bigraded space gives us the direct-product total space
\begin{equation}\label{eq:GaredeOuest}
\totCpolyL{n}=\prod_{-p+q = n} \big(\Cpolyl{}\big)^{-p,q}
=\prod_{p = 0}^\infty \CpolyL{-p}{n} .
\end{equation}

%It is known that there is a Tamarkin--Tsygan calculus structure
%on the cohomologies of the pair $\big(\totDpolyL{\bullet},\totCpolyL{\bullet}\big)$.
%This structure is generated by the formal groupoid operations on $\jet{\cL}$
%--- see~\cite{paper-zero} --- and reduces to the homotopy Cartan calculus
%described in Section~\ref{sec:HtpCalDGMfd} in the case $\cL=T_\cM$.

The differential in \eqref{eq:jetL} is compatible with all the structure maps of the formal groupoid
$\jet{\cL}$ so that $\jet{\cL}$ is a dg formal groupoid over $(\cR,Q)$ (see Appendix~\ref{sec:FormalGpdJet}
for the structure maps of the formal groupoid $\jet{\cL}$). \footnote{When $\cL$ is
the Lie algebroid of a dg Lie groupoid $\cG\toto\cM$,
$\jet{\cL}$ is the space of formal functions along the unit $\cM$.}
As a consequence, we have $[\cQ,\hochschildb]=0$, and therefore $\big(\dQ+\hochschildb\big)^2=0$, where $b$ is the Hochschild homology differential \eqref{eq:hochschildb}.

The Hochschild homology of the dg Lie algebroid $\cL$,
denoted $\cohomology{}\big(\totCpolyL{\bullet}, \dQ+\hochschildb\big)$,
is the cohomology of the direct-product total cochain complex $ \big(\totCpolyL{\bullet}, \dQ+\hochschildb\big)$ of the double complex $(\Cpolyl{}, \hochschildb, \dQ)$.

It is known that the operations $\iI$ \eqref{eq:ILiealgebroid},
$\liederivative{}$ \eqref{eq:L3} and $B$ \eqref{eq:Connes3}
satisfy the Tamarkin--Tsygan calculus operations 
Definition~\ref{def:calc} (ii)--(v), up to homotopy
--- see~\cite{paper-zero} --- and 
%on the cohomologies of the pair $\big(\totDpolyL{\bullet},\totCpolyL{\bullet}\big)$.
%This structure is generated by the formal groupoid operations on $\jet{\cL}$
%--- see~\cite{paper-zero} --- and reduces to the homotopy Cartan calculus
%described in Section~\ref{sec:HtpCalDGMfd} in the case $\cL=T_\cM$.
%
%It is simple to see that all the operations $\iI$ \eqref{eq:ILiealgebroid},
%$\liederivative{}$ \eqref{eq:L3} and $B$ \eqref{eq:Connes3} 
descend to their respective Hochschild (co)homologies so that we have

\begin{proposition}\label{prop:CalH_DGLieAbd}
Let $\cL$ be a dg Lie algebroid. The pair
\[ \Theta^{\bullet}=\cohomology{}\big( \totDpolyL{\bullet}, \dQ+\hochschild\big)
\qquad\text{and}\qquad
\Xi_{\bullet}= \cohomology{}\big(\totCpolyL{\bullet}, \dQ+\hochschildb \big) \]
admits a Tamarkin--Tsygan calculus
structure, where the differential $d$,
%\eqref{eq:BM}
the action operators $\iI$ and $\iL$ are defined,
respectively by \eqref{eq:Connes3}, \eqref{eq:ILiealgebroid}
and~\eqref{eq:L3}.
\end{proposition}

Such a calculus is denoted $\calculus_H(\cL,\cQ)$.
In the case $\cL = T_\cM$, the calculus $\calculus_H(\cL,\cQ)$ reduces to 
noncommutative calculus $\calculus_H (\cM,\cQ)$ of the dg manifold
$(\cM,Q)$ in Section~\ref{sec:HtpCalDGMfd}.

%FXXXXXXXXXXXXXXX STOP HERE
\section{Fedosov dg Lie algebroids of graded manifolds}

\subsection{Fedosov dg manifolds}

In this subsection, we recall some basic results and
notations regarding Fedosov formal dg manifolds
associated to a graded manifold $\cM$, which will
be needed in future discussions. For details,
see~\cite{MR3910470}.
 
\subsubsection{Connections on graded manifold}

Let $\cE\to\cM$ be a graded vector bundle.
% object in the category of graded manifolds.
A linear connection on $\cE\to\cM$ is a $\KK$-linear map $\nabla:\sections{T_\cM}\otimes\sections{\cE}\to\sections{\cE}$ of degree~$0$, satisfying
\[ \nabla_{f\cdot X}S=f\cdot\nabla_X S
\qquad\text{and}\qquad
\nabla_X (f\cdot S)=X(f)\cdot S+(-1)^{\degree{X}\degree{f}}f\cdot\nabla_X S ,\]
for all homogeneous $f\in C^\infty(\cM)$, $X\in\sections{T_\cM}$, and $S\in\sections{\cE}$.
Its covariant differential is the map $d^\nabla:\OO^\bullet(\cM,\cE)\to\OO^{\bullet+1}(\cM,\cE)$ of degree~$(+1)$ satisfying 
\[ \interior{X}(d^\nabla S)=\nabla_X S
\qquad\text{and}\qquad
d^\nabla(\alpha\wedge\beta)=d\alpha\wedge\beta+(-1)^{\degree{\alpha}}\alpha\wedge
d^\nabla\beta ,\]
for all $X\in\sections{T_\cM}$, $S\in\sections{\cE}$,
$\alpha\in\OO(\cM)$, and $\beta\in\OO(\cM,\cE)$.
Its curvature is the $2$-form $R^\nabla\in\OO^2\big(\cM,\End(\cE)\big)$
defined by
\[ R^\nabla(X,Y)=(-1)^{\degree{Y}-1}\big\{\nabla_X\nabla_Y
-(-1)^{\degree{X}\degree{Y}}\nabla_Y\nabla_X-\nabla_{\lie{X}{Y}}\big\} ,\]
for all homogeneous $X,Y\in\sections{T_\cM}$.
For all $\omega\in\OO(\cM,\cE)$, we have $(d^\nabla)^2\omega=R^\nabla\wedge\omega$.

The connection $\nabla$ determines an induced linear connection
(also denoted $\nabla$ by abuse of notation) on the dual vector bundle $\cE^\vee\to\cM$
through the relation
\[ X(\duality{\zeta}{S})=\duality{\nabla_X\zeta}{S}
+(-1)^{\degree{X}\degree{\zeta}}\duality{\zeta}{\nabla_X S} ,\]
%\color{red}
%\[ X(\duality{S}{\zeta})=\duality{\nabla_X S}{\zeta}
%+(-1)^{\degree{X}\degree{S}}\duality{S}{\nabla_X\zeta} ,\]
%\color{black}
for all $X\in\sections{T_\cM}$, $\zeta\in\sections{\cE^\vee}$ of
homogeneous degrees and $S\in\sections{\cE}$.

If $\cE=T_\cM$, we also have the notion of torsion,
%$T^\nabla:\sections{T_\cM}\times\sections{T_\cM}\to\sections{T_\cM}$
%of the connection $\nabla$, which is
a tensor field of type $(1,2)$ on $T_\cM$ defined by
\[ T^\nabla(X,Y)=\nabla_X Y-(-1)^{\degree{X}\degree{Y}}\nabla_Y X-\lie{X}{Y} ,\]
for all homogeneous $X,Y\in\sections{T_\cM}$.

\subsubsection{Fedosov formal manifold}\label{sec:FedosovMfd}

Throughout the present paper, the graded manifold
$\cN=T_\cM[1]\oplus \Tformal\cM$ (with support $M$) will play a central role.
%It has the feature that local functions involve formal power series in
%some (but not all) coordinates. Such a graded manifold will be referred to as
%a \emph{formal graded manifold}. See~\cite[Section~1.2]{MR4727081}.

Let $\cM$ be a graded manifold. Denote by $\Tformal\cM$ the formal neighborhood
of the zero section in the tangent bundle.
It is the graded manifold with the support $M$ and the structure sheaf 
\[ \fA_{\Tformal\cM} = \varprojlim_p \frac{\fA_{\tilde T_\cM}}{I^p} ,\]
where $I$ is the sheaf of ideals generated by the fiberwise linear functions on $T_\cM$.
Here $\fA_{\tilde T_\cM}$ is a sheaf over $M$, given by the sections of $ST_\cM\dual$.
In other words, we have $\fA_{\Tformal\cM}(U) = \sections{\cM|_U, \hat{S} T_{\cM|_U}\dual}$
for any open subset $U \subset M$.

The shifted tangent bundle $T_\cM[1]$ is a graded vector bundle over $\cM$.
Its total space, also denoted by $T_\cM[1]$, is the graded manifold with support $M$
and structure sheaf $\fA_{T_\cM[1]}$ induced by sections of the graded vector bundle
$\hat S(T_\cM[1])\dual\to\cM$. 

The \emph{Fedosov manifold} is the graded manifold
$\cN=T_\cM[1]\oplus\Tformal\cM$ with support $M$, whose structure sheaf
$\fA_\cN$ is given by the sections of vector bundle
$\hat S(T_\cM[1] \oplus \Tformal \cM)\dual \to \cM$. 
%It also can be considered as the $I$-adic completion
%of $\hat S(T_\cM[1])\dual\cotimes S T_\cM\dual$ with respect to the sheaf of ideals
%$I =\hat S(T_\cM[1])\dual \cotimes S^{\geq 1} T_\cM\dual$. 
%In particular,
%\ping{ check this if correct}
%
%\begin{equation}\label{eq:DefFedosovMfd}
%C^\infty(\cN) = \sections{S(T_\cM[1])\dual \hotimes \hat S T_\cM\dual}
%=\varprojlim_p \sections{S(T_\cM[1])\dual \otimes S^{\leq p} T_\cM\dual}.
%\end{equation}
Abusing notations, we will make frequent use of the heuristic notation
$\Omega\big(\cM,\hat{S}(T^\vee_\cM)\big)$
to denote the algebra $C^\infty(\cN)$.

\begin{remark}\label{rmk:CoordFedosovMfd}
A local chat $(x_1,\cdots, x_{m+r})$ on $\cM$ as in Section~\ref{corniche}
induces local functions $(\xi_1,\cdots, \xi_{m+r})$ on $T_\cM[1]$
and $(y_1,\cdots, y_{m+r})$ on $\Tformal\cM$, respectively.
Here $\xi_k$ and $y_k$ stand for the linear coordinate function
on the fibers of the vector bundle $T_\cM[1]$ and $\Tformal\cM$, respectively.
% measuring the component of a tangent vector in the direction of $\frac{\partial}{\partial x_k}$. 
The local functions $(x_1,\cdots, x_{m+r}, \xi_1,\cdots, \xi_{m+r}, y_1, \cdots, y_{m+r})$
form a local chat on $\cN = T_\cM[1]\oplus \Tformal\cM$ 
with degrees $|\xi_k|=|x_k|+1$ and $|y_k|=|x_k|$, respectively. 
If $(x_1, \cdots, x_m)$ are smooth coordinates on an open subset $U$
in the support $M$, and $(x_{m+1},\cdots,x_{m+r})$ are
virtual homogeneous coordinate functions on $\cM$,
then a function on $\cN$ locally can be identified with an element in
\[ C^\infty(U)\llbracket x_{m+1},\cdots,x_{m+r},\xi_1,\cdots,\xi_{m+r},
y_1,\cdots,y_{m+r}\rrbracket .\]
%\ping{$C^\infty (x_1,\cdots,x_m)$?}
\end{remark}

\subsubsection{Fedosov dg formal manifolds}
\label{sec:FedosovConnection}

Let $\cM$ be a graded manifold and let $\cD(\cM)$
denote its algebra of differential
operators. Given an affine connection $\nabla$ on $\cM$
(i.e.\ a linear connection on $T_\cM$),
there exists a unique well-defined
morphism of left $C^\infty(\cM)$-modules
\begin{equation}\label{eq:pbw}
\pbw^\nabla:\sections{S(T_\cM)}\to\cD(\cM)
\end{equation}
satisfying the relations
\begin{gather*}
\pbw^\nabla(f)=f, \quad\forall f\in C^\infty(\cM); \\
\pbw^\nabla(X)=X, \quad\forall X\in\sections{T_\cM};
\end{gather*}
and, for all $n\in\NN$ and any homogeneous $X_0,X_1,\dots,X_n\in\sections{T_\cM}$,
\[ \pbw^\nabla(X_0\odot\cdots\odot X_n)
=\frac{1}{n+1}\sum_{k=0}^{n}\epsilon_k\Big\{
X_k\cdot\pbw^\nabla(X^{\{k\}})-\pbw^\nabla\big(\nabla_{X_k}(X^{\{k\}})\big)
\Big\} ,\]
where $\epsilon_k=(-1)^{\degree{X_k}(\degree{X_0}+\cdots+\degree{X_{k-1}})}$
and $X^{\{k\}}=X_0\odot\cdots\odot X_{k-1}\odot X_{k+1}\odot\cdots\odot X_n$.
See~\cite{MR3910470}. 
When $\cM$ is an ordinary manifold $M$,
this map $\pbw^\nabla$ is the fiberwise $\infty$-order jet of
the fiber bundle map
$\exp^\nabla:T_M\to M\times M$ over $M$ induced by 
the exponential map of the connection $\nabla$.
Therefore, $\pbw^\nabla$ is called the \emph{`formal exponential map'}
arising form the affine connection $\nabla$.

The order of the differential operators and polynomials
determine natural filtrations on the associative algebras $\cD(\cM)$
and $\sections{S(T_\cM)}$, respectively. 
Besides being associative algebras, both
$\sections{S(T_\cM)}$ and $\cD(\cM)$
are naturally filtered coalgebras over
$\cR$ with deconcatenation as comultiplication:
\begin{multline}
\label{DOODLE}
\Delta(X_1\cdots X_n)=1\otimes(X_1\cdots X_n)
+\sum_{p=1}^{n-1}\sum_{\sigma\in\shuffle{p}{n-p}}
\epsilon^{\sigma}_{X_1,\dots,X_n}
(X_{\sigma_1}\cdots X_{\sigma(p)})\otimes
(X_{\sigma_{p+1}}\cdots X_{\sigma(n)}) \\
+(X_1\cdots X_n)\otimes 1
.\end{multline} 
The symbol $\epsilon^{\sigma}_{X_1,\dots,X_n}$ appearing in the terms of the sum
denotes the Koszul sign of the permutation $\sigma$ of the order in which the homogeneous elements
$X_1,\dots,X_n$ of $\sections{T_\cM}$ appear in that term. More precisely, we have
$X_{\sigma(1)} \odot \cdots \odot X_{\sigma(n)}=\epsilon^{\sigma}_{X_1,\dots,X_n} X_1
\odot \cdots \odot X_n$ in $\sections{ST_\cM}$.

It is known that the formal exponential map
$\pbw^\nabla:\sections{S(T_\cM)}\to\cD(\cM)$
is an isomorphism of filtered (left) coalgebras over $\cR$ \cite[Proposition~4.2]{MR3910470}.

We now introduce another linear connection $\nabla^\lightning$
on the vector bundle $S(T_\cM)\to\cM$ by
\begin{equation}
\label{eq:cntb}
\nabla^\lightning_X S = (\pbw^\nabla)^{-1}\big(X\cdot\pbw^\nabla(S)\big) ,
\end{equation}
for all $X\in\sections{T_\cM}$ and $S\in\sections{S(T_\cM)}$.
Thus each vector field $X\in\sections{\cM}$ determines
a coderivation $\nabla^\lightning_X$ of the coalgebra $\sections{S(T_\cM)}$
and, consequently, a derivation, also denoted $\nabla^\lightning_X$
by abuse of notation, of the dual algebra $\sections{\hat{S}(T_\cM^\vee)}$.
Since the formal exponential map $\pbw^\nabla:\sections{S^{\leq n}(T_\cM)}\to\cD^{\leq n}(\cM)$ respects the filtrations, the coderivation $\nabla^\lightning_X$
maps $\sections{S^{\leq n}(T_\cM)}$ to $\sections{S^{\leq n+1}(T_\cM)}$ and, consequently,
its dual $\nabla^\lightning_X$ is a derivation mapping
$\sections{\hat{S}^{\geq n}(T_\cM^\vee)}$
to $\sections{\hat{S}^{\geq n-1}(T_\cM^\vee)}$.

Unlike $\nabla$, this new linear connection $\nabla^\lightning$ on $S(T_\cM)$ is clearly flat.
Its induced linear connection $\nabla^\lightning$
on the dual bundle $\hat{S}(T_\cM^\vee)$ is also flat,
and is called \emph{Fedosov connection}.
We use the symbol $\fedosov$ to denote its corresponding covariant differential
\[ \fedosov:\OO^{\bullet}\big(\cM,\hat{S}(T_\cM^\vee)\big)
\to\OO^{\bullet+1}\big(\cM,\hat{S}(T_\cM^\vee)\big) .\]
Thus, we have $(\fedosov)^2=0$. Since $\nabla^\lightning_X:
\sections{\hat{S}^{}(T_\cM^\vee)}
\to\sections{\hat{S}^{}(T_\cM^\vee)}$ is a derivation,
it follows that $\fedosov$ is a degree $(+1)$ derivation of
$\OO^{\bullet}\big(\cM,\hat{S}(T_\cM^\vee)\big)$, and
hence is a homological vector field on the Fedosov
manifold $\cN = T_\cM[1]\oplus \Tformal\cM$.
Thus we have

\begin{proposition}[\cite{MR3910470}]
Let $\cM$ be any graded manifold and $\nabla$ an affine connection on $\cM$.
Then $\big(T_\cM[1]\oplus \Tformal\cM, \fedosov\big)$ is a dg formal 
manifold.
\end{proposition}

We will refer to $\big(T_\cM[1]\oplus \Tformal\cM, \fedosov\big)$
as a Fedosov dg formal manifold associated to $\cM$.

According to \cite[Theorem~8.1]{MR3910470}, the cochain complex $(C^\infty(\cN),\fedosov)$
is quasi-isomorphic to the cochain complex $(C^\infty(\cM),0)$. 
We can interpret this quasi-isomorphism as saying that the Fedosov dg manifold
$(\cN,\fedosov)$ is quasi-isomorphic to $(\cM,0)$, the dg manifold obtained
by endowing the graded manifold $\cM$ with the trivial homological vector field.

\subsection{Fedosov connection in local coordinates}
\label{sect:2.2}
The affine connection $\nabla$ on $T_\cM$ determines a connection $\nabla$
on the vector bundle $S(T_\cM)\to\cM$,
also denoted $\nabla$ by abuse of notation,
through the relation
\[ \nabla_X(X_1\odot\cdots\odot X_n)
=\sum_{k=1}^n (-1)^{\degree{X}(\degree{X_1}+\cdots+\degree{X_{k-1}})}
X_1\odot\cdots X_{k-1}\odot\nabla_X X_k\odot X_{k+1}\odot\cdots\odot X_n .\]
For all $X\in\XX(\cM)$, the operator $\nabla_X$ is a coderivation
of the coalgebra $\sections{S(T_\cM)}$.
%Abusing notations,
We use the same symbol $\nabla$ to denote the induced connection
on the dual vector bundle $\hat{S}(T_\cM^\vee)\to\cM$.
In local coordinates, the derivation $\nabla_{\frac{\partial}{\partial x_k}}$
of the algebra $\sections{\hat{S}(T_\cM^\vee)}$ reads
\begin{equation}\label{Yaoundé}
\nabla_{\frac{\partial}{\partial x_k}} = \frac{\partial}{\partial x_k}\otimes 1
- \sum_{l=1}^{m+r}\sum_{j=1}^{m+r} {(-1)^{|y_l|+ |y_l||y_j|}}
\Gamma_{k,l}^j \otimes y_l\frac{\partial}{\partial y_j}
,\end{equation}
where $\Gamma_{k,l}^j\in C^\infty(\cM)$ are the Christoffel symbols
of the connection $\nabla$ on $T_\cM$: $ \nabla_{\frac{\partial}{\partial x_k}}\frac{\partial}{\partial x_l}
=\sum_{j=1}^{m+r} \Gamma_{k,l}^j\frac{\partial}{\partial x_j} .$ 
The corresponding covariant differential $d^{\nabla}: \OO^\bullet\big(\cM,\hat{S}(T_\cM^\vee)\big)
\to \OO^{\bullet+1}\big(\cM,\hat{S}(T_\cM^\vee)\big)$ is the derivation
\begin{equation}\label{Lagos}
d^{\nabla}=\sum_{k=1}^{m+r}\xi_k\cdot\nabla_{\frac{\partial}{\partial x_k}}
\end{equation}
of the algebra $\OO^\bullet\big(\cM,\hat{S}(T_\cM^\vee)\big)$.

%We use the symbol $d^\nabla$ to denote the covariant differential of the induced connection
%on the dual vector bundle $\hat{S}(T_\cM^\vee)$.

%Let $\cM$ be a finite-dimensional $\ZZ$-graded manifold with support $M$.
Next, we introduce the operators $\delta:\OO^p\big(\cM,S^q(T^\vee_\cM)\big)
\to\OO^{p+1}\big(\cM,S^{q-1}(T^\vee_\cM)\big)$ and $\eta:\OO^p\big(\cM,S^q(T^\vee_\cM)\big)
\to\OO^{p-1}\big(\cM,S^{q+1}(T^\vee_\cM)\big).$ 
They are locally defined by the relations
\begin{equation}\label{Cotonou}
\delta=\sum_{k=1}^{m+r} \xi_k\otimes\frac{\partial}{\partial y_k}
\qquad\text{and}\qquad
\eta=\sum_{k=1}^{m+r} \frac{\partial}{\partial\xi_k}\otimes y_k
.\end{equation}
Here the local coordinate functions
$x_1,\cdots,x_{m+r},\xi_1,\cdots,\xi_{m+r},y_1,\cdots,y_{m+r}$
on $\cN=T_{\cM}[1]\oplus T_{\cM}$
associated with a choice of local coordinates $x_1,\cdots,x_{m+r}$
on $\cM$ are as in Remark~\ref{rmk:CoordFedosovMfd}.
More precisely, for all homogeneous $\omega\in\OO^p(\cM)$
and $P\in\sections{S^q T^\vee_\cM}$, we have
\[ \delta(\omega\otimes P) =\sum_{k=1}^{m+r}
(-1)^{\degree{\frac{\partial}{\partial y_k}}\degree{\omega}}
(\xi_k\cdot\omega)\otimes\frac{\partial}{\partial y_k}(P) \]
and
\[ \eta(\omega\otimes P) =\sum_{k=1}^{m+r}
(-1)^{\degree{y_k}\degree{\omega}}\frac{\partial}{\partial\xi_k}(\omega)
\otimes (y_k\cdot P) .\]
Both operators $\delta$ and $\eta$ are well defined,
i.e.\ independent of the choice of local coordinates on $\cM$, and are 
derivations of the graded algebra $\OO^\bullet\big(\cM,S(T^\vee_\cM)\big)$.
The operator $\delta$ is of degree~$(+1)$ while the operator $\eta$ is
of degree~$(-1)$.
Obviously, we have $\delta\circ\delta=0$ and $\eta\circ\eta=0$.
Indeed $\delta$ is the fiberwise Koszul operator.

We define an operator $h$ on $\OO^\bullet\big(\cM,\hat{S}(T^\vee_\cM)\big)$
by declaring that $h(x)=0$ if $x\in\OO^0\big(\cM,S^0(T^\vee_\cM)\big)$ and
\begin{equation}
h(x)= - \frac{1}{p+q}\eta(x)
\end{equation}
for any $x\in\OO^p\big(\cM,S^q(T^\vee_\cM)\big)$ with $p\geq 1$ or $q\geq 1$.
The operator $h$ satisfies $h\circ h=0$ but, unlike $\eta$, it is \emph{not}
a derivation of the algebra $C^\infty(\cN)\cong 
\OO^\bullet\big(\cM,\hat{S}(T^\vee_\cM)\big)$.

%Furthermore, we have a homotopy operator 
%\[ h:\OO^p\big(\cM,S^q(T^\vee_\cM)\big)
%\to\OO^{p-1}\big(\cM,S^{q+1}(T^\vee_\cM)\big) \]
%defined by 
%$$
%h(\omega\otimes P) = \dfrac{1}{p+q} \eta(\omega\otimes P),
%$$
%for all $\omega\in\OO^p(\cM)$ and $P\in\sections{S^q T^\vee_\cM}$. 

Since $\delta$ and $h$, respectively, raise and lower the filtration
\[ \cdots\subset\Omega\big(\cM,S^{\geq q+1}(T^\vee_\cM)\big)\subset
\Omega\big(\cM,S^{\geq q}(T^\vee_\cM)\big)\subset
\Omega\big(\cM,S^{\geq q-1}(T^\vee_\cM)\big)\subset\cdots \]
by one rung, they can both be extended to the completed algebra
\[ \OO^\bullet\big(\cM,\hat{S}(T^\vee_\cM)\big) 
=\varprojlim_{q}\frac{\OO^\bullet\big(\cM,S(T^\vee_\cM)\big)}
{\OO^\bullet\big(\cM,S^{\geq q}(T^\vee_\cM)\big)} .\]

%%%%%%%%%%%%%%%%%%%%%%%%%%

\begin{remark}\label{warzee}
It was proved in~\cite[Theorem~6.6 and Proposition~6.1]{MR3910470} that,
provided the connection $\nabla$ on $T_\cM$ is torsion-free,
we have \[ \fedosov=-\delta+d^{\nabla}+A^\nabla ,\] where
\begin{equation}\label{Conakry}
A^\nabla=\sum_{k=1}^{m+r} \xi_k\Bigg(\sum_{\substack{L\in\ZZ_{\geq 0}^{m+r} \\
\abs{L}\geq 2}}
\sum_{j=1}^{m+r} A^k_{L,j} \otimes y^L\frac{\partial}{\partial y_j}\Bigg),
\end{equation}
with $ A^k_{L,j}\in C^\infty (\cM)$, 
is an element of degree~$(+1)$ of
$\OO^1\big(\cM,\hat{S}^{\geq 2}(T_\cM^\vee)\otimes T_\cM\big)$
--- regarded as an operator acting on the algebra
$\OO^\bullet\big(\cM,\hat{S}(T_\cM^\vee)\big)$ by derivation 
satisfying 
\[ h_\natural(A^\nabla)=(h\otimes\id_{T_\cM})(A^\nabla)=0 .\]

Actually, $A^\nabla\in \OO^1\big(\cM,\hat{S}^{\geq 2}(T_\cM^\vee)\otimes T_\cM\big)$ 
is uniquely determined by the equations below:
\[ \begin{cases}
\big(-\delta+d^{\nabla}+A^\nabla \big)^2 = 0 \\
h_\natural(A^\nabla)=0
\end{cases} \]
\end{remark}

%Each vector field $X\in\sections{\cM}$ determines
%a coderivation $\nabla^\lightning_X$ of the filtered coalgebra
%\[ \sections{S^{\leq 0}(T_\cM)}\subset
%\sections{S^{\leq 1}(T_\cM)}\subset
%\sections{S^{\leq 2}(T_\cM)}\subset\cdots \]
%and, consequently, a derivation (also denoted $\nabla^\lightning_X$) of the dual filtered algebra
%\[ \sections{\hat{S}^{\geq 0}(T_\cM^\vee)}\supset
%\sections{\hat{S}^{\geq 1}(T_\cM^\vee)}\supset
%\sections{\hat{S}^{\geq 2}(T_\cM^\vee)}\supset\cdots \]

%The derivation $\nabla^\lightning_X$ of $\sections{\hat{S}(T_\cM^\vee)}$
%is a `vertical' vector field on $T_\cM$, i.e.\ its expression in local coordinates is
%of the following type:
%\[ \nabla^\lightning_X=\sum_{J\in\ZZ_{\geq 0}^{m+r}}
%\sum_{k=1}^{m+r} X_{J,k}\otimes y^J\frac{\partial}{\partial y_k}
%\qquad\text{with}\quad X_{J,k}\in C^\infty(\cM) .\]

%In local coordinates, the covariant differential $\fedosov$ reads
%\[ \fedosov = \sum_{k=1}^{m+r} dx_k\otimes\nabla^\lightning_{\frac{\partial}{\partial x_k}} ,\]
%where
%\[ \nabla^\lightning_{\frac{\partial}{\partial x_k}}
%= - 1\otimes \frac{\partial}{\partial y_k}
%+\overset{\nabla_{\frac{\partial}{\partial x_k}}}
%{\overbrace{\textcolor{red}{\frac{\partial}{\partial x_k}\otimes 1}
%- \sum_{l=1}^{m+r}\sum_{j=1}^{m+r} \Gamma_{k,l}^j\otimes y_l\frac{\partial}{\partial y_j}}}
%+ \sum_{\substack{L\in\ZZ^{m+r}_{\geq 0}\\ \abs{L}\geq 2}}
%\sum_{j=1}^{m+r} A^k_{L,j}\otimes y^L\frac{\partial}{\partial y_j} ,\]
%and the functions $\Gamma_{k,l}^j\in C^\infty(\cM)$ are the Christoffel symbols
%of the connection $\nabla$ on $T_\cM$.

In local coordinates, the covariant differential $\fedosov$ reads
\begin{equation}\label{Abidjan}
\begin{split}
\fedosov = \sum_{k=1}^{m+r} \xi_k
\Bigg( -1\otimes\frac{\partial}{\partial y_k}
+\overset{\nabla_{\frac{\partial}{\partial x_k}}}
{\overbrace{\frac{\partial}{\partial x_k}\otimes 1
-\sum_{l=1}^{m+r}\sum_{j=1}^{m+r} {(-1)^{|y_l|+ |y_l||y_j|}}
\Gamma_{k,l}^j\otimes y_l\frac{\partial}{\partial y_j}}}
% \\
+\sum_{\substack{L\in\ZZ^{m+r}_{\geq 0}\\ \abs{L}\geq 2}}
\sum_{j=1}^{m+r} A^k_{L,j}\otimes y^L\frac{\partial}{\partial y_j} \Bigg)
,
\end{split}
\end{equation}
where
%\[ \nabla^\lightning_{\frac{\partial}{\partial x_k}}
%= - 1\otimes \frac{\partial}{\partial y_k}
%+\overset{\nabla_{\frac{\partial}{\partial x_k}}}
%{\overbrace{\textcolor{red}{\frac{\partial}{\partial x_k}\otimes 1}
%- \sum_{l=1}^{m+r}\sum_{j=1}^{m+r} \Gamma_{k,l}^j\otimes y_l\frac{\partial}{\partial y_j}}}
%+ \sum_{\substack{L\in\ZZ^{m+r}_{\geq 0}\\ \abs{L}\geq 2}}
%\sum_{j=1}^{m+r} A^k_{L,j}\otimes y^L\frac{\partial}{\partial y_j} ,\]
%and
the Christoffel symbols $\Gamma_{k,l}^j$ of the connection $\nabla$ on $T_\cM$
and the coefficients $A^k_{L,j}$ are functions of the variables $x_1,\dots,x_{m+r}$ exclusively.

%\[ \nabla_{\frac{\partial}{\partial x_k}}\frac{\partial}{\partial x_l}
%=\sum_{j=1}^{m+r}\Gamma_{k,l}^j\frac{\partial}{\partial x_j} ,\]
%or, more precisely,
%\[ \fedosov (\omega\otimes P) = \sum_{k=1}^{m+r} \textcolor{red}{\pm}
%(dx_k\wedge\omega)\otimes\nabla^\lightning_{\frac{\partial}{\partial x_k}}P ,\]
%for all homogeneous $\omega\in\OO(\cM)$ and all $P\in\sections{\hat{S}(T^\vee_\cM)}$.

Note that Fedosov dg manifolds of graded manifolds
were also studied by Cattaneo-Felder \cite{MR2304327}. 

\subsection{Tamarkin--Tsygan calculi of Fedosov dg Lie algebroids}

\subsubsection{Fedosov dg Lie algebroid}

%Let $\cN=T_\cM[1]\oplus \Tformal\cM$ be the formal $\ZZ$-graded manifold with support $M$
%and function algebra $C^\infty(\cN)=\OO\big(\cM,\hat{S}(T^\vee_\cM)\big)$
%as described in Section~\ref{sec:FedosovMfd}.
%Given any torsion-free affine
%connection $\nabla$ on $\cM$, the derivation $\fedosov$ of $C^\infty(\cN)$
%arising as the covariant differential of the induced Fedosov flat connection $\nabla^\lightning$
%on $\hat{S}(T^\vee_\cM)$ may be reinterpreted as a homological vector field on $\cN$.
%Any dg manifold $(\cN,\fedosov)$ obtained in this way is called
%a \emph{Fedosov dg manifold} associated with the $\ZZ$-graded manifold $\cM$. 

The ideal 
$I=\sections{\hat S^{\geq 1}(T_\cM[1] \oplus T_\cM)\dual}$
of the algebra
$C^\infty(\cN) = \sections{\hat S^{\geq 1}
(T_\cM[1] \oplus T_\cM)\dual}$
induces an $I$-adic topology on $C^\infty(\cN)$.
We will consider derivations of the algebra $C^\infty(\cN)$
which are continuous with respect to the $I$-adic topology
so that they are uniquely determined by their restrictions
to $\sections{\hat S^{\leq 1}(T_\cM[1] \oplus T_\cM)\dual}$. 
See, for example, \cite[Section~2.1]{MR4727081}. 
Therefore, in local coordinates, every (continuous) derivation $\cX$
of $C^\infty(\cN) = \OO\big(\cM,\hat{S}(T_\cM^\vee)\big)$ reads
\[ \cX = \sum_{k=1}^{m+r} \sum_{J\in\ZZ_{\geq 0}^{m+r}}
\left\{ \lambda_{k,J} \, y^J \otimes \frac{\partial}{\partial x_k}
+ \mu_{k,J} \, y^J \otimes \frac{\partial}{\partial\xi_k}
+ \nu_{k,J} \, y^J \otimes \frac{\partial}{\partial y_k} \right\} \]
for some $\lambda_{k,J}$, $\mu_{k,J}$ and $\nu_{k,J}$ in $\OO(\cM)$. 
We will use the symbol $\XX(\cN)$ to denote
the space of (continuous) derivations of the algebra $C^\infty(\cN)$. 
Together, the Lie algebra of derivations $\XX(\cN)$
and the algebra of functions $C^\infty(\cN)$ form a Lie--Rinehart algebra. 
Indeed, $\XX(\cN)$ can be considered
as the space of sections of the tangent bundle $T_\cN\to\cN$.

We will be most interested in the Lie subalgebroid $\cF\to\cN$ of
$T_\cN\to\cN$ whose sections $\cY$ are vector fields on $\cN$ of the following type:
\[ \cY = \sum_{k=1}^{m+r} \sum_{J\in\ZZ_{\geq 0}^{m+r}}
\nu_{k,J} \otimes y^J \frac{\partial}{\partial y_k} \]
with $\nu_{k,J}$ in $\OO(\cM)$.
As a vector bundle, $\cF\to\cN$ is the pullback of the vector bundle
$T_\cM\to\cM$ through the surjective submersion $\cN\onto\cM$.
It is a graded vector bundle whose total space $\cF$ is a graded manifold with support $M$.
We have the canonical identification
\begin{equation}
\sections{\cF}\cong C^\infty(\cN)\otimes_{C^\infty(\cM)}\sections{T_\cM},
\end{equation}
which will be denoted $\OO\big(\cM,\hat{S}(T_\cM^\vee)\otimes T_\cM\big)$ by abuse of notation.
It is clear that $\cF\to \cN$ is a Lie subalgebroid of $T_\cN\to \cN$.

The vector fields $\delta$ and $A^\nabla$ on $\cN$ introduced earlier
--- see Equations~\eqref{Cotonou} and~\eqref{Conakry} ---
are sections of $\cF\to\cN$, while the vector field $d^{\nabla}$ is not
--- see Equations~\eqref{Lagos} and~\eqref{Yaoundé}.
However, using Equation~\eqref{Abidjan},
one checks by inspection that the subspace
$\sections{\cF}$ of $\XX(\cN)$ is indeed stable under the endomorphism
$\lie{\fedosov}{\argument}$ of $\XX(\cN)$ encoding the dg structure
on the tangent Lie algebroid $T_\cN\to\cN$
of the Fedosov dg manifold $(\cN,\fedosov)$. Thus we have

\begin{lemma}\label{lem:Rome}
The pullback bundle $\cF\to\cN$ is a dg Lie subalgebroid of the tangent
dg Lie algebroid $T_{\cN}\to\cN$ of the Fedosov dg manifold $(\cN,\fedosov)$.
\end{lemma}

In other words, $\cF$ is a dg foliation of the dg manifold $(\cN,\fedosov)$.
Each leaf of this foliation is essentially diffeomorphic to a fixed formal graded
vector space: the standard fiber of the vector bundle $T_\cM\to\cM$.
Any such dg Lie algebroid $\cF\to\cN$ is called a \emph{Fedosov dg Lie algebroid}
associated with the graded manifold $\cM$.

\subsubsection{Tamarkin--Tsygan calculus $\calculus_C(\cF,\fedosov)$}\label{sec:CalCFedosov}

Heuristically, a Fedosov dg Lie algebroid $\cF\to\cN$ associated to
a graded manifold $\cM$ is
``homotopy equivalent" to the tangent Lie algebroid $T_\cM\to\cM$.
In particular, according Section~\ref{sec:CalculiDGLieAbd}, to $\cF\to\cN$,
there are associated two calculi $\calculus_C(\cF,\fedosov)$
and $\calculus_H(\cF,\fedosov)$, respectively.
They are expected to be homotopy equivalent
to the calculi $\calculus_C(\cM)$, and $\calculus_H(\cM)$, respectively.
In order to establish the homotopy equivalence, one must make
some adaption in defining the spaces involved compared to
the general construction in Section~\ref{sec:CalCdgAbd}
and Section~\ref{sec:CalHdgAbd}.
 
Similar to the algebra of functions, for any $p\geq 1$, the spaces of p-polyvector fields
and p-differential forms are completed in the following way:
\begin{align}
\Tpolyf{p} &= \GG\big(S^p(\cF[-1])\big) \cong
\sections{(\hat S(T_\cM[1])\dual\cotimes\hat S T_\cM\dual)\otimes S^p(T_\cM[-1])}
,\label{eq:TpolypF} \\
\Apolyf{-p} &= \GG\big(S^p(\cF\dual[1])\big) \cong 
\sections{(\hat S(T_\cM[1])\dual\cotimes\hat S^{k}T_\cM\dual)\otimes S^p(T_\cM\dual[1])}
.\label{eq:ApolypF}
\end{align}
%The homogeneous component $\big(\Tpolyf{p}\big)^n$ of $\Tpolyf{p}$ of degree $n$ is defined to be 
%$$
%\big(\Tpolyf{p} \big)^n = \prod_{k=0}^\infty
%\big(\sections{S (T_\cM[1])\dual \otimes S^{k} T_\cM\dual
%\otimes S^p (T_\cM[-1])} \big)^n,
%$$ 
%and the homogeneous components of $\Apolyf{p}$ are defined similarly. 
Abusing notations, we will make frequent use of the heuristic notations
\[ \OO\big(\cM,\hat{S}(T^\vee_\cM)\otimes S^p (T_\cM[-1])\big)
\qquad\text{and}\qquad
\OO\big(\cM,\hat{S}(T^\vee_\cM)\otimes S^p (T_\cM\dual[1])\big) \]
to denote $\Tpolyf{p}$ and $\Apolyf{-p}$, respectively. Set 
\begin{align*}
\Tpolyf{} & = \bigoplus_{p \geq 0} \Tpolyf{p} ,\\
\cApolyf{} & = \prod_{p \geq 0} \Apolyf{-p}
.\end{align*}
Note that it is possible to have certain infinite sums in $\Tpolyf{p}$
or $\Apolyf{p}$ which cause algebraic difficulties.
See Section~\ref{sec:FedosovDGmfd}. In order to avoid these difficulties,
we consider the following direct-sum total spaces:
\begin{align}
\totTpolyF{n} &=\bigoplus_{\substack{p,q,r\in\ZZ \\ p+q+r=n \\
p\geq 0,\ r\geq 0}} \prescript{\cF}{}{\sT}^{r,q,p}, \label{eq:totnTpolyF} \\
\totcApolyF{n} & =\bigoplus_{r\in\ZZ, \, r\geq 0}
\prod_{\substack{p,q\in\ZZ, \, p\geq 0 \\ -p+q=n-r }}
\prescript{\cF}{}{\sA}^{r,q,-p}, \label{eq:totnApolyF}
\end{align}
where 
\begin{align}
\prescript{\cF}{}{\sT}^{r,q,p} & =
\Big(\OO^r\big(\cM,\hat{S}(T^\vee_\cM)\otimes S^p( T_\cM[-1])\big) \Big)^{p+q+r}
\subset \Big(\sections{\cN; S^p(\cF[-1])}\Big)^{p+q+r}, \label{eq:TtripleDeg} \\
\prescript{\cF}{}{\sA}^{r,q,-p} & =
\Big(\OO^r\big(\cM,\hat{S}(T^\vee_\cM)
\otimes S^p( T_\cM\dual [1]) \big)\Big)^{-p+q+r}
\subset \Big(\sections{\cN;S^p( \cF^\vee[1])}\Big)^{-p+q+r} . \label{eq:AtripleDeg} 
\end{align}
The spaces defined in Equation~\eqref{eq:totnTpolyF}
and Equation~\eqref{eq:totnApolyF} are, respectively,
the direct-sum total spaces of the following bigraded spaces:
\begin{align*}
\Big(\Tpolyf{}\Big)^{r,s} & = \bigoplus_{p+q=s} \prescript{\cF}{}{\sT}^{r,q,p} ,\\
\Big(\cApolyf{}\Big)^{r,s} & = \prod_{-p+q=s} \prescript{\cF}{}{\sA}^{r,q,-p}
.\end{align*}

\begin{remark}\label{rmk:SumTotIsSmaller}
The total spaces defined in~\eqref{eq:totnTpolyF} and~\eqref{eq:totnApolyF}
are slightly different from the general total spaces defined in~\eqref{eq:GaredeNord}
and~\eqref{eq:GaredeEst}. If one follows \eqref{eq:GaredeNord}
and~\eqref{eq:GaredeEst}, and considers 
$$
\sT^n  =\bigoplus_{p\geq 0}\TpolyF{p}{n} \qquad \text{and} \qquad \sA^n  = \prod_{p \geq 0}\ApolyF{-p}{n},
$$
% \begin{align*}
% \sT^n & =\bigoplus_{p\geq 0}\TpolyF{p}{n}, \\
% \sA^n & = \prod_{p \geq 0}\ApolyF{-p}{n},
% \end{align*}
then, in general,
$$
\totTpolyF{n} \subsetneq \sT^n \qquad \text{and} \qquad \totcApolyF{n}  \subsetneq \sA^n.
$$
% \begin{align*}
% \totTpolyF{n} & \subsetneq \sT^n, \\
% \totcApolyF{n} & \subsetneq \sA^n.
% \end{align*}
For instance, if $(x_1,\cdots, x_{m+r})$ are local coordinate functions
on $\cM$ such that $|x_{m+1}|=-3$ and $|x_{m+2}|=2$,
then the induced local coordinate functions
$(x_1, \cdots, x_{m+r}, \xi_1, \cdots, \xi_{m+r}, y_1, \cdots, y_{m+r})$
on $\cN$ have the properties that $|\xi_{m+1}|=-2$ and $|y_{m+2}|=2$.
See Remark~\ref{rmk:CoordFedosovMfd}.
Then $(\xi_{m+1}\otimes y_{m+2})^n \in\prescript{\cF}{}{\sT}^{n,-n,0}$,
and the series \[ \sum_{n=1}^\infty (\xi_{m+1} \otimes y_{m+2})^n \in\TpolyF{0}{0} \]
is an element in $\big(\sT^0 - \totTpolyF{0}\big)$ and also an element
in $\big(\sA^0 - \totcApolyF{0}\big)$.
\end{remark}

Although $\totTpolyF{n}$ and $\totcApolyF{n}$ are generally smaller
than the total spaces defined in~\eqref{eq:GaredeNord} and~\eqref{eq:GaredeEst},
it is not difficult to show that $\totTpolyF{\bullet}$ and $\totcApolyF{\bullet}$
are closed under the calculus operations defined in Section~\ref{sec:CalCdgAbd}.
Therefore, by Poposition~\ref{prop:CalC_DGLieAbd}, we have 

\begin{corollary}
Let $\cF \to \cN$ be a Fedosov dg Lie algebroid associated with a graded manifold $\cM$.
The pair
\[ \Theta^{\bullet}=\cohomology{}\big(\totTpolyF{\bullet},\iL_{\fedosov}\big)
\qquad\text{and}\qquad
\Xi_{\bullet}=\cohomology{}\big(\totcApolyF{\bullet},\iL_{\fedosov}\big) \]
admits a Tamarkin--Tsygan calculus structure. 
\end{corollary}

Such a calculus is denoted $\calculus_C(\cF,\fedosov)$.

\subsubsection{Tamarkin--Tsygan calculus $\calculus_H(\cF,\fedosov)$}\label{sec:CalHFedosov}

Regarding the calculus $\calculus_H(\cF,\fedosov)$,
 since the dg Lie algebroid $\cF\to\cN$ is a dg foliation on $\cN$,
its universal enveloping algebra $\enveloping{\cF}$ can be considered
 as the algebra
of leafwise differential operators on $\cN$.
It can be identified in a natural way with the $C^\infty(\cN)$-module
$\OO\big(\cM,\hat{S}(T^\vee_\cM)\otimes S(T_\cM)\big)$. 
Dually, the space $\jet{\cF}$ of $\cF$-jets can be identified with the $C^\infty(\cN)$-module
$\OO\big(\cM,\hat{S}(T^\vee_\cM)\otimes \hat S(T_\cM\dual)\big)$. 

\begin{remark}\label{rmk:VerticalDiffOp}
The universal enveloping algebra
$\enveloping{\cF}\cong\OO\big(\cM,\hat{S}(T^\vee_\cM)\otimes S(T_\cM)\big)$
can be identified with the algebra of continuous $\OO(\cM)$-linear differential
operators on $C^\infty(\cN)=\OO\big(\cM,\hat{S}(T^\vee_\cM)\big)$.
Thus, elements in $\enveloping{\cF}$ can be considered as differential operators
on $\cN$ that are tangent to the fibers of $\cN=T_\cM[1]\oplus\Tformal\cM\to T_\cM[1]$,
and they are sometimes referred to as vertical differential operators. See~\cite{MR4727081}.
\end{remark}

Consequently, we have the following identifications 
\begin{align}
\Dpolyf{p} &=\big(\pshift\,\enveloping{\cF}\big)^{\otimes p}
\cong\sections{(\hat S(T_\cM[1])\dual\cotimes\hat S T_\cM\dual) \otimes ((ST_\cM)[-1])^{\otimes p}},
\\
\Cpolyf{-p} & = \big(\nshift\jet{\cF}\big)^{\cotimes p} 
\cong\sections{(\hat S(T_\cM[1])\dual\cotimes\hat S T_\cM\dual)
\otimes ((\hat ST_\cM\dual)[1])^{\cotimes p}}.
\end{align}
Abusing notations, we often write $\OO\big(\cM,\hat{S}(T^\vee_\cM)
\otimes ((ST_\cM)[-1])^{\otimes p}\big)$ for $\Dpolyf{p}$ and write $\OO\big(\cM,\hat{S}(T^\vee_\cM)
\otimes ((\hat ST_\cM\dual)[1])^{\cotimes p}\big)$ for $\Cpolyf{-p}$.
% \begin{align*}
% \Dpolyf{p} &= \OO\big(\cM,\hat{S}(T^\vee_\cM)
% \otimes ((ST_\cM)[-1])^{\otimes p}\big), \\
% \Cpolyf{-p} &= \OO\big(\cM,\hat{S}(T^\vee_\cM)
% \otimes ((\hat ST_\cM\dual)[1])^{\cotimes p}\big).
% \end{align*}

Similarly to $\totTpolyF{n}$ and $\totcApolyF{n}$, we set
\begin{align}
\totDpolyF{n} &=\bigoplus_{\substack{p,q,r\in\ZZ \\ p+q+r=n \\
p\geq 0,\ r\geq 0}} \prescript{\cF}{}{\sD}^{r,q,p}, \label{eq:totnDpolyF} \\
\totcCpolyF{n} & =\bigoplus_{r\in\ZZ, \, r\geq 0}
\prod_{\substack{p,q\in\ZZ, \, p\geq 0 \\ -p+q=n-r }}
\prescript{\cF}{}{\sC}^{r,q,-p}, \label{eq:totnCpolyF}
\end{align}
where 
\begin{align}
\prescript{\cF}{}{\sD}^{r,q,p} &=
\Big(\OO^r\big(\cM,\hat{S}(T^\vee_\cM)
\otimes ((ST_\cM)[-1])^{\otimes p}\big) \Big)^{p+q+r} ,\label{eq:DtripleDeg} \\
\prescript{\cF}{}{\sC}^{r,q,-p} &=
\Big(\OO^r\big(\cM,\hat{S}(T^\vee_\cM)
\otimes ((\hat ST_\cM\dual)[1])^{\cotimes p}\big)\Big)^{-p+q+r}
.\label{eq:CtripleDeg} 
\end{align}
The spaces defined in Equation~\eqref{eq:totnDpolyF}
and Equation~\eqref{eq:totnCpolyF} are, respectively,
the direct-sum total spaces of the following bigraded spaces:
\begin{align*}
\Big(\Dpolyf{}\Big)^{r,s} &= \bigoplus_{p+q=s} \prescript{\cF}{}{\sD}^{r,q,p}, \\
\Big(\cCpolyf{}\Big)^{r,s} &= \prod_{-p+q=s} \prescript{\cF}{}{\sC}^{r,q,-p}.
\end{align*}

Again, the spaces $\totDpolyF{n}$ and $\totcCpolyF{n}$ are generally smaller
than the total spaces defined in \eqref{eq:GaredeSud} and \eqref{eq:GaredeOuest}.
However, they are closed under the operations defined in Section~\ref{sec:CalHdgAbd}. 
According to Proposition~\ref{prop:CalH_DGLieAbd}, we have 

\begin{corollary}
Let $\cF \to \cN$ be a Fedosov dg Lie algebroid associated with a graded manifold $\cM$. 
The pair 
\[ \Theta^{\bullet}=\cohomology{}\big(\totDpolyF{\bullet},\hochschild+\iL_{\fedosov}\big)
\quad\quad\text{and}\quad\quad
\Xi_{\bullet}=\cohomology{}\big(\totcCpolyF{\bullet},\hochschildb+\iL_{\fedosov}\big) \]
admits a Tamarkin--Tsygan calculus structure.
\end{corollary}
Such a calculus is denoted $\calculus_H(\cF,\fedosov)$.

\section{Fedosov contractions for graded manifolds}

In order to establish Fedosov contractions for dg manifolds,
we need, first of all, consider the case of graded manifolds.
In this section, we investigate the relations between the poly spaces of a graded manifold $\cM$ 
and the poly spaces of an associated Fedosov dg Lie algebroid $\cF \to \cN$, and 
construct a set of contractions between them.
These contractions are referred to as \emph{Fedosov contractions}. 

\subsection{Fedosov contraction for tensor fields}

The main purpose of this subsection is to establish
Fedosov contractions for tensor fields, polyvector fields
and differential forms on a graded manifold
(Proposition~\ref{repository} and Corollary~\ref{cor:contractionTpolyApoly}).
We start with a brief review of the Fedosov contractions
for functions on a graded manifold \cite{MR3910470}. 
%For details, see \cite{MR3910470}. 

\subsubsection{Contraction of functions}

Let $\cM$ be a graded manifold, and $\cN=T_\cM[1]\oplus\Tformal\cM$
the associated Fedosov manifold. 
Recall from Section~\ref{sect:2.2} that there are operators
$\delta$, $\eta$, $h$ and $D^\nabla$ acting
on $C^\infty(\cN)=\OO\big(\cM,\hat{S}(T^\vee_\cM)\big)$. 
The canonical projection
\begin{equation}\label{eq:NAP}
\sigma:\OO^\bullet\big(\cM,\hat{S}(T^\vee_\cM)\big)
\to\OO^0\big(\cM,S^0(T^\vee_\cM)\big)=C^\infty(\cM) 
\end{equation}
and injection
\begin{equation}\label{eq:SCALEA}
\tau:C^\infty(\cM)=\OO^0\big(\cM,S^0(T^\vee_\cM)\big)
\to\OO^\bullet\big(\cM,\hat{S}(T^\vee_\cM)\big) 
\end{equation}
realize a chain equivalence between the cochain complex
$\big(\OO^\bullet\big(\cM,\hat{S}(T^\vee_\cM)\big),-\delta\big)$
and $C^\infty(\cM)$ regarded as a cochain complex concentrated in degree~$0$.
Indeed, we have a Kozul contraction:

\begin{equation}\label{Amos} \begin{tikzcd}
C^\infty(\cM) \arrow[r, "\tau", shift left] &
\big(\OO^\bullet\big(\cM,\hat{S}(T^\vee_\cM)\big),-\delta\big)
\arrow[l, "\sigma", shift left] \arrow[loop, "h", out=5, in=-5, looseness = 3]
%\arrow["h", loop right]
\end{tikzcd} .\end{equation}

%The cochain complex $\big(\OO^\bullet\big(\cM,\hat{S}(T^\vee_\cM)\big),-\delta\big)$
%is endowed with the exhaustive and complete filtration 
%\begin{equation}\label{eq:filtration_function}
% \cdots\subset\Omega\big(\cM,\hat S^{\geq q+1}(T^\vee_\cM)\big)\subset
%\Omega\big(\cM,\hat S^{\geq q}(T^\vee_\cM)\big)\subset
%\Omega\big(\cM,\hat S^{\geq q-1}(T^\vee_\cM)\big)\subset\cdots.
%\end{equation}
%The derivation $\varrho:=\textcolor{red}{-}(d^{\nabla}+A^\nabla)$
%of the algebra $\OO^\bullet\big(\cM,\hat{S}(T_\cM^\vee)\big)$
%is a small perturbation of the the contraction \eqref{Amos}
%with respect to this filtration \eqref{eq:filtration_function}.
%See Appendix~\ref{Vilnius} or \cite[Appendix~A]{MR4665716}. 
%Indeed, we have $\fedosov=-\delta-\varrho$ and $(\fedosov)^2=0$. 
%The homological perturbation lemma\footnote{Lemma~\ref{Riga}} can be applied
%to the contraction \eqref{Amos} and the perturbation $\varrho$
%so as to obtain a new contraction.

By homological pertubation, one proves the following

%The following proposition was proved in \cite{MR3910470}.

\begin{proposition}[{\cite[Section~8]{MR3910470}}]
\label{Zwijndrecht}
Let $\cM$ be a graded manifold.
Given any torsion-free affine connection $\nabla$ on $\cM$,
let \[ \varrho=-(d^{\nabla}+A^\nabla)=-(\fedosov+\delta)
\qquad\text{and}\qquad \breve{\varrho}=\sum_{k=0}^\infty\varrho(h\varrho)^k .\]
Then, there is a contraction:
\begin{equation}\label{Fortune}
\begin{tikzcd}
C^\infty(\cM) \arrow[r, "\breve{\tau}", shift left] &
\big(\OO^\bullet\big(\cM,\hat{S}(T^\vee_\cM)\big),\fedosov\big)
\arrow[l, "\sigma", shift left] \arrow[loop, "\breve{h}", out=5, in=-5, looseness = 3]
%\arrow["\breve{h}", loop right]
\end{tikzcd} ,\end{equation}
where \[ \breve{\tau}=\tau+h\breve{\varrho}\tau
\qquad\text{and}\qquad \breve{h}=h+h\breve{\varrho}h .\]
\end{proposition}
%\begin{proof}
%It is immediate that $\sigma\circ\varrho=0$.
%Therefore, according to Lemma~\ref{Riga},
%the perturbed surjection is 
%\[ \sigma+\sigma\breve{\varrho}h
%=\sigma(\id+\breve{\varrho}h)
%=\sigma\sum_{k=0}^\infty(\varrho h)^k
%=\sigma \]
%and the perturbed differential on $C^\infty(\cM)$ is 
%\[ 0+\sigma\breve{\varrho}\tau
%=\sigma\varrho\sum_{k=0}^\infty(h\varrho)^k\tau=0.\]
%This explains why the sujection and the differential on the subcomplex
%are not modified by the homological perturbation.
%\end{proof}

%\begin{comment}
%The filtration in \cite[Section~8]{MR3910470} is 
%\[ F^m = \prod_{p+q \geq m} \OO^p(\cM,S^q(T_\cM\dual)) \]
%which is NOT correct. This $F^m$ is not even a subset of
%$\OO\big(\cM,\hat{S}(T^\vee_\cM)\big) 
%=\varprojlim_{q}\frac{\OO\big(\cM,S(T^\vee_\cM)\big)}
%{\OO\big(\cM,S^{\geq q}(T^\vee_\cM)\big)}$.
%For example, if $x$ is a coordinate of degree one,
%then $\sum_{p=m}^\infty (dx)^p$ is a nonzero element in $F^m$, 
%not in $\OO\big(\cM,\hat{S}(T^\vee_\cM)\big)$.
%\end{comment}

\subsubsection{Contraction of tensor fields}

We note that the vector field $\delta$ is a section
of the Fedosov dg Lie algebroid $\cF\to\cN$. 
Therefore, the subspace $\sections{\cN,\cF}$ of $\XX(\cN)$
is stable under the Lie derivative $\liederivative{\delta}$.

Likewise, the space $\sections{\cN,\cF^{\otimes k}\otimes(\cF^\vee)^{\otimes l}}$ 
of $\cF$-longitudinal tensor fields of type $(k,l)$
on $\cN$ is an $\liederivative{\delta}$-stable
subspace of $\sections{\cN,T_\cN^{\otimes k}\otimes(T_\cN^\vee)^{\otimes l}}$.
It is immediate that, for all $g\in C^\infty(\cN)$
and all possible choices of indices 
$i_1,\ldots,i_k,j_1,\ldots,j_l$, we have 
\[ \liederivative{\delta}\left(g\cdot
\tfrac{\partial}{\partial y_{i_1}}\otimes\cdots\otimes\tfrac{\partial}{\partial y_{i_k}}
\otimes dy_{j_1}\otimes\cdots\otimes dy_{j_l}\right)
=\delta(g)\cdot
\tfrac{\partial}{\partial y_{i_1}}\otimes\cdots\otimes\tfrac{\partial}{\partial y_{i_k}}
\otimes dy_{j_1}\otimes\cdots\otimes dy_{j_l} .\]
The local coordinate functions
$x_1,\cdots,x_{m+r}, \xi_1,\cdots, \xi_{m+r},y_1,\cdots,y_{m+r}$
on $\cN=T_{\cM}[1]\oplus \Tformal{\cM}$ associated with a choice
of local coordinates $x_1,\cdots,x_{m+r}$ on $\cM$
are as in Remark~\ref{rmk:CoordFedosovMfd}.

Hence, modifying slightly the Koszul contraction \eqref{Amos},
we obtain the contraction
\begin{equation}\label{eq:Churchill}
\begin{tikzcd}[cramped]
\Big(\GG\big(\cM;(T_\cM)^{\otimes k}\otimes(T_\cM^\vee)^{\otimes l}\big),0\Big)
\arrow[r, "\tau_\natural^{k,l}", shift left] &
\Big(\GG\big(\cN;\cF^{\otimes k}\otimes(\cF^\vee)^{\otimes l}\big),\iL_{-\delta}\Big)
\arrow[l, "\sigma_\natural^{k,l}", shift left]
\arrow[loop, "h_\natural^{k,l}", out=7, in=-7, looseness=3]
\end{tikzcd}
\end{equation}
with the homotopy $h_\natural$, the surjection $\sigma_\natural$,
and the injection $\tau_\natural$, respectively, defined by the relations
\begin{gather}
h_\natural\left(g\cdot
\tfrac{\partial}{\partial y_{i_1}}\otimes\cdots\otimes\tfrac{\partial}{\partial y_{i_k}}
\otimes dy_{j_1}\otimes\cdots\otimes dy_{j_l}\right)
=h(g)\cdot
\tfrac{\partial}{\partial y_{i_1}}\otimes\cdots\otimes\tfrac{\partial}{\partial y_{i_k}}
\otimes dy_{j_1}\otimes\cdots\otimes dy_{j_l}, \label{eq:NAP1}
\\
\sigma_\natural\left(g\cdot
\tfrac{\partial}{\partial y_{i_1}}\otimes\cdots\otimes\tfrac{\partial}{\partial y_{i_k}}
\otimes dy_{j_1}\otimes\cdots\otimes dy_{j_l}\right)
=\sigma(g)\cdot
\tfrac{\partial}{\partial x_{i_1}}\otimes\cdots\otimes\tfrac{\partial}{\partial x_{i_k}}
\otimes dx_{j_1}\otimes\cdots\otimes dx_{j_l}, \label{eq:NAP2}
\\
\tau_\natural\left(f\cdot
\tfrac{\partial}{\partial x_{i_1}}\otimes\cdots\otimes\tfrac{\partial}{\partial x_{i_k}}
\otimes dx_{j_1}\otimes\cdots\otimes dx_{j_l}\right)
=\tau(f)\cdot
\tfrac{\partial}{\partial y_{i_1}}\otimes\cdots\otimes\tfrac{\partial}{\partial y_{i_k}}
\otimes dy_{j_1}\otimes\cdots\otimes dy_{j_l}, \label{eq:NAP3}
\end{gather}
for all $g\in C^\infty(\cN)$, $f\in C^\infty(\cM)$,
and all possible choices of indices 
$i_1,\ldots,i_k,j_1,\ldots,j_l$.

The degree of cochain in the complex
$\Big(\GG\big(\cN;\cF^{\otimes k}\otimes
(\cF^\vee)^{\otimes l}\big),\iL_{-\delta}\Big)$
%$\sections{\cN;\cF^{\otimes k}\otimes(\cF^\vee)^{\otimes l}}$ 
is revealed by the canonical identification
\[ \begin{split}
\GG\big(\cN;\cF^{\otimes k}\otimes(\cF^\vee)^{\otimes l}\big)
& \cong C^\infty(\cN)\cotimes_{\cR}
\GG\big(\cM;(T_\cM)^{\otimes k}\otimes(T_\cM^\vee)^{\otimes l}\big) \\
& \cong \OO^\bullet\big(\cM,\hat{S}(T^\vee_\cM)
\otimes(T_\cM)^{\otimes k}\otimes(T_\cM^\vee)^{\otimes l}\big) ;
\end{split} \]
it is the degree ``$\bullet$'' of the differential form. 
On the other hand, $\GG\big(\cM;(T_\cM)^{\otimes k}\otimes(T_\cM^\vee)^{\otimes l}\big)$ 
is seen as a cochain complex concentrated in degree~$0$.

In particular, for $k=0$ and $l=0$, we obtain the Koszul contraction \eqref{Amos}
for functions; for $k=1$ and $l=0$, we obtain the Koszul contraction for vector fields:
\begin{equation}\label{eq:FcontractionVFnoPert}
\begin{tikzcd}
\XX(\cM) \arrow[r, "\tau_\natural", shift left] &
\big(\sections{\cN;\cF}=
\OO^\bullet\big(\cM,\hat{S}(T^\vee_\cM)\otimes T_\cM\big),\liederivative{-\delta}\big)
\arrow[l, "\sigma_\natural", shift left]
\arrow[loop, "h_\natural", out=5, in=-5, looseness = 3]
\end{tikzcd}
\end{equation} 

Let $\cR=C^\infty(\cM)$. 
The filtration of $C^\infty(\cN) \cong \OO\big(\cM,\hat{S}(T_\cM^\vee)\big)$:
\[ \cdots\subset\Omega\big(\cM,\hat S^{\geq q+1}(T^\vee_\cM)\big)\subset
\Omega\big(\cM,\hat S^{\geq q}(T^\vee_\cM)\big)\subset
\Omega\big(\cM,\hat S^{\geq q-1}(T^\vee_\cM)\big)\subset\cdots .\]
induces a filtration $\fF_\bullet$ on
$\sections{\cN;\cF^{\otimes k}\otimes(\cF^\vee)^{\otimes l}}$: 
\begin{equation}\label{eq:FiltationFedosovTensor}
\fF_{-u}=
\OO\big(\cM,\hat{S}^{\geq u}(T_\cM^\vee)
\otimes (T_\cM)^{\otimes k}\otimes(T_\cM^\vee)^{\otimes l}\big),
\end{equation}
which is exhaustive and complete.
Here $\hat{S}^{\geq u}(T_\cM^\vee) :=\hat{S}(T_\cM^\vee)$ for all $u\leq 0$. 

While the homological vector field $\fedosov$ on $\cN$
is not tangent to the foliation $\cF$,
the Lie derivative $\liederivative{\fedosov}$
stabilizes the subspace $\sections{\cN;\cF}$ of $\XX(\cN)$.
Indeed, the operator $\liederivative{\fedosov}$
on $\sections{\cN;\cF}$ encodes the dg vector bundle structure of $\cF\to\cN$.
More generally, $\liederivative{\fedosov}$ is, like $\liederivative{-\delta}$,
an endomorphism of the space 
$\sections{\cN,\cF^{\otimes k}\otimes(\cF^\vee)^{\otimes l}}$ 
of $\cF$-longitudinal tensor fields of type $(k,l)$ on $\cN$.

\begin{lemma}
\label{Varseille}
The restriction of $h_\natural\circ\liederivative{\varrho}$
to $\sections{\cN,\cF^{\otimes k}\otimes(\cF^\vee)^{\otimes l}}$
satisfies $(h_\natural\liederivative{\varrho})(\fF_{-u})\subset\fF_{-u-1}$
for all $u\in\ZZ$.
\end{lemma}

\begin{proof}
The proof is a straightforward computation.
\end{proof}

Since $\fedosov=-\delta-\varrho$ is a \emph{homological} vector field,
it follows immediately from Lemma~\ref{Varseille}
that $\liederivative{\varrho}$ is a small perturbation
of the contraction~\eqref{eq:Churchill}. See Appendix~\ref{Vilnius}.
Hence, the homological perturbation
of the contraction \eqref{eq:Churchill} by $\liederivative{\varrho}$
yields a new contraction. See Lemma~\ref{Riga}.
Thus we have the following

\begin{proposition}\label{repository}
Given a graded manifold $\cM$ and a torsion-free
affine connection $\nabla$ on $\cM$, let $\fedosov$ be
the induced Fedosov homological vector field
on $\cN=T_\cM[1]\oplus \Tformal\cM$ and  $\cF\to\cN$ 
the corresponding Fedosov dg Lie algebroid.
Then there are contractions for tensor fields
of any type $(k,l)$ with $k\geq 0$ and $l\geq 0$:
\begin{equation}
\label{eq:Tkl}
\begin{tikzcd}[cramped]
\Big(\sections{\cM;(T_\cM)^{\otimes k}\otimes(T_\cM^\vee)^{\otimes l}},0\Big)
\arrow[r, "\tauTensorF{k}{l}", shift left]
& \Big(\sections{\cN;\cF^{\otimes k}\otimes(\cF^\vee)^{\otimes l}},\iL_{\fedosov}\Big)
\arrow[l, "\sigmaTensorF{k}{l}", shift left]
\arrow[loop, "\hTensorF{k}{l}", out=7, in=-7, looseness=3]
\end{tikzcd}
\end{equation}
Moreover, we have $\breve{h}_\natural=\sum_{n=0}^\infty(h_\natural
\liederivative{\varrho})^n h_\natural$
and $\breve{\tau}_\natural=\sum_{n=0}^\infty 
(h_\natural\liederivative{\varrho})^n\tau_\natural$
with $\varrho=-(\fedosov+\delta)$.
\end{proposition}
\begin{proof}
It remains to show that the surjection and the differential
on $\sections{\cM; (T_\cM)^{\otimes k}\otimes (T_\cM^\vee)^{\otimes l}}$
are not modified by the homological perturbation.
It is immediate that $\sigma_\natural\circ\liederivative{\varrho}=0$.
Therefore, according to Lemma~\ref{Riga}, the perturbed surjection
is $\sigma_\natural\sum_{n=0}^\infty(\liederivative{\varrho}h_\natural)^n=\sigma_\natural$
and the perturbed differential is 
$0+\sigma_\natural\liederivative{\varrho}
\sum_{n=0}^\infty(h_\natural\liederivative{\varrho})^n\tau_\natural=0$.
This concludes the proof of the proposition.
\end{proof}

In particular, for $k=1$ and $l=0$, we obtain
a Fedosov contraction for vector fields
\begin{equation}\label{Pohang} \begin{tikzcd}
\XX(\cM) \arrow[r, "\breve{\tau}_\natural", shift left] &
\big(\sections{\cN;\cF}=
\OO^\bullet\big(\cM,\hat{S}(T^\vee_\cM)\otimes T_\cM\big),\liederivative{\fedosov}\big)
\arrow[l, "\sigma_\natural", shift left]
\arrow[loop, "\breve{h}_\natural", out=5, in=-5, looseness=3]
\end{tikzcd} .\end{equation}
Here, $\XX(\cM)$ is again seen as a cochain complex concentrated in degree~$0$.

For $k=0$ and $l=0$, the contraction \eqref{eq:Tkl} reduces
to the contraction \eqref{Fortune} for the spaces of functions.

\subsubsection{Fedosov contractions for $\sT_{\poly}$ and $\sA_{\poly}$}
\label{sec:Subcontraction_TA}

%\subsubsection{Contractions of polyvector fields and differential forms}

Recall that
\begin{align*}
\Tpolym{p}&=\GG\big(\cM; S^p(T_\cM[-1])\big)\cong\GG\big(\Lambda^p T_\cM\big)[-p] ,\\
\Tpolyf{p}&=\GG\big(\cN; S^p(\cF[-1])\big)
=\OO\big(\cM,\hat{S}(T^\vee_\cM)\otimes S^p( T_\cM[-1])\big)
\cong\GG\big(\Lambda^p\cF\big)[-p] ,\\
\Apolym{-p}&=\GG\big(\cM; S^p(T^\vee_\cM[1])\big)
\cong\GG\big(\Lambda^p T^\vee_\cM\big)[p] ,\\ 
\Apolyf{-p}& =\GG\big(\cN; S^p(\cF^\vee[1])\big)
=\OO\big(\cM,\hat{S}(T^\vee_\cM)\otimes S^p(T_\cM\dual[1])\big)
\cong\GG\big(\Lambda^p\cF^\vee\big)[p]
.\end{align*}
For instance, if $X_1,\ldots,X_p\in\XX(\cM)$ are
homogeneous vector fields and $\degree{X_k}$ denotes
the degree of $X_k$ in $\XX(\cM)$, 
then $X_1\wedge\cdots\wedge X_p$ is a homogeneous element of degree
$\degree{X_1}+\cdots+\degree{X_p}$ in $\GG\big(\Lambda^p T_\cM\big)$ and 
\[ \pshift X_1\odot\cdots\odot\pshift X_p
=\pm (\pshift)^{\otimes p}(X_1\wedge\cdots\wedge X_p)
\in\big(\Tpolym{p}\big)^{p+\degree{X_1}+\cdots+\degree{X_p}} .\]
We have similar identifications for other spaces.

\begin{corollary}\label{cor:contractionTpolyApoly}
Under the same hypothesis as in Proposition~\ref{repository},
for every $p\geq 0$, there is a pair of contractions
\begin{enumerate}
\item at the level of polyvector fields:
\begin{equation}\label{eq:Debrecen}
\begin{tikzcd}[cramped]
\Big(\Tpolym{p},0\Big)
\arrow[r, "\tauTpolyF", shift left] &
\Big(\Tpolyf{p},\tQF\Big)
\arrow[l, "\sigmaTpolyF", shift left]
\arrow[loop, "\hTpolyF" , out=7, in=-7, looseness=3]
\end{tikzcd}
\end{equation}
\item and at the level of differential forms:
\begin{equation}
%\label{eq:Ocontraction0}
\label{eq:Oneonta}
\begin{tikzcd}[cramped]
\Big(\Apolym{-p},0\Big)
\arrow[r, "\tauTpolyF", shift left] &
\Big(\Apolyf{-p},\aQF\Big)
\arrow[l, "\sigmaTpolyF", shift left]
\arrow[loop, "\hTpolyF", out=7, in=-7, looseness=3]
\end{tikzcd}
\end{equation}
\end{enumerate}
\end{corollary}

\begin{proof}
It is simple to see that the contraction \eqref{eq:Tkl}
respects both symmetric and skew-symmetric tensors.
Restricting the contraction \eqref{eq:Tkl} to skew-symmetric tensors 
of type $(p,0)$, we obtain the contraction
\begin{equation}
\label{Lloydminster}
\begin{tikzcd}[cramped]
\Big( \sections{\cM;\Lambda^p T_\cM} , 0 \Big)
\arrow[r, "\tauTensorF{p}{0}", shift left] &
\Big( \sections{\cN;\Lambda^p\cF} , \iL_{\fedosov} \Big)
\arrow[l, "\sigmaTensorF{p}{0}", shift left]
\arrow[loop, "\hTensorF{p}{0}",out=7,in=-7,looseness = 3]
\end{tikzcd}
\end{equation}
The desired contraction \eqref{eq:Debrecen} arises
as the $p$-fold suspension of the contraction \eqref{Lloydminster}.
The second contraction \eqref{eq:Oneonta} can be obtained similarly.
\end{proof}

%\subsubsection{Fedosov contractions for $\sT_{\poly}$ and $\sA_{\poly}$}
%\label{sec:Subcontraction_TA}

To establish Fedosov contractions for dg manifolds,
we will need certain subcontractions of \eqref{eq:Debrecen} and \eqref{eq:Oneonta}.
These subcontractions arise from a detailed analysis of different types of grading. 

Recall that, for $p\in\ZZ_{\geq 0}$, $q\in\ZZ$, and $r\in\ZZ_{\geq 0}$,
we considered the subspaces $\prescript{\cF}{}{\sT}^{r,q,p}\subset\big(\Tpolyf{p}\big)^{p+q+r}$
in~\eqref{eq:TtripleDeg} and $\prescript{\cF}{}\sA^{r,q,-p}\subset\big(\Apolyf{-p}\big)^{-p+q+r}$
in ~\eqref{eq:AtripleDeg}. Parallelly, set 
\[ \sT^{r,q,p}=\begin{cases} \Big(\GG\big(\cM;S^p( T_\cM[-1])\big)\Big)^{p+q}
& \text{if}\ r=0 \\ 0 & \text{if}\ r>0 \end{cases} \]
and
\[ \sA^{r,q,-p}=\begin{cases} \Big(\GG\big(\cM;S^p( T_\cM\dual [1])\big)\Big)^{-p+q}
& \text{if}\ r=0 \\ 0 & \text{if}\ r>0. \end{cases} \]

%Thus, we have $\sT^{0,q,p}=\big(\Tpolym{p}\big)^{p+q}$
%and $\prescript{\cF}{}{\sT}^{r,q,p}\subset\big(\Tpolyf{p}\big)^{p+q+r}$. 

Observe that 
\begin{align*}
&\tQF  :\prescript{\cF}{}{\sT}^{r,q,p} \to \prescript{\cF}{}{\sT}^{r+1,q,p}, &  & \hTpolyF  :\prescript{\cF}{}{\sT}^{r,q,p} \to \prescript{\cF}{}{\sT}^{r-1,q,p}, \\
& \tauTpolyF : \sT^{r,q,p} \to \prescript{\cF}{}{\sT}^{r,q,p}, &  & \sigmaTpolyF  :\prescript{\cF}{}{\sT}^{r,q,p} \to \sT^{r,q,p}.
\end{align*}
Similarly,
\begin{align*}
&\tQF  :\prescript{\cF}{}{\sA}^{r,q,-p} \to \prescript{\cF}{}{\sA}^{r+1,q,-p}, &  & \hTpolyF  :\prescript{\cF}{}{\sA}^{r,q,-p} \to \prescript{\cF}{}{\sA}^{r+1,q,-p}, \\
& \tauTpolyF : \sA^{r,q,-p} \to \prescript{\cF}{}{\sA}^{r,q,-p}, &  & \sigmaTpolyF  :\prescript{\cF}{}{\sA}^{r,q,-p} \to \sA^{r,q,-p}.
\end{align*}

%Similarly, let 
%\[ \sA^{r,q,-p}=\begin{cases} \Big(\GG\big(\cM;S^p( T_\cM\dual [1])\big)\Big)^{-p+q}
%& \text{if}\ r=0 \\ 0 & \text{if}\ r>0 \end{cases} \]
%and
%\[ \prescript{\cF}{}{\sA}^{r,q,-p}=
%\Big(\OO^r\big(\cM,\hat{S}(T^\vee_\cM)
%\otimes S^p( T_\cM\dual [1])\big) \big)\Big)^{-p+q+r}
%\subset \Big(\sections{\cN;S^p( \cF^\vee[1])}\Big)^{q+r} .\]
%Then 
%\begin{align*}
% \tQF & :\prescript{\cF}{}{\sA}^{r,q,p} \to \prescript{\cF}{}{\sA}^{r+1,q,p}, \\
% \hTpolyF & :\prescript{\cF}{}{\sA}^{r,q,p} \to \prescript{\cF}{}{\sA}^{r-1,q,p}, \\
% \tauTpolyF &: \sA^{r,q,p} \to \prescript{\cF}{}{\sA}^{r,q,p}, \\
% \sigmaTpolyF & :\prescript{\cF}{}{\sA}^{r,q,p} \to \sA^{r,q,p}.
%\end{align*}

\begin{corollary}
\label{cor:contractionTA_HomogeneousParts}
Under the same hypothesis as in Proposition~\ref{repository},
for every $p,q\in\ZZ$ with $p\geq 0$, there is a pair of contractions:
%\begin{enumerate}
%\item at the level of polyvector fields:
\begin{equation}
\label{eq:Tcontraction0}
\begin{tikzcd}[cramped]
(\sT^{\bullet,q,p},0)
\arrow[r, "\breve{\tau}_\natural", shift left] &
(\prescript{\cF}{}{\sT}^{\bullet,q,p},\tQF)
\arrow[l, "\sigmaTpolyF", shift left]
\arrow[loop, "\hTpolyF" , out=7, in=-7, looseness=3]
\end{tikzcd}
\end{equation}
%\item and at the level of differential forms:
and
\begin{equation}
\label{eq:Ocontraction0}
\begin{tikzcd}[cramped]
(\sA^{\bullet,q,-p},0)
\arrow[r, "\breve{\tau}_\natural", shift left] &
(\prescript{\cF}{}{\sA}^{\bullet,q,-p},\tQF)
\arrow[l, "\sigmaTpolyF", shift left]
\arrow[loop, "\hTpolyF" , out=7, in=-7, looseness=3]
\end{tikzcd}
\end{equation}
%\end{enumerate}
\end{corollary}
Here, by abuse of notation, we omit the superscripts for the contraction maps
$\breve{\tau}_\natural$, $\sigma_\natural$, and $\breve{h}_\natural$.

\begin{proof}
It suffices to observe that 
%$\sigma_\natural$ and $\tau_\natural$
%stabilize $\prescript{\cF}{}{\sT}^{r,q,p}$;
\begin{gather*}
\tau_\natural(\sT^{r,q,p})\subset\prescript{\cF}{}{\sT}^{r,q,p}; \qquad
\sigma_\natural(\prescript{\cF}{}{\sT}^{r,q,p})\subset\sT^{r,q,p}; \qquad
\liederivative{\delta}(\prescript{\cF}{}{\sT}^{r,q,p})
\subset\prescript{\cF}{}{\sT}^{r+1,q,p}; \\
h_\natural(\prescript{\cF}{}{\sT}^{r,q,p})\subset\prescript{\cF}{}{\sT}^{r-1,q,p};
\qquad\text{and}\qquad 
\liederivative{\varrho}(\prescript{\cF}{}{\sT}^{r,q,p})
\subset\prescript{\cF}{}{\sT}^{r+1,q,p}
.\end{gather*}
Indeed, it follows that $h_\natural\circ\liederivative{\varrho}$ 
and $\liederivative{\varrho}\circ h_\natural$ stabilize 
$\prescript{\cF}{}{\sT}^{r,q,p}$,
and thence that 
\[ \tQF(\prescript{\cF}{}{\sT}^{r,q,p})
\subset\prescript{\cF}{}{\sT}^{r+1,q,p};
\quad
\breve{h}_\natural(\prescript{\cF}{}{\sT}^{r,q,p})
\subset\prescript{\cF}{}{\sT}^{r-1,q,p};
\quad\text{and}\quad 
\breve{\tau}_\natural(\sT^{r,q,p})
\subset\prescript{\cF}{}{\sT}^{r,q,p} .\]
%It suffices to observe that $\sigma_\natural$ and $\breve{\tau}_\natural$
%stabilize $\prescript{\cF}{}{\sT}^{r,q,p}$, while 
%\[ \tQF(\prescript{\cF}{}{\sT}^{r,q,p})\subset\prescript{\cF}{}{\sT}^{r+1,q,p}
%\qquad\text{and}\qquad \breve{h}_\natural(\prescript{\cF}{}{\sT}^{r,q,p})
%\subset\prescript{\cF}{}{\sT}^{r-1,q,p} .\]
The same argument applies to $\sA^{\bullet,q,-p}$
and $\prescript{\cF}{}{\sA}^{\bullet,q,-p}$.
\end{proof}

By Equation~\eqref{eq:GaredeNord} and Equation~\eqref{eq:totnTpolyF},
it is clear that 
\[\totTpolyM{n}=\bigoplus_{\substack{p,q, r \in\ZZ \\
p+q+r=n \\ p\geq 0,\ r\geq 0}}\sT^{r,q,p}
,\qquad \totTpolyF{n}=\bigoplus_{\substack{p,q,r\in\ZZ \\ p+q+r=n \\
p\geq 0,\ r\geq 0}} \prescript{\cF}{}{\sT}^{r,q,p}. \]

Similarly, by Equation~\eqref{eq:GaredeEst}
and Equation~\eqref{eq:totnApolyF}, we also have
\[ \totApolyM{n}=\prod_{\substack{p,q,r\in\ZZ \\ p+q=n \\
p\geq 0,\ r\geq 0}}\sA^{r,q,-p}, \quad
\totcApolyF{n}=\bigoplus_{r\in\ZZ, \, r\geq 0}
\prod_{\substack{p,q\in\ZZ, \, p\geq 0 \\ -p+q=n-r }}
\prescript{\cF}{}{\sA}^{r,q,-p} .\]

By the same proof of Corollary~\ref{cor:contractionTA_HomogeneousParts},
we have the following 

\begin{corollary}
\label{pro:contractionTO}
Under the same hypothesis as in Proposition~\ref{repository},
there are contractions
\begin{enumerate}
\item at the level of polyvector fields:
\begin{equation}
\label{eq:Tcontraction}
\begin{tikzcd}[cramped]
\Big(\totTpolyM{\bullet},0\Big)
\arrow[r, "\tauTpolyF", shift left] &
\Big(\totTpolyF{\bullet},\tQF\Big)
\arrow[l, "\sigmaTpolyF", shift left]
\arrow[loop, "\hTpolyF" ,out=7,in=-7,looseness = 3]
\end{tikzcd}
\end{equation}
\item and at the level of differential forms:
\begin{equation}
\label{eq:Ocontraction}
\begin{tikzcd}[cramped]
\Big(\totApolyM{\bullet},0\Big)
\arrow[r, "\tauTpolyF", shift left] &
\Big(\totcApolyF{\bullet},\aQF\Big)
\arrow[l, "\sigmaTpolyF", shift left]
\arrow[loop, "\hTpolyF" ,out=7,in=-7,looseness = 3]
\end{tikzcd}
\end{equation}
\end{enumerate}
\end{corollary}

\begin{remark}\label{rmk:BiCx_TpolyApoly}
The cochain complex $\Big(\totTpolyF{\bullet},\iL_{\fedosov}\Big)$
can be considered as the direct-sum total complexes of the double complex: 
\begin{equation}
\begin{small}
\begin{tikzcd}
& \vdots & \vdots & \vdots & \\
0 \ar[r] & \displaystyle \bigoplus_{p+q=s+1}
\prescript{\cF}{}{\sT}^{0,q,p} \ar[u,"0"] \ar[r, "\iL_{\fedosov}"]
& \displaystyle \bigoplus_{p+q=s+1}\prescript{\cF}{}{\sT}^{1,q,p}
\ar[u,"0"] \ar[r, "\iL_{\fedosov}"]
& \displaystyle \bigoplus_{p+q=s+1}\prescript{\cF}{}{\sT}^{2,q,p}
\ar[u,"0"] \ar[r, "\iL_{\fedosov}"] & \cdots \\
0 \ar[r] & \displaystyle \bigoplus_{p+q=s}\prescript{\cF}{}{\sT}^{0,q,p}
\ar[u,"0"] \ar[r, "\iL_{\fedosov}"]
& \displaystyle \bigoplus_{p+q=s}\prescript{\cF}{}{\sT}^{1,q,p}
\ar[u,"0"] \ar[r, "\iL_{\fedosov}"]
& \displaystyle \bigoplus_{p+q=s}\prescript{\cF}{}{\sT}^{2,q,p}
\ar[u,"0"] \ar[r, "\iL_{\fedosov}"] & \cdots \\
& \vdots \ar[u,"0"] & \vdots \ar[u,"0"] & \vdots \ar[u,"0"] & 
\end{tikzcd}
\end{small}
\end{equation}
Parallelly, the cochain complex
$\Big(\totcApolyF{\bullet},\iL_{\fedosov}\Big)$
can be considered as the direct-sum total complexes of the double complex: 
\begin{equation}
\begin{small}
\begin{tikzcd}
& \vdots & \vdots & \vdots & \\
0 \ar[r] & \displaystyle \prod_{-p+q=s+1}\prescript{\cF}{}{\sA}^{0,q,-p}
\ar[u,"0"] \ar[r, "\iL_{\fedosov}"]
& \displaystyle \prod_{-p+q=s+1} \prescript{\cF}{}{\sA}^{1,q,-p}
\ar[u,"0"] \ar[r, "\iL_{\fedosov}"]
& \displaystyle \prod_{-p+q=s+1} \prescript{\cF}{}{\sA}^{2,q,-p}
\ar[u,"0"] \ar[r, "\iL_{\fedosov}"] & \cdots \\
0 \ar[r] & \displaystyle \prod_{-p+q=s}\prescript{\cF}{}{\sA}^{0,q,-p}
\ar[u,"0"] \ar[r, "\iL_{\fedosov}"]
& \displaystyle \prod_{-p+q=s} \prescript{\cF}{}{\sA}^{1,q,-p}
\ar[u,"0"] \ar[r, "\iL_{\fedosov}"]
& \displaystyle\prod_{-p+q=s} \prescript{\cF}{}{\sA}^{2,q,-p}
\ar[u,"0"] \ar[r, "\iL_{\fedosov}"] & \cdots \\
& \vdots \ar[u,"0"] & \vdots \ar[u,"0"] & \vdots \ar[u,"0"] & 
\end{tikzcd}
\end{small}
\end{equation}
\end{remark}

\subsection{Fedosov contractions for polydifferential operators and polyjets}

To construct Fedosov contraction for the calculus $\calculus_H(\cM)$
of a graded manifold $\cM$, we need contractions
of polydifferential operators and polyjets.

\subsubsection{Universal enveloping algebra of Fedosov Lie algebroid}

Recall that the universal enveloping algebra $\enveloping{\cF}$
of the Fedosov dg Lie algebroid $\cF\to\cN$ can be identified,
in a natural way, with the $C^\infty(\cN)$-module
$\GG\big(\cN;S(\cF)\big)$. In fact, this natural identification
can be established via a canonical connection of $\cF$
and its formal exponential map. 

As earlier, let $x_1,\cdots,x_{m+r},\xi_1,\cdots,\xi_{m+r},y_1,\cdots,y_{m+r}$
denote the local coordinate functions
on $\cN=T_{\cM}[1]\oplus \Tformal{\cM}$ as in Remark~\ref{rmk:CoordFedosovMfd}
with a choice of local coordinates $x_1,\cdots,x_{m+r}$ on $\cM$.
There exists a canonical Lie algebroid connection $\nabla^\cF:\sections{\cF}\times\sections{\cF}\to\sections{\cF}$ fully characterized, in local coordinates, by the relations
\begin{equation}\label{Tehran}
\nabla^\cF_{\tfrac{\partial}{\partial y_i}}\tfrac{\partial}{\partial y_j}=0, 
\quad\forall i,j\in\{1,2,\cdots,m+r\}
.\end{equation}
It induces a `formal exponential map,' 
which arises as the unique isomorphism of left $C^\infty(\cN)$-modules
\[ \pbw^\cF:\GG\big(\cN;S(\cF)\big)\to\enveloping{\cF} \]
satisfying the relations 
\begin{gather}
\label{Mashhad} \pbw^\cF(f)=f, \quad\forall f\in C^\infty(\cN); \\
\label{Isfahan} \pbw^\cF(X)=X, \quad\forall X\in\GG(\cF);
\end{gather}
and, for all $n\in\NN$ and any homogeneous $X_0,X_1,\ldots,X_n\in\sections{\cF}$,
\begin{equation}\label{Ardabil}
\pbw^\cF(X_0\odot X_1\odot\cdots\odot X_n)
=\frac{1}{n+1}\sum_{k=0}^n\epsilon_k\Big\{X_k\cdot\pbw^\cF(X^{\{k\}})
-\pbw^\cF\big(\nabla^\cF_{X_k}(X^{\{k\}})\big)\Big\}
,\end{equation}
where $\epsilon_k=(-1)^{\degree{X_k}(\degree{X_0}+\cdots+\degree{X_{k-1}})}$
and $X^{\{k\}}=X_0\odot\cdots\odot X_{k-1}\odot X_{k+1}\odot\cdots\odot X_n$.

It follows immediately from Equations~\eqref{Tehran}, \eqref{Mashhad},
\eqref{Isfahan}, and \eqref{Ardabil} that, in any local chart for $\cN$
of the type described in Section~\ref{corniche}, we have 
\[ \pbw^\cF\left(\frac{\partial}{\partial y_{i_1}}
\odot\frac{\partial}{\partial y_{i_2}}\odot\cdots
\odot\frac{\partial}{\partial y_{i_n}}\right)=
\frac{\partial^n}{\partial y_{i_1}\partial y_{i_2}\cdots
\partial y_{i_n}} .\]

\subsubsection{Contractions of polydifferential operators and polyjets}

Since $\delta$ is a section of the Fedosov dg Lie algebroid $\cF\to\cN$,
the operator $U\mapsto\gerstenhaber{-\delta}{U}$ is an endomorphism
of $\enveloping{\cF}$, which we denote by $\liederivative{-\delta}$
by slight abuse of notations.
It is immediate that, for all $g\in C^\infty(\cN)$
and all possible choices of indices $i_1,\ldots,i_n$, we have 
\[ \liederivative{-\delta}\circ\pbw^\cF
\left(g\cdot
\frac{\partial}{\partial y_{i_1}}
\odot\frac{\partial}{\partial y_{i_2}}\odot\cdots
\odot\frac{\partial}{\partial y_{i_n}}\right)
=\pbw^\cF\left(-\delta(g)
\cdot\frac{\partial}{\partial y_{i_1}}
\odot\frac{\partial}{\partial y_{i_2}}\odot\cdots
\odot\frac{\partial}{\partial y_{i_n}}\right) .\]
%Note that $\omega$ is a function of the variables
%$x_1,\cdots,x_{m+r},d x_1,\cdots,d x_{m+r}$ exclusively.

Hence, modifying slightly the Koszul contraction \eqref{Amos},
we obtain a contraction
\begin{equation}
%\label{Churchill}
\begin{tikzcd}[cramped]
\big(\cD(\cM),0\big)
\arrow[r, "\tau_\natural", shift left] &
\big(\enveloping{\cF},\iL_{-\delta}\big)
\arrow[l, "\sigma_\natural", shift left]
\arrow[loop, "h_\natural", out=7, in=-7, looseness=3]
\end{tikzcd}
\end{equation}
whose homotopy $h_\natural$, surjection $\sigma_\natural$,
and injection $\tau_\natural$ are the $\KK$-linear maps defined by the relations
\begin{gather}
h_\natural\circ\pbw^\cF
\left(g\cdot\tfrac{\partial}{\partial y_{i_1}}
\odot\tfrac{\partial}{\partial y_{i_2}}\odot\cdots
\odot\tfrac{\partial}{\partial y_{i_n}}\right)
=\pbw^\cF\left(h(g)
\cdot\tfrac{\partial}{\partial y_{i_1}}
\odot\tfrac{\partial}{\partial y_{i_2}}\odot\cdots
\odot\tfrac{\partial}{\partial y_{i_n}}\right) \label{eq:htp_UF}
\\
\sigma_\natural\circ\pbw^\cF
\left(g\cdot\tfrac{\partial}{\partial y_{i_1}}
\odot\tfrac{\partial}{\partial y_{i_2}}\odot\cdots
\odot\tfrac{\partial}{\partial y_{i_n}}\right)
=\pbw^\nabla\left(\sigma(g)
\cdot\tfrac{\partial}{\partial x_{i_1}}
\odot\tfrac{\partial}{\partial x_{i_2}}\odot\cdots
\odot\tfrac{\partial}{\partial x_{i_n}}\right) \label{eq:sigma_UF}
\\
\tau_\natural\circ\pbw^\nabla
\left(f\cdot\tfrac{\partial}{\partial x_{i_1}}
\odot\tfrac{\partial}{\partial x_{i_2}}\odot\cdots
\odot\tfrac{\partial}{\partial x_{i_n}}\right)
=\pbw^\cF\left(\tau(f)
\cdot\tfrac{\partial}{\partial y_{i_1}}
\odot\tfrac{\partial}{\partial y_{i_2}}\odot\cdots
\odot\tfrac{\partial}{\partial y_{i_n}}\right) \label{eq:tau_UF}
\end{gather}
for all $g\in C^\infty(\cN)$, $f\in C^\infty(\cM)$,
and all possible choices of indices $i_1,\ldots,i_n$.

Similarly, in general, for every $p\in\NN$, we obtain a Koszul contraction
\begin{equation}\label{Razgrad}
\begin{tikzcd}[cramped]
\big(\cD(\cM)^{\otimes p},0\big)
\arrow[r, "\tau_\natural", shift left] &
\big(\enveloping{\cF}^{\otimes p},\iL_{-\delta}\big)
\arrow[l, "\sigma_\natural", shift left]
\arrow[loop, "h_\natural", out=7, in=-7, looseness=3]
\end{tikzcd}
,\end{equation}
where $\cD(\cM)^{\otimes p}$ denotes the tensor product
of $p$-copies of the \emph{left $C^\infty(\cM)$-module} $\cD(\cM)$
and $\enveloping{\cF}^{\otimes p}$ denotes the tensor product
of $p$-copies of the \emph{left $C^\infty(\cN)$-module} $\enveloping{\cF}$.

The degree of the cochain complex
$\big(\enveloping{\cF}^{\otimes p},\iL_{-\delta}\big)$
is the one through the canonical identification
\[ \begin{split}
\enveloping{\cF}^{\otimes p} 
\xleftarrow[\cong]{(\pbw^\cF)^{\otimes p}} 
\GG\big(\cN;(S(\cF))^{\otimes p}\big)
%& \cong C^\infty(\cN)\cotimes_{C^\infty(\cM)} \GG\big(\cM;(S(T_\cM))^{\otimes p}\big) \\
%& \cong \OO^\bullet\big(\cM,\hat{S}(T^\vee_\cM)\big)
%\otimes_{C^\infty(\cM)}\GG\big(\cM;(S(T_\cM))^{\otimes p}\big) \\
& \cong \OO^\bullet\big(\cM,\hat{S}(T^\vee_\cM)
\otimes(S(T_\cM))^{\otimes p}\big) ;
\end{split} \]
i.e. the degree ``$\bullet$'' of the differential form.
On the other hand, $\cD(\cM)^{\otimes p}$
is seen as a cochain complex concentrated in degree~$0$.

Let $\fF_{-u}$ be the complete exhaustive filtration on 
$\enveloping{\cF}^{\otimes p}$
defined by 
\begin{equation*}
\fF_{-u}\big(\enveloping{\cF}^{\otimes p}\big)=(\pbw^\cF)^{\otimes p}
\bigg(
\OO\big(\cM,\hat{S}^{\geq u}(T_\cM^\vee)
\otimes 
S(T_\cM)^{\otimes p} \big)\bigg)
.\end{equation*}
While the homological vector field $\fedosov$
on $\cN$ is not tangent to the foliation $\cF$,
the Lie derivative $\liederivative{\fedosov}$ 
is the endomorphism of $\sections{\cN;\cF}$ which
encodes the dg vector bundle structure of $\cF\to\cN$,
and, consequently, the operator
$\liederivative{\fedosov}=\gerstenhaber{\fedosov}{\argument}$
stabilizes the subspace $\enveloping{\cF}^{\otimes p}$ of the space 
$\cD(\cN)^{\otimes p}$ of $p$-differential operators on $\cN$.
%More generally, $\liederivative{\fedosov}$ is,
%like $\liederivative{-\delta}$, an endomorphism of the space 
%$\sections{\cN,\cF^{\otimes k}\otimes(\cF^\vee)^{\otimes l}}$ 
%of $\cF$-longitudinal tensor fields of type $(k,l)$ on $\cN$.
Hence, the difference $\liederivative{\varrho}=
\liederivative{\fedosov}-\liederivative{-\delta}$
stabilizes $\enveloping{\cF}^{\otimes p}$ as well. 
The following lemma can be verified by a straightforward computation.

\begin{lemma}
The operator $\liederivative{\varrho}=\gerstenhaber{\varrho}{\argument}$
on $\enveloping{\cF}^{\otimes p}$ satisfies
$(h_\natural \liederivative{\varrho})(\fF_{-u})\subset\fF_{-u-1}$ for all $u\in\NN$.
\end{lemma}
%\begin{proof}
%The proof is a straightforward computation.
%\end{proof}

Since the vector field $\fedosov=-\delta-\varrho$ is homological,
it follows immediately that $\liederivative{\varrho}$ is a small perturbation 
%\footnote{See Appendix~\ref{Vilnius}.}
of the Koszul differential $\liederivative{-\delta}$
on $\enveloping{\cF}^{\otimes p}$.
Hence, by Lemma~\ref{Riga}, the homological perturbation 
%\footnote{See Lemma~\ref{Riga}.}
of the contraction \eqref{Razgrad}
by $\liederivative{\varrho}$ yields a new contraction.

\begin{proposition}
\label{Rousse}
Given a graded manifold $\cM$
and a torsion-free affine connection $\nabla$ on $\cM$,
let $\fedosov$ be the induced Fedosov homological vector field
on $\cN=T_\cM[1]\oplus \Tformal{\cM}$ and $\cF\to\cN$
 the corresponding Fedosov dg Lie algebroid.
Then, for every $p\geq 0$, there is a contraction:
\begin{equation}
\label{Sudbury}
\begin{tikzcd}[cramped]
\big(\cD(\cM)^{\otimes p},0\big)
\arrow[r, "\breve{\tau}_\natural", shift left] &
\big(\enveloping{\cF}^{\otimes p},\iL_{\fedosov}\big)
\arrow[l, "\sigma_\natural", shift left]
\arrow[loop, "\breve{h}_\natural", out=7, in=-7, looseness=3]
\end{tikzcd}
\end{equation}
Moreover, we have $\breve{h}_\natural=\sum_{n=0}^\infty(h_\natural
\liederivative{\varrho})^n h_\natural$
and $\breve{\tau}_\natural=\sum_{n=0}^\infty 
(h_\natural\liederivative{\varrho})^n\tau_\natural$
with $\varrho=-(\fedosov+\delta)$.
\end{proposition}

The proof of Proposition~\ref{Rousse} is virtually identical
to that of Proposition~\ref{repository}.

For $p=0$, the contraction \eqref{Sudbury} again reduces
to the contraction \eqref{Fortune} for the spaces of functions.

Dually, one can verify the following

\begin{proposition}
Given a graded manifold $\cM$
and a torsion-free affine connection $\nabla$ on $\cM$,
let $\fedosov$ be the induced Fedosov homological vector field
on $\cN=T_\cM[1]\oplus \Tformal{\cM}$ and $\cF\to\cN$
the corresponding Fedosov dg Lie algebroid.
Then, for every $p\geq 0$, there is a contraction:
\begin{equation}\label{eq:FcJetsTensor}
\begin{tikzcd}[cramped]
\big((\jm)^{\cotimes p},0\big)
\arrow[r, "\breve{\tau}_\natural", shift left] &
\big(\jet{\cF}^{\cotimes p},\iL_{\fedosov}\big)
\arrow[l, "\sigma_\natural", shift left]
\arrow[loop, "\breve{h}_\natural", out=7, in=-7, looseness=3]
\end{tikzcd}
\end{equation}
where $\jm=\Hom_{\cR}(\cD(\cM),\cR)$,
$\jet{\cF}=\Hom_{\cNR}(\enveloping{\cF},\cNR)$,
$\cR=C^\infty(\cM)$,
$\cNR=C^\infty(\cN)$,
and the completed tensor products of $\jm$ and $\jet{\cF}$
are over $\cR$ and $\cNR$, respectively. 

Moreover, we have $\breve{h}_\natural=\sum_{n=0}^\infty(h_\natural
\liederivative{\varrho})^n h_\natural$
and $\breve{\tau}_\natural=\sum_{n=0}^\infty 
(h_\natural\liederivative{\varrho})^n\tau_\natural$
with $\varrho=-(\fedosov+\delta)$.
\end{proposition}

%\begin{proof}
%It remains to explain why the surjection
%and the differential on $\cD(\cM)^{\otimes p}$
%%$\sections{\cM; (T_\cM)^{\otimes k}\otimes (T_\cM^\vee)^{\otimes l}}$
%are not modified by the homological perturbation.
%It is immediate that $\sigma_\natural\circ\liederivative{\varrho}=0$.
%Therefore, according to Lemma~\ref{Riga}, the perturbed surjection
%is $\sigma_\natural\sum_{n=0}^\infty(\liederivative{\varrho}h_\natural)^n=\sigma_\natural$
%and the perturbed differential is 
%$0+\sigma_\natural\liederivative{\varrho}
%\sum_{n=0}^\infty(h_\natural\liederivative{\varrho})^n\tau_\natural=0$.
%\end{proof}

Recall that
\begin{small}
\begin{align*}
\Dpolym{p} & =\textstyle \bigotimes_{\cR}^p\big(\pshift\cD(\cM)\big)
\cong\big(\cD(\cM)^{\otimes p}\big)[-p]
\cong\GG\big(\cM;(S(T_\cM)[-1])^{\otimes p}\big) ,\\
\Dpolyf{p}& = \textstyle \bigotimes_{\cNR}^p\big(\pshift \enveloping{\cF}\big)
\cong\big(\enveloping{\cF}^{\otimes p}\big)[-p]
\cong \OO\big(\cM,\hat{S}(T^\vee_\cM)
\otimes(S(T_\cM)[-1])^{\otimes p}\big) ,\\
\Cpolym{-p} & =\Hom_{\cR}\big(\Dpolym{p},\cR \big)
\cong\big((\jm)^{\hat\otimes p}\big)[p]
\cong \GG\big(\cM;(\hat{S}(T^\vee_\cM)[1])^{\hat\otimes p}\big) ,\\
\Cpolyf{-p}& =\Hom_{\cNR}\big(\Dpolyf{p},\cNR\big)
\cong\big(\jet{\cF}^{\hat\otimes p}\big)[p] 
\cong \OO\big(\cM,\hat{S}(T^\vee_\cM)
\otimes(\hat{S}(T^\vee_\cM)[1])^{\hat\otimes p}\big) .
\end{align*}
\end{small}
Here, the isomorphisms in the last $\cong$ are via the pbw maps.
For instance, if $D_1,\ldots,D_p\in\cD(\cM)$ are homogeneous differential operators
and $\degree{D_k}$ denotes the degree of $D_k$ in $\cD(\cM)$, 
then $D_1\otimes\cdots\otimes D_p$ is a homogeneous element of degree
$\degree{D_1}+\cdots+\degree{D_p}$ in $\cD(\cM)^{\otimes p}$ and
\begin{align*}
\pshift D_1\otimes\cdots\otimes\pshift D_p & \in
\big(\cD(\cM)[-1]\big)^{1+\degree{D_1}}\otimes\cdots\otimes
\big(\cD(\cM)[-1]\big)^{1+\degree{D_p}} \subset\big(\Dpolym{p}\big)^{p+\degree{D_1}+\cdots+\degree{D_p}} 
.\end{align*}
We have similar identifications for other spaces as well.

\begin{corollary}
\label{Maastricht}
Under the same hypothesis as in Proposition~\ref{Rousse},
for every $p\in\ZZ_{\geq 0}$, there is a pair of contractions
\begin{enumerate}
\item at the level of polydifferential operators:
\begin{equation}
\label{eq:OrlyD}
\begin{tikzcd}[cramped]
\Big(\Dpolym{p},0\Big)
\arrow[r, "\tauDpolyF", shift left] &
\Big(\Dpolyf{p},\dQF\Big)
\arrow[l, "\sigmaDpolyF", shift left]
\arrow[loop, "\hDpolyF",out=7,in=-7,looseness = 3]
\end{tikzcd}
\end{equation}
\item and at the level of polyjets:
\begin{equation}
\label{eq:OrlyC}
\begin{tikzcd}[cramped]
\Big(\Cpolym{-p},0\Big)
\arrow[r, "\tauDpolyF", shift left] &
\Big(\Cpolyf{-p},\cQF\Big)
\arrow[l, "\sigmaDpolyF", shift left]
\arrow[loop, "\hDpolyF",out=7,in=-7,looseness = 3]
\end{tikzcd}
\end{equation}
\end{enumerate}
\end{corollary}

\begin{proof}
The desired contraction \eqref{eq:OrlyD} arises
as the $p$-fold suspension of the contraction \eqref{Sudbury}.
The construction \eqref{eq:OrlyC} can be constructed in a dual way. 
\end{proof}

%The proof of Proposition~\ref{Rousse} is virtually identical to that of Proposition~\ref{repository}.

\subsubsection{Fedosov contractions for $\sD_{\poly}$ and $\sC_{\poly}$}
\label{sec:Subcontraction_DC}

Recall that, for $p\in\ZZ_{\geq 0}$, $q\in\ZZ$, and $r\in\ZZ_{\geq 0}$,
we defined the subspaces $\prescript{\cF}{}{\sD}^{r,q,p}
\subset\big(\Dpolyf{p}\big)^{p+q+r}$ in Equation~\eqref{eq:DtripleDeg}
and $\prescript{\cF}{}\sC^{r,q,-p}\subset\big(\Cpolyf{-p}\big)^{-p+q+r}$
in Equation~\eqref{eq:CtripleDeg}. Parallelly, set 
\[ \sD^{r,q,p}=\begin{cases} \Big(\big(\pshift\cD(\cM)\big)^{\otimes p}\Big)^{p+q}
& \text{if}\ r=0 \\ 0 & \text{if}\ r>0 \end{cases} \]
and 
\[ \sC^{r,q,-p}=\begin{cases} \Big(\big(\nshift\jm\big)^{\hat\otimes p}\Big)^{-p+q}
& \text{if}\ r=0 \\ 0 & \text{if}\ r>0 \end{cases} \]

%Thus, we have $\sT^{0,q,p}=\big(\Tpolym{p}\big)^{p+q}$
%and $\prescript{\cF}{}{\sT}^{r,q,p}\subset\big(\Tpolyf{p}\big)^{p+q+r}$. 

Observe that 
\begin{align*}
    & \tQF  :\prescript{\cF}{}{\sD}^{r,q,p} \to \prescript{\cF}{}{\sD}^{r+1,q,p}, && \hTpolyF  :\prescript{\cF}{}{\sD}^{r,q,p} \to \prescript{\cF}{}{\sD}^{r-1,q,p}, \\
    & \tauTpolyF : \sD^{r,q,p} \to \prescript{\cF}{}{\sD}^{r,q,p}, && \sigmaTpolyF  :\prescript{\cF}{}{\sD}^{r,q,p} \to \sD^{r,q,p}.
\end{align*}
Similarly,
\begin{align*}
    & \tQF  :\prescript{\cF}{}{\sC}^{r,q,-p} \to \prescript{\cF}{}{\sC}^{r+1,q,-p}, && \hTpolyF  :\prescript{\cF}{}{\sC}^{r,q,-p} \to \prescript{\cF}{}{\sC}^{r-1,q,-p}, \\
    & \tauTpolyF : \sC^{r,q,-p} \to \prescript{\cF}{}{\sC}^{r,q,-p}, && \sigmaTpolyF :\prescript{\cF}{}{\sC}^{r,q,-p} \to \sC^{r,q,-p}.
\end{align*}

Similarly to Corollary~\ref{cor:contractionTA_HomogeneousParts}, we have the following 

\begin{corollary}
\label{cor:contractionCD_HomogeneousParts}
Under the same hypothesis as in Proposition~\ref{Rousse},
for every $p,q\in\ZZ$ with $p\geq 0$, there is a pair of contractions:
%\begin{enumerate}
%\item at the level of polyvector fields:
\begin{equation}
\label{eq:Dcontraction0}
\begin{tikzcd}[cramped]
(\sD^{\bullet,q,p},0)
\arrow[r, "\breve{\tau}_\natural", shift left] &
(\prescript{\cF}{}{\sD}^{\bullet,q,p},\tQF)
\arrow[l, "\sigmaTpolyF", shift left]
\arrow[loop, "\hTpolyF" , out=7, in=-7, looseness=3]
\end{tikzcd}
\end{equation}
%\item and at the level of differential forms:
and
\begin{equation}
\label{eq:Ccontraction0}
\begin{tikzcd}[cramped]
(\sC^{\bullet,q,-p},0)
\arrow[r, "\breve{\tau}_\natural", shift left] &
(\prescript{\cF}{}{\sC}^{\bullet,q,-p},\tQF)
\arrow[l, "\sigmaTpolyF", shift left]
\arrow[loop, "\hTpolyF" , out=7, in=-7, looseness=3]
\end{tikzcd}
\end{equation}
%\end{enumerate}
\end{corollary}
Here, by abuse of notation, we omit the superscripts for the contraction maps
$\breve{\tau}_\natural$, $\sigma_\natural$, and $\breve{h}_\natural$.

It is simple to see that 
\begin{align*}
\totDpolyM{n}=\bigoplus_{\substack{p,q, r \in\ZZ \\ p+q+r=n \\ p\geq 0,\ r\geq 0}}\sD^{r,q,p},
& \quad \totDpolyF{n}=\bigoplus_{\substack{p,q,r\in\ZZ \\ p+q+r=n \\
p\geq 0,\ r\geq 0}} \prescript{\cF}{}{\sD}^{r,q,p}, \\
\totCpolyM{n}=\prod_{\substack{p,q,r\in\ZZ \\ p+q=n \\ p\geq 0,\ r\geq 0}}\sC^{r,q,-p}, 
&\quad \totcCpolyF{n} =\bigoplus_{r\in\ZZ, \, r\geq 0}
\prod_{\substack{p,q\in\ZZ, \, p\geq 0 \\ -p+q=n-r }} \prescript{\cF}{}{\sC}^{r,q,-p}. 
\end{align*}

Analogous to Corollary~\ref{pro:contractionTO}, we have 

\begin{corollary}
\label{cor:repositoryD}
%Given a finite-dimensional $\ZZ$-graded manifold $\cM$
%and a torsion-free affine connection $\nabla$ on $\cM$,
%let $\cF\to\cN$ be the corresponding Fedosov dg Lie algebroid and let $\fedosov$ be the
%corresponding Fedosov homological vector field on $\cN$.
Under the same hypothesis as in Proposition~\ref{Rousse},
there are contractions
%the same contraction data in \eqref{eq:Dcontraction} and \eqref{eq:Ccontraction1}
%in Proposition~\ref{repositoryD}
%also induce contractions
\begin{enumerate}
\item at the level of polydifferential operators:
\begin{equation}
\label{eq:Dcontraction1}
\begin{tikzcd}[cramped]
\Big(\totDpolyM{\bullet},0\Big)
\arrow[r, "\tauDpolyF", shift left] &
\Big(\totDpolyF{\bullet},\dQF\Big)
\arrow[l, "\sigmaDpolyF", shift left]
\arrow[loop, "\hDpolyF",out=7,in=-7,looseness = 3]
\end{tikzcd}
\end{equation}
\item and at the level of polyjets:
\begin{equation}
\label{eq:Ccontraction1}
\begin{tikzcd}[cramped]
\Big(\totCpolyM{\bullet},0\Big)
\arrow[r, "\tauDpolyF", shift left] &
\Big(\totcCpolyF{\bullet},\cQF\Big)
\arrow[l, "\sigmaDpolyF", shift left]
\arrow[loop, "\hDpolyF",out=7,in=-7,looseness = 3]
\end{tikzcd}
\end{equation}
\end{enumerate}
\end{corollary}

\section{Algebraic properties of the injection maps}

In this section, we show that the injective maps
$\tauDpolyF$ in the contractions \eqref{eq:Tcontraction},
\eqref{eq:Ocontraction}, \eqref{eq:Dcontraction1} and \eqref{eq:Ccontraction1}
respect the calculus structures. 
This fact is crucial to establish the Fedosov contractions
for dg manifolds. See Section~\ref{sect:5}. 

%\subsection{An alternative description of the injection maps}

%\subsection{Properties of the map $\etendu{\perturbed{\tau}}$
%in \eqref{eq:Tcontraction} and \eqref{eq:Ocontraction}}

\subsection{Properties of the maps $\breve{\tau}_\sharp$ for $\sT_{\poly}$ and $\sA_{\poly}$}

The next proposition gives an alternative 
characterization of the map
$\etendu{\perturbed{\tau}}$ in \eqref{eq:Tcontraction}
and \eqref{eq:Ocontraction}
as the solution of an initial value problem.

\begin{proposition}\label{Mandalay}
Let $\etendu{\perturbed{\tau}}$ be the injection map
\eqref{eq:Tcontraction} as in Corollary~\ref{pro:contractionTO}.
Given $x\in\Tpolym{p}$ and $y\in\Tpolyf{p}$, we have
\[ \etendu{\perturbed{\tau}}(x)=y
%\qquad\text{if and only if}\qquad
%\begin{cases}
%\etendu{h}(y) = 0 \\
%\liederivative{\fedosov}(y) = 0 \\
%\etendu{\sigma}(y) = x
%\end{cases}
\qquad\text{if and only if}\qquad
\begin{cases}
\etendu{h}(y) = 0 \\
\liederivative{\fedosov}(y) = 0 \\
\etendu{\sigma}(y) = x
\end{cases} \]
The same statement holds for 
$x\in\Apolym{-p}$
and $y\in\Apolyf{-p}$.
\end{proposition}

%The derivation $\eta \in \XX (\cN)$ was defined in Equation~\eqref{Cotonou}.

%\begin{remark}
%It is not difficult to show that
%$\ker(\etendu{h})=\ker(\liederivative{\eta})$.
%Thus, we have
%\[ \etendu{h}(y)=0
%\qquad\text{if and only if}\qquad
%\liederivative{\eta}(y)=0 .\]
%\end{remark}

The following lemma can be easily verified.

\begin{lemma}\label{rmk:Paris}
For every $p\geq 0$, there is a pair of natural identifications
\begin{gather}
\label{eq:Paris}
\Tpolyf{p} \xto{\cong}\OO \big(\cM,\hat{S}(T^\vee_\cM)
\otimes\Lambda^p T_\cM\big)[-p],
\\
\label{eq:Rome}
\Apolyf{-p} \xto{\cong}\OO \big(\cM,\hat{S}(T^\vee_\cM)
\otimes\Lambda^p T^\vee_\cM\big)[p]
.\end{gather}

Given $y\in\Tpolyf{p}$, 
the following assertions are equivalent:
\begin{itemize}
%\item $\liederivative{\eta}(y)=0$;
\item $\etendu{h}(y)=0$;
\item $y\in\OO^0\big(\cM,\hat{S}(T^\vee_\cM)
\otimes\Lambda^p T_\cM\big)[-p]$;
\item $\etendu{\perturbed{h}}(y)=0$.
\end{itemize}

Similarly, given $y\in\Apolyf{-p}$, 
the following assertions are equivalent:
\begin{itemize}
%\item $\liederivative{\eta}(y)=0$;
\item $\etendu{h}(y)=0$;
\item $y\in\OO^0\big(\cM,\hat{S}(T^\vee_\cM)\otimes\Lambda^p T^\vee_\cM\big)[p]$;
\item $\etendu{\perturbed{h}}(y)=0$.
\end{itemize}
Note that
\[ \etendu{\perturbed{h}}
=\sum_{k=0}^\infty(\etendu{h}\liederivative{\varrho})^k\etendu{h}
=(\id-\etendu{h}\liederivative{\varrho})^{-1}\etendu{h} \]
and
\[ \etendu{h}=(\id+\etendu{\perturbed{h}}\liederivative{\varrho})^{-1}\etendu{\perturbed{h}} \]
--- see Appendix~\ref{Vilnius}.
\end{lemma}

%Under the identification \eqref{eq:Paris},
%an element $y\in\Tpolyf{k}$ satisfying the condition that
%\[ \hTpolyF(y)=0
%\qquad\text{if and only if}\qquad
%y\in\OO^0\big(\cM,\hat{S}(T^\vee_\cM)
%\otimes\Lambda^k T_\cM\big)[-k] .\]
%Similarly under the identification \eqref{eq:Rome}, an element
%$y\in\Opolyf{k}$ satisfying the condition
%that
%\[ \hTpolyF (y) = 0
%\qquad\text{if and only if}\qquad
%y\in\OO^0\big(\cM,\hat{S}(T^\vee_\cM)\otimes \Lambda^k T^\vee_\cM\big) [-k] .\]

\begin{proof}[Proof of Proposition~\ref{Mandalay}]
Assume $\etendu{\perturbed{\tau}}(x)=y$.
It is immediate that 
\[ \etendu{\sigma}(y)=\etendu{\sigma}\etendu{\perturbed{\tau}}(x)=x .\]
%From $\etendu{h}\etendu{\tau}=0$ and $\etendu{h}\etendu{h}=0$,
%we get \[ \etendu{h}\etendu{\perturbed{\tau}}=
%\etendu{h}\sum_{l=0}^\infty
%(\etendu{h}\liederivative{\varrho})^l\etendu{\tau}
%=\etendu{h}\etendu{\tau}+\etendu{h}(\etendu{h}\liederivative{\varrho})\sum_{l=0}^\infty
%(\etendu{h}\liederivative{\varrho})^l\etendu{\tau}
%=0 .\]
Since $\breve{h}_\natural\circ\breve{\tau}_\natural=0$, we have 
\[ \etendu{h}(y)=\etendu{h}\etendu{\perturbed{\tau}}(x)=0 \]
%and thence \[ \liederivative{\eta}(y)=0 \]
by Lemma~\ref{rmk:Paris}.
Furthermore, since $\etendu{\perturbed{\tau}}$ is a cochain map,
we have $\liederivative{\fedosov}\circ\etendu{\perturbed{\tau}}=
\etendu{\perturbed{\tau}}\circ 0=0$
and thus
\[ \liederivative{\fedosov}(y)=
\liederivative{\fedosov}\big(\etendu{\perturbed{\tau}}(x)\big)=0 .\]

Conversely, assuming
$\etendu{h}(y)=0$;
%$\liederivative{\eta}(y)=0$;
$\liederivative{\fedosov}(y)=0$;
and $\etendu{\sigma}(y)=x$,
it follows from Lemma~\ref{rmk:Paris} that
$\breve{h}_\natural(y)=0$.
Thus it follows from 
\[ \id-\etendu{\perturbed{\tau}}\etendu{\sigma}
=\etendu{\perturbed{h}}\liederivative{\fedosov}
+\liederivative{\fedosov}\etendu{\perturbed{h}}, \]
that
\[ y-\etendu{\perturbed{\tau}}(x)=
y-\etendu{\perturbed{\tau}}\etendu{\sigma}(y)=
\etendu{\perturbed{h}}\liederivative{\fedosov}(y)
+\liederivative{\fedosov}\etendu{\perturbed{h}}(y)=0 .\]
%\\
%=\sum_{l=0}^\infty(\etendu{h}\liederivative{\varrho})^l
%\etendu{h}\liederivative{\fedosov}(y)
%+\liederivative{\fedosov}
%\sum_{l=0}^\infty(\etendu{h}
%\liederivative{\varrho})^l
%\etendu{h}(y)=0
%\end{multline*}
Hence we conclude that
$\etendu{\perturbed{\tau}}(x)=y$.
\end{proof}

\begin{remark}
Although we will not need this,
we would like to point out that the map $\tauTensorF{k}{l}$
in~\eqref{eq:Tkl} in Proposition~\ref{repository}
admits a similar characterization in terms of an initial value problem
as well.
\end{remark}

The pair of maps $\tauTpolyF$ in Corollary~\ref{pro:contractionTO} 
enjoys good algebraic properties --- see Proposition~\ref{pro:Paris} below.
A few lemmas are needed in order to establish them.

\begin{lemma}
\label{lem:sigmainitialT}
Consider the pair of surjective maps $\sigmaTpolyF$ 
%in \eqref{eq:Tcontraction} and \eqref{eq:Ocontraction}, 
in the contractions of Corollary~\ref{pro:contractionTO}.
\begin{enumerate}
\item The maps $\etendu{\sigma}$ in \eqref{eq:Tcontraction}
and \eqref{eq:Ocontraction} respect the wedge products
of polyvector fields and of differential forms.
%respects the wedge product, and similarly for the map $\sigmaTpolyF$ in \eqref{eq:Ocontraction};
\item The maps $\etendu{\sigma}$ and $\tauApolyF$ in \eqref{eq:Ocontraction}
on differential forms satisfy
\begin{equation}\label{eq:wuxi4}
\etendu{\sigma}\circ d_\cF \circ \tauApolyF =d_{\DR}
.\end{equation}
\item The pair of maps $\etendu{\sigma}$ in \eqref{eq:Tcontraction} and
\eqref{eq:Ocontraction} satisfies the following relations:
\begin{equation}\label{eq:wuxi5}
\sigmaTpolyF(\iI_{\tilde{X}}\tilde{\omega})=
\iI_{\sigmaTpolyF(\tilde{X})}\sigmaTpolyF(\tilde{\omega})
\end{equation}
for all $\tilde{X}\in\Tpolyf{d}$ and $\tilde{\omega}\in\Apolyf{-p}$.
\end{enumerate}
\end{lemma}
\begin{proof}
The assertions (1) and (3) can be verified directly
using the definition \eqref{eq:NAP2} of $\sigmaTpolyF$.
For the assertion (2), recall from Proposition~\ref{repository}
that the map $\tauApolyF$ is given by the formula 
\[ \tauApolyF = \sum_{n=0}^\infty 
(h_\natural\liederivative{\varrho})^n\tau_\natural .\]
As in \eqref{eq:FiltationFedosovTensor}, let
\[ \fF_{-u} = \OO\big(\cM,\hat{S}^{\geq u}(T^\vee_\cM)\otimes
S^p(T_\cM\dual[1])\big)\subset\Apolyf{-p} .\]
Note that
\[ \Apolyf{-p}=\fF_0\supset\ker(\etendu{\sigma})\supset\fF_{-1}\supset\fF_{-2}\supset\cdots \]
We obviously have $d_\cF(\fF_{-u})\subset\fF_{-u+1}$ and, by Lemma~\ref{Varseille},
$(h_\natural\liederivative{\varrho})(\fF_{-u})\subset\fF_{-u-1}$. 
Furthermore, the vector field $A^\nabla$ of Remark~\ref{warzee}
satisfies $\varrho=d^\nabla+A^\nabla$ and
$(h_\natural\liederivative{A^\nabla})(\fF_{-u}) \subset \fF_{-u-2}$.
Therefore, we have
\[ \im\big(d_\cF\circ(\etendu{h}L_\varrho)^2\big)\subset\ker(\etendu{\sigma})
\qquad\text{and}\qquad
\im\big(d_\cF\circ(\etendu{h}L_{A^\nabla})\big)\subset\ker(\etendu{\sigma}) :\]
\[ \begin{tikzcd}
{\Apolyf{-p}} & {\fF_0} & {\ker(\etendu{\sigma})} & {\fF_{-1}} & {\fF_{-2}} & {\cdots}
\arrow["{=}"{description}, draw=none, from=1-1, to=1-2]
\arrow["\supset"{description}, draw=none, from=1-2, to=1-3]
\arrow["{\etendu{h}L_\varrho}", bend left = 20, from=1-2, to=1-4]
\arrow["{\etendu{h}L_{A^\nabla}}", bend left=50, dashed, from=1-2, to=1-5]
\arrow["\supset"{description}, draw=none, from=1-3, to=1-4]
\arrow["{d_\cF}", bend left=20, from=1-4, to=1-2]
\arrow["\supset"{description}, draw=none, from=1-4, to=1-5]
\arrow["{\etendu{h}L_\varrho}", bend left=20, from=1-4, to=1-5]
\arrow["{d_\cF}", bend left=20, from=1-5, to=1-4]
\arrow["\supset"{description}, draw=none, from=1-5, to=1-6]
\end{tikzcd} \]
Hence, we obtain
\[ \etendu{\sigma}\circ d_\cF \circ \tauApolyF
=\etendu{\sigma}\circ d_\cF \circ \sum_{n=0}^\infty 
(h_\natural\liederivative{\varrho})^n\tau_\natural
=\etendu{\sigma}\circ d_\cF \circ (\id+h_\natural\liederivative{\varrho}) \circ \tau_\natural
=\etendu{\sigma}\circ d_\cF \circ (\id+h_\natural\liederivative{d^\nabla}) \circ \tau_\natural
\]
and, since $d_\cF\circ\tau_\natural=0$,
\[ \etendu{\sigma}\circ d_\cF\circ\tauApolyF
=\etendu{\sigma}\circ d_\cF\circ h_\natural\circ
\liederivative{d^\nabla}\circ\tau_\natural .\]

%======
%
%Since $d_\cF(\fF_{-u})\subset\fF_{-u+1}$
%and $\fF_{-u}\subset\ker(\etendu{\sigma})$ for all $u\geq 1$,
%we have
%$$
%\etendu{\sigma}\circ d_\cF \circ \tauApolyF = \etendu{\sigma}\circ d_\cF \circ (\tau_\natural + h_\natural\liederivative{\varrho} \tau_\natural ).
%$$
%It is clear that $ d_\cF \circ \tau_\natural =0$. Furthermore, since $\varrho = d^\nabla + A^\nabla$ and $(h_\natural\liederivative{A^\nabla})(\fF_u) \subset \fF_{u-2}$, we have 
%$$
%\etendu{\sigma}\circ d_\cF \circ \tauApolyF = \etendu{\sigma}\circ d_\cF \circ h_\natural \circ \liederivative{d^\nabla} \circ \tau_\natural .
%$$
%
%======

By Equations~\eqref{Yaoundé} and~\eqref{Lagos}, the covariant derivative
$d^{\nabla}=\sum_{k=1}^{m+r}\xi_k\cdot\nabla_{\frac{\partial}{\partial x_k}}$
can be locally expressed as 
\[ d^\nabla = d-\Gamma ,\]
where $d=\sum_{k=1}^{m+r} \xi_k\frac{\partial}{\partial x_k}$
and $\Gamma=\sum_{j,k,l=1}^{m+r} {(-1)^{|y_l|+ |y_l||y_j|}}
\xi_k \Gamma_{k,l}^j \, y_l\frac{\partial}{\partial y_j}$.
A direct computation shows that 
\[ \etendu{\sigma}\circ d_\cF \circ h_\natural \circ
\liederivative{d} \circ \tau_\natural = d_{\DR} ,\]
and that 
\begin{align*}
\etendu{\sigma}\circ d_\cF & \circ h_\natural \circ
\liederivative{\Gamma} \circ \tau_\natural (f dx_{n_1} \cdots dx_{n_i})
= \etendu{\sigma}\circ d_\cF \circ h_\natural \circ
\liederivative{\Gamma} (f dy_{n_1} \cdots dy_{n_i})
\quad \quad (\text{by Equation \eqref{eq:NAP3})} \\
& = \sum_{v=1}^i \varepsilon_v \, \etendu{\sigma}\circ
d_\cF \circ h_\natural \Big(\liederivative{\Gamma}
(dy_{n_v}) f dy_{n_1} \cdots \widehat{dy_{n_v}} \cdots dy_{n_i} \Big) \\ 
%\quad \quad (\text{by Equation \eqref{eq:NAP3})}\\
& = \sum_{v=1}^i \sum_{k,l=1}^{m+r} (-1)^{|y_l||y_{n_v}| -|y_{n_v}|}
\varepsilon_v \, \etendu{\sigma}\circ d_\cF \circ h_\natural
\Big((\xi_k \Gamma_{k,l}^{n_v} dy_l) f dy_{n_1} \cdots
\widehat{dy_{n_v}} \cdots dy_{n_i} \Big) \\
& = \sum_{v=1}^i \sum_{k,l=1}^{m+r} \varepsilon_v \,
\etendu{\sigma} \Big(\big((-1)^{|y_l||y_{n_v}| -|y_{n_v}|}
dy_k \Gamma_{k,l}^{n_v} dy_l \big) f dy_{n_1} \cdots
\widehat{dy_{n_v}} \cdots dy_{n_i} \Big),
\end{align*}
where $f$ is a local function (of the variables $x_1,\dots,x_{m+r}$) on $\cM$
and \[ \varepsilon_v = (-1)^{|dy_{n_v}|(|f|+|dy_{n_1}|+ \cdots + |dy_{n_{v-1}}|)} .\]
Since the chosen connection $\nabla$ is torsion-free, we have the equation
$\Gamma_{k,l}^j = (-1)^{|y_k||y_l|} \Gamma_{l,k}^j$. Thus, 
\begin{align*}
\sum_{k,l=1}^{m+r} (-1)^{|y_l||y_{n_v}| -|y_{n_v}|}
\, dy_k \Gamma_{k,l}^{n_v} dy_l &= \sum_{k,l=1}^{m+r}
(-1)^{|y_l||y_{n_v}| -|y_{n_v}|} (-1)^{|y_k||y_l|} \, dy_k \Gamma_{l,k}^{n_v} dy_l \\
& = \sum_{k,l=1}^{m+r} \varepsilon_{v,k,l} \, dy_l \Gamma_{l,k}^{n_v} dy_k \\
& = \sum_{k,l=1}^{m+r} \varepsilon_{v,l,k} \, dy_k \Gamma_{k,l}^{n_v} dy_l ,
\end{align*} 
where 
\begin{align*}
\varepsilon_{v,k,l} & = (-1)^{(|y_l|-1)|y_{n_v}|}
(-1)^{|y_k||y_l|+ |dy_k|(|y_{n_v}|-|y_k|-|y_l|+|dy_l|)
+ |dy_l|(|y_{n_v}|-|y_k|-|y_l|)} \\
& = (-1)^{1 + | y_k||y_{n_v}|-|y_{n_v}|}.
\end{align*}
This computation shows that
$\sum_{k,l=1}^{m+r} (-1)^{|y_l||y_{n_v}| -|y_{n_v}|} \,
dy_k \Gamma_{k,l}^{n_v} dy_l$ vanishes, and thus 
\[ \etendu{\sigma}\circ d_\cF \circ h_\natural \circ
\liederivative{\Gamma} \circ \tau_\natural =0 .\]
Therefore, 
\[ \etendu{\sigma}\circ d_\cF \circ \tauApolyF = d_{\DR} ,\]
and the proof is complete.
\end{proof}

We are now ready to prove the following 

\begin{proposition}
\label{pro:Paris}
%Under the same hypothesis as in Proposition~\ref{repository},
Consider the pair of injective maps $\tauTpolyF$
in the contractions of
 Corollary~\ref{pro:contractionTO}.
\begin{enumerate}
%\item The injection $\tauTpolyF$ on polyvector fields 
% \eqref{eq:Tcontraction}
%is a morphism of differential Gerstenhaber algebras.
\item The injection $\tauTpolyF$ in~\eqref{eq:Tcontraction}
and~\eqref{eq:Ocontraction} respect the wedge products
of polyvector fields and of differential forms.
\item The injection $\tauTpolyF$ on differential forms 
\eqref{eq:Ocontraction} satisfies
\begin{equation}\label{eq:wuxi1}
\breve{\tau}_\natural\circ d_{\DR}=d_\cF\circ\breve{\tau}_\natural
.\end{equation}
\item The pair of injections $\tauTpolyF$
in~\eqref{eq:Tcontraction} and~\eqref{eq:Ocontraction}
respects interior products and Lie derivatives:
\begin{gather}
\breve{\tau}_\natural(\iI_{X}\omega)
=\iI_{\breve{\tau}_\natural(X)}\breve{\tau}_\natural(\omega);
\label{eq:wuxi2} \\
\breve{\tau}_\natural(\iL_{X}\omega)
=\iL_{\breve{\tau}_\natural(X)}\breve{\tau}_\natural(\omega),
\label{eq:wuxi3}
\end{gather}
for all $X\in\Tpolym{d}$ and $\omega\in\Apolym{-p}$.
\end{enumerate}
\end{proposition}
\begin{proof}
%(1) Since the graded Lie brackets of the Gerstenhaber algebras
%$\totTpolyM{\bullet}$ and $\totTpolyF{\bullet}$
%are generated by the graded
%Lie brackets between elements in $\Tpolym{k}$, $k=0, 1$,
%and $\Tpolyf{k}$, $k=0, 1$, respectively,
%subject to the graded Lebniz rule,
%by Lemma~\ref{lem:Paris}, it suffices to prove that
(1) We prove that $\tauTpolyF$ in~\eqref{eq:Tcontraction}
respects the wedge product. That is, for all $X
\in \Tpolym{k}$ and $Y\in\Tpolym{l}$ with $k,l\geq 0$, we have
\[ \tauTpolyF(X\wedge Y)=\tauTpolyF(X)\wedge\tauTpolyF(Y) .\]
Since $\hTpolyF\circ\tauTpolyF=0$,
it follows from Lemma~\ref{rmk:Paris} that
\[ \tauTpolyF(X)\in
\OO^0\big(\cM,\hat{S}(T^\vee_\cM)\otimes\Lambda^k T_\cM\big)[-k] \]
and \[ \tauTpolyF(Y)\in\OO^0\big(\cM,\hat{S}(T^\vee_\cM)\otimes
\Lambda^l T_\cM\big)[-l] .\]
Therefore, we obtain \[ \tauTpolyF(X)\wedge\tauTpolyF(Y)\in
\OO^0\big(\cM,\hat{S}(T^\vee_\cM)\otimes\Lambda^{k+l} T_\cM\big)[-k-l] \]
and, according to Lemma~\ref{rmk:Paris}, 
\[ \etendu{h}
%\hTpolyF
\big(\tauTpolyF(X)\wedge\tauTpolyF(Y)\big)=0 .\]
Since $\liederivative{\fedosov}\circ
\breve{\tau}_\natural=0$, we also have
\[ \tQF \big(\tauTpolyF(X) \wedge \tauTpolyF(Y) \big)
=\tQF \big(\tauTpolyF(X)\big) \wedge \tauTpolyF(Y)
+(-1)^k \tauTpolyF(X) \wedge \tQF\big(\tauTpolyF(Y)\big) =0 .\]
Moreover, by Lemma~\ref{lem:sigmainitialT}, we have
\[ \sigmaTpolyF\big(\tauTpolyF(X)\wedge\tauTpolyF(Y)\big)
=\sigmaTpolyF \tauTpolyF (X) \wedge \sigmaTpolyF \tauTpolyF (Y)
=X\wedge Y .\]
From Proposition~\ref{Mandalay}, it follows that
$\tauTpolyF (X) \wedge \tauTpolyF (Y)=\tauTpolyF (X\wedge Y)$.
Hence the injection $\tauTpolyF$ in \eqref{eq:Tcontraction}
respects the wedge product. 

The same argument can be used to prove
that the injection $\tauTpolyF$ in~\eqref{eq:Ocontraction}
respects the wedge product as well.

(2) Since $\breve{h}_\natural\circ
\breve{\tau}_\natural=0$, it follows from
Lemma~\ref{rmk:Paris} that
\[ \breve{\tau}_\natural(\omega)
\in \OO^0\big(\cM,\hat{S}(T^\vee_\cM)\otimes\Lambda^p T^\vee_\cM\big)[p] \] 
for all $\omega\in\Apolym{-p}$.
% Let us first prove Equation~\eqref{eq:wuxi1}.
%First of all,
By the very definition of $d_\cF$ \eqref{eq:CE2}, we have
\[ d_\cF\big(\OO^0\big(\cM,\hat{S}(T^\vee_\cM)\otimes\Lambda^p T^\vee_\cM\big)[p]\big)\subseteq
\OO^0 \big(\cM,\hat{S}(T^\vee_\cM)\otimes \Lambda^{p+1} T^\vee_\cM\big)[p+1] .\]
It thus follows from Lemma~\ref{rmk:Paris} that
%for any $\omega\in\Opolym{n}$, we have
\[ \etendu{h}\big(d_\cF
\breve{\tau}_\natural(\omega)\big)=0 .\]
Since $\liederivative{\fedosov}\circ
\breve{\tau}_\natural=0$, we also have
\[ \aQF\big(d_\cF\tauTpolyF(\omega)\big)
= - d_\cF(\aQF\tauTpolyF(\omega))=0 .\]
Finally, according to Equation~\eqref{eq:wuxi4}, we have
\[ \sigmaTpolyF\big(d_\cF\tauTpolyF(\omega)\big)
= d_{\DR}(\omega) .\]
Therefore, by Proposition~\ref{Mandalay}, we conclude that
$d_\cF \tauTpolyF (\omega)= \tauTpolyF d_{\DR} (\omega)$.

(3) Since $\hTpolyF\circ\tauTpolyF=0$, 
it follows from Lemma~\ref{rmk:Paris} that
\[ \tauTpolyF (X)\in \OO^0 \big(\cM,\hat{S}(T^\vee_\cM)
\otimes \Lambda^d T_\cM\big)[d] \]
and \[ \tauTpolyF ( \omega) \in
\OO^0 \big(\cM,\hat{S}(T^\vee_\cM)\otimes \Lambda^p T^\vee_\cM\big) [-p] .\]
Therefore, we obtain
\[ \iI_{\tauTpolyF(X)}\tauTpolyF(\omega)\in
\OO^0\big(\cM,\hat{S}(T^\vee_\cM)\otimes
\Lambda^{p-d}T^\vee_\cM\big)[-p+d] \]
and, according to Lemma~\ref{rmk:Paris},
\[ \etendu{h}
%\hTpolyF
\big(\iI_{\tauTpolyF(X)}\tauTpolyF(\omega)\big)=0 .\]
Since $\liederivative{\fedosov}\circ
\breve{\tau}_\natural=0$, we also have
\[ \aQF\big(\iI_{\tauTpolyF(X)}\tauTpolyF(\omega)\big)
=\iI_{\aQF\tauTpolyF(X)}\big(\tauTpolyF(\omega)\big)\pm
\iI_{\tauTpolyF(X)}\big(\aQF\tauTpolyF(\omega)\big)=0 .\]

Finally, it follows from
%Lemma~\ref{lem:sigmainitialT},
Equation~\eqref{eq:wuxi5} that
\[ \sigmaTpolyF\big(\iI_{\tauTpolyF(X)}\tauTpolyF(\omega)\big)=
\iI_{\sigmaTpolyF\tauTpolyF(X)}\big(\sigmaTpolyF\tauTpolyF(\omega)\big)=
\iI_{X}\omega .\]
Thus, by Proposition~\ref{Mandalay}, we conclude that
$\iI_{\tauTpolyF(X)}\tauTpolyF(\omega)=\tauTpolyF(\iI_{X}\omega)$.

%To prove Equation~\eqref{eq:wuxi2},
%since $\hTpolyF \tauTpolyF (X)=0$ and
%$\hTpolyF \tauTpolyF (\omega)=0$, according to ,
This concludes the proof of the proposition,
for Equation~\eqref{eq:wuxi3} follows directly
from Equations~\eqref{eq:wuxi1} and~\eqref{eq:wuxi2}
together with Cartan's formula.
\end{proof}

%\subsection{I want to perturb the Hochschild differential and the dg structure at the same time in the next chapter}
%\textcolor{blue}{NEXT: INCORPORATING THE HOCHSCHILD DIFFERENTIAL}

%\subsection{Properties of the map $\breve{\tau}_\natural$ in \eqref{eq:Dcontraction1} and \eqref{eq:Ccontraction1}} 

\subsection{The maps $\breve{\tau}_\sharp$ respect the Hopf algebroid structures}
\label{sec:Hopf}

In order to show that the injections $\tauDpolyF$ in~\eqref{eq:Dcontraction1}
and~\eqref{eq:Ccontraction1} preserve the calculus operations,
we consider the contractions \eqref{Sudbury} and \eqref{eq:FcJetsTensor}. 
%
%%is a morphism of dglas.
%%respects the Gerstenhaber brackets on both sides, i.e.
%is a morphism of dglas, we consider the contraction \eqref{Sudbury}.
Specializing them to the particular case that $p=1$,
we obtain the following contractions:
\begin{equation}
\label{eq:Metz}
\begin{tikzcd}[cramped]
\Big(\cD(\cM),0\Big)
\arrow[r, "\tauDpolyF", shift left] & \Big(\enveloping{\cF},\dQF\Big)
\arrow[l, "\sigmaDpolyF", shift left] \arrow[loop, "\hDpolyF", out=7, in=-7, looseness=3]
\end{tikzcd}
\end{equation}
and
\begin{equation}\label{eq:FcontractionJets}
\begin{tikzcd}[cramped]
\big(\jm,0\big) \arrow[r, "\breve{\tau}_\natural", shift left]
& \big(\jet{\cF},\iL_{\fedosov}\big)
\arrow[l, "\sigma_\natural", shift left]
\arrow[loop, "\breve{h}_\natural", out=7, in=-7, looseness=3]
\end{tikzcd}.
\end{equation}

For a given dg Lie algebroid $(\cL,Q)$, since all the algebraic calculus
operations in $\calculus_H(\cL,\cQ)$ can be expressed in terms of
(i) the Hopf algebroid structure of $\enveloping{\cL}$,
(ii) the formal groupoid structure of $\jet{\cL}$
and (iii) the pairing between $\enveloping{\cL}$ and $\jet{\cL}$.
We aim to show that the injections $\tauDpolyF$
in~\eqref{eq:Metz} and~\eqref{eq:FcontractionJets}
preserve these structures (i), (ii) and (iii).

%In what follows, we will establish the algebraic properties
%for the map $\tauDpolyF$ in \eqref{eq:Metz} and establish
More precisely, the aim of the current section is to establish the following

\begin{proposition}\label{pro:Metz}
The cochain map
\[ \tauDpolyF:\big(\cD(\cM),0\big) \to\big(\enveloping{\cF},\iL_{\fedosov}\big) \]
respects {the source map, the target map,} the algebra structures,
the coalgebra structures and the counit maps,
and hence is a morphism of dg Hopf algebroids\footnote{They
are also called left bialgebroids \cite{MR506407}
or left Hopf algebroids \cite{MR1800718} in the literature.}
\end{proposition}

\subsubsection{Properties of the map $\breve{\tau}_\sharp$ for differential operators}

\begin{proposition}\label{pro:Jolly}
Let $\tauDpolyF$ be the map as in~\eqref{eq:Metz}.
%which is the map resulting from
%specializing the contraction \eqref{Sudbury}
%to the particular case that $p=1$.
Then the following assertions hold:
\begin{enumerate}
\item
The map $\tauDpolyF$ preserves the natural filtrations
on $\cD(\cM)$ and $\enveloping{\cF}$ determined by the orders.
\item
Recall that a function can be seen as
a polyvector field or as a polydifferential operator. 
The map $\breve{\tau}$ appearing in the contraction \eqref{Amos}
is the restriction to functions of the maps $\breve{\tau}_{\natural}$
on polyvector fields and polydifferential operators
appearing in the contractions \eqref{eq:Tcontraction} and \eqref{eq:Dcontraction1}.
\item
Recall that a vector field can be seen as
a polyvector field or as a polydifferential operator. 
The map $\breve{\tau}$ appearing in the contraction \eqref{Pohang}
is the restriction to vector fields of the maps $\breve{\tau}_{\natural}$
on polyvector fields and polydifferential operators
appearing in the contractions \eqref{eq:Tcontraction} and \eqref{eq:Dcontraction1}.
\end{enumerate}
\end{proposition}

\begin{proof}
Recall that all the injections in the contractions \eqref{Amos}, \eqref{Pohang}, 
\eqref{eq:Tcontraction} and \eqref{eq:Dcontraction1} are given by the same formula 
$\breve{h}_\natural=\sum_{n=0}^\infty(h_\natural
\liederivative{\varrho})^n h_\natural$. The assertions (2) and (3) hold.
For the first assertion, recall that the injection $\tauDpolyF$
in~\eqref{eq:Metz} can be decomposed as
\begin{equation}\label{eq:TauDDecomposition}
\tauDpolyF = (\pbw^{\cF})\circ\Big(\sum_{k=0}^\infty\tauTensorF{k}{0}\Big)
\circ\pbw\inv: \cD(\cM)\to\enveloping{\cF}
,\end{equation}
where $\tauTensorF{k}{0}$ is the restriction of the injection
in~\eqref{eq:Tkl} to the symmetric $k$-tensors:
\begin{gather*}
\tauTensorF{k}{0} : \sections{\cM; S^k T_\cM} \to \sections{\cN; S^k\cF}, \\
\sum_{k=0}^\infty \tauTensorF{k}{0} : \sections{\cM; S T_\cM}
=\bigoplus_{k=0}^\infty \sections{\cM; S^k T_\cM}\to\sections{\cN; S \cF}
=\bigoplus_{k=0}^\infty \sections{\cN; S^k\cF}
.\end{gather*}
It is clear that the map $\sum_{k=0}^\infty\tauTensorF{k}{0}$
preserves the natural filtrations $\sections{\cM; S^{\leq p} T_\cM}$
and $\sections{\cN;S^{\leq p}\cF}$. 
Finally, since both maps $\pbw: \sections{\cM; ST_\cM} \to \cD(\cM)$
and $\pbw^\cF:\sections{\cN; S\cF} \to \enveloping{\cF}$
preserve the natural filtrations, it follows
from Equation~\eqref{eq:TauDDecomposition}
that $\tauDpolyF:\cD(\cM)\to\enveloping{\cF}$
preserves the natural filtrations. 
\end{proof}

The next proposition gives an alternative characterization
of the injections $\breve{\tau}_\natural$ in~\eqref{Sudbury}
% and \eqref{eq:Dcontraction1}
as the solutions of initial value problems.

\begin{proposition}\label{prop:IVP_TauD}
Let $\breve{\tau}_\natural$ be the injection in the contraction \eqref{Sudbury}. 
Given $x\in \cD(\cM)^{\otimes p}$ and $y\in \enveloping{\cF}^{\otimes p}$, we have
\[ \etendu{\perturbed{\tau}}(x)=y
%\qquad\text{if and only if}\qquad
%\begin{cases}
%\etendu{h}(y) = 0 \\
%\liederivative{\fedosov}(y) = 0 \\
%\etendu{\sigma}(y) = x
%\end{cases}
\qquad\text{if and only if}\qquad
\begin{cases}
\etendu{h}(y) = 0 \\
\liederivative{\fedosov}(y) = 0 \\
\etendu{\sigma}(y) = x
\end{cases} \]
The same statement holds for 
$x\in\Dpolym{p}$
and $y\in\Dpolyf{p}$.
\end{proposition}

The proof of Proposition~\ref{prop:IVP_TauD} is virtually identical
to the proof of Proposition~\ref{Mandalay},
one can prove it by the following lemma. 
We omit the proof here.

\begin{lemma}
\label{lem:IVP_TauD}
For every $p\geq 0$,
there is a pair of natural identifications
\begin{gather}
\label{eq:FDtensorIso}
\enveloping{\cF}^{\otimes p} \xto{\cong}\OO\big(\cM,\hat{S}(T^\vee_\cM)
\otimes(S T_\cM)^{\otimes p}\big)
\\
\label{eq:FDpolyIso}
\Dpolyf{p}\xto{\cong}\OO \big(\cM,\hat{S}(T^\vee_\cM)
\otimes(S T_\cM)^{\otimes p}\big)[-p]
.\end{gather}

Given $y\in\enveloping{\cF}^{\otimes p}$, 
the following assertions are equivalent:
\begin{itemize}
\item $\etendu{h}(y)=0$
\item $y\in\OO^0\big(\cM,\hat{S}(T^\vee_\cM)
\otimes(S T_\cM)^{\otimes p}\big)$
\item $\etendu{\perturbed{h}}(y)=0$
\end{itemize}
Similarly, 
$y\in\Dpolyf{p}$, 
the following assertions are equivalent:
\begin{itemize}
\item $\etendu{h}(y)=0$
\item $y\in\OO^0\big(\cM,\hat{S}(T^\vee_\cM)
\otimes(S T_\cM)^{\otimes p}\big)[-p]$
\item $\etendu{\perturbed{h}}(y)=0$
\end{itemize}
Note that $\etendu{\perturbed{h}}=
\sum_{k=0}^\infty(\etendu{h}\liederivative{\varrho})^k\etendu{h}=
(\id-\etendu{h}\liederivative{\varrho})^{-1}
\etendu{h}$.
\end{lemma}

\begin{proposition}\label{prop:TauDiffOpTensor}
The injections $\breve{\tau}_\natural$ in~\eqref{Sudbury}
and~\eqref{eq:Dcontraction1} preserve the tensor products.
\end{proposition}

\begin{proof}
It is clear that surjections $\etendu{\sigma}$ in \eqref{Sudbury}
and \eqref{eq:Dcontraction1} preserve the tensor products.
As a consequence of Lemma~\ref{lem:IVP_TauD}, the assertion
follows from an argument parallel to the proof of Proposition~\ref{pro:Paris}~(1). 
\end{proof}

\subsubsection{Actions induced by the multiplication of jets}

Recall that the \emph{Grothendieck flat 
connection}\footnote{Such a Grothendieck flat connection was
introduced by Nest--Tsygan \cite[Proposition~2.7]{MR1913813}
for the space of $L$-jets of any Lie algebroid $L$.}
$\nablag$ on $\jm=\Hom_\cR\big(\cD(\cM),\cR\big)$
is characterized by the relation
\begin{equation}\label{eq:GrothedieckCntnMfd}
\duality{\nablag_X\xi}{u}
={X}(\duality{\xi}{u})- {(-1)^{|X||\xi|} } \duality{\xi}{X\cdot u} ,
\end{equation}
for all homogeneous $X\in\sections{T_\cM}$; $\xi\in\jm$; and $u\in \cD(\cM)$.
Dualizing the isomorphism \eqref{eq:pbw} of (left) $\cR$-coalgebras
$\sections{S(T_\cM)}\to\cD(\cM)$,
we obtain an isomorphism of associative (left) $\cR$-algebras
\begin{equation}
\label{eq:pbwBT}
\pbwb^\top:\jm \xto{\cong}\sections{\hat{S}(T_\cM^\vee)} .
\end{equation}

It is simple to see that the isomorphism $\pbwb^\top$
identifies the Grothendieck flat connection $\nablag$ on $\jm$
with the flat connection $\nabla^\lightning$
on $\sections{\hat{S}(T_\cM^\vee)}$ dual to the flat connection
on $\sections{S(T_\cM)}$ defined by Equation~\eqref{eq:cntb}.
That is, the diagram
\begin{equation}\label{eq:Gconnection}
\begin{tikzcd}
\jm \arrow[r, "\nablag_X"] \arrow[d, "\cong", "\pbwb^\top"']
& \jm \arrow[d, "\cong"', "\pbwb^\top"]
\\ \sections{\hat{S}(T_\cM^\vee)} \arrow[r, "\nabla^\lightning_X"']
& \sections{\hat{S}(T_\cM^\vee)}
\end{tikzcd}
\end{equation}
commutes for all $X\in\sections{T_\cM}$.

Given any homogeneous $v\in\cD(\cM)$, introduce the map $R_v$
as the multiplication by $v$ from the right in $\cD(\cM)$:
\[ R_v:\cD(\cM)\to\cD(\cM), \qquad R_v(u)= {(-1)^{|v||u|}} u\cdot v .\]
It is clear that $R_v$ is an endomorphism of the (left) $\cR$-module $\cD(\cM)$.
Consider its dual $\cR$-endomorphism
\[ R_v^\top:\jm\to\jm, \qquad\duality{R_v^\top(\xi)}{u}
={(-1)^{|v|(|\xi|+|u|)}} \duality{\xi}{u\cdot v} .\]

\begin{lemma}\label{lem:GcntnRvCommute}
For any homogeneous $X\in\sections{T_\cM}$ and $v\in\cD(\cM)$, 
\[ [\nabla^G_X, R_v^\top] = \nabla^G_X \circ R_v^\top
- {(-1)^{|X||v|} } R_v^\top \circ \nabla^G_X = 0 .\]
\end{lemma}
\begin{proof}
For any $X\in\sections{T_\cM}$, $u,v\in\cD(\cM)$ and $\xi\in\jm$, 
\begin{align*}
\duality{\nabla^G_X R_v^\top(\xi)}{u}
&= (-1)^{|v|(|\xi|+|u|)} X(\duality{\xi}{u\cdot v})
- (-1)^{|X||\xi|+|v|(|\xi|+|u|)} \duality{\xi}{X\cdot u\cdot v}, \\
\duality{ R_v^\top \nabla^G_X (\xi)}{u}
&= (-1)^{|v|(|X|+|\xi|+|u|)} X(\duality{\xi}{u\cdot v})
- (-1)^{|X||\xi|+|v|(|X|+|\xi|+|u|)} \duality{\xi}{X\cdot u\cdot v}.
\end{align*}
Thus, the lemma follows.
\end{proof}

The following lemma asserts that $\tauDpolyF(v)$
coincides with the operator $R_v^\top$ conjugated by the algebra isomorphism
$\pbwb^\top$ as in \eqref{eq:pbwBT}. Note that
$\tauDpolyF (v) \in \OO^0 \big(\cM,\hat{S}(T^\vee_\cM)\otimes
S T_\cM \big) \subset \enveloping{\cF}$
is considered as a vertical differential operator on $T_\cM$,
i.e. an $\cR$-linear
differential operator on the algebra $\sections{\hat{S}(T_\cM^\vee)}$.

\begin{lemma}
\label{lem:tauR}
For any given $v\in\cD(\cM)$, the diagram
\[ \begin{tikzcd}
\jm \arrow[r, "R_v^\top"] \arrow[d, "\cong", "\pbwb^\top"']
& \jm \arrow[d, "\cong"', "\pbwb^\top"]
\\ \sections{\hat{S}(T_\cM^\vee)}
\arrow[r, "\etendu{\perturbed{\tau}}(v)"']
& \sections{\hat{S}(T_\cM^\vee)}
\end{tikzcd} \]
commutes.
\end{lemma}
\begin{proof}
Given any $v\in\cD(\cM)$, set
\begin{equation}
\label{Yrhs}
\cY_v=\pbwb^\top\circ R_v^\top \circ (\pbwb^\top)\inv
.\end{equation}

We divide our proof into two steps.

\textsl{Step 1 --- We prove that $\cY_v$ is an $\cR$-linear differential
operator on the algebra $\sections{\hat{S}(T_\cM^\vee)}$.}

We adopt the Sweedler notation
\[ \Delta(u)=\sum_{(u)}u_{(1)}\otimes u_{(2)}
=\sum_{(u)} {(-1)^{|u_{(1)}||u_{(2)}|}} u_{(2)} \otimes u_{(1)} \]
%=\sum_{(u)}u_{(2)}\otimes u_{(1)} \]
to denote the cocommutative comultiplication on $\cD(\cM)$
defined by Equation~\eqref{DOODLE}. It is simple to see
that, for all $f\in\cR\subset\cD(\cM)$ and $u\in\cD(\cM)$, we have
\begin{equation}
\label{eq:Rf}
R_f (u)= {(-1)^{|f||u|}} u\cdot f
=\sum_{(u)} { (-1)^{|f||u_{(2)}| + |u_{(1)}||u_{(2)}|} } \,
{u_{(2)}}(f)\cdot u_{(1)}.
\end{equation}
Hence, for any $\xi\in \jm$, and $u\in \cD(\cM)$, we have
\begin{align*}
\duality{R^\top_f (\xi)}{u}
&= {(-1)^{|f||\xi|}} \duality{\xi}{R_f(u)}
=\sum_{(u)} {(-1)^{|f|(|\xi|+ |u_{(2)}|) + |u_{(1)}||u_{(2)}|}}
\big\langle\xi\big|{u_{(2)}}(f)\cdot u_{(1)}\big\rangle \\
&= \sum_{(u)} {(-1)^{|f|(|\xi|+ |u|) }} \duality{\xi}{u_{(1)}}\cdot {u_{(2)}}(f)
=\sum_{(u)} {(-1)^{|f|(|\xi|+ |u|) }} \duality{\xi}{u_{(1)}}
\cdot\duality{\one}{u_{(2)}\cdot f} \\
&= \sum_{(u)} {(-1)^{|f|(|\xi|+ |u_{(1)}|)}}
\duality{\xi}{u_{(1)}}\cdot\duality{R_f^\top(\one)}{u_{(2)}} \\
&= {(-1)^{|f||\xi|}} \big\langle\xi\otimes R_f^\top(\one)
\big|\sum_{(u)}u_{(1)}\otimes u_{(2)}\big\rangle \\
&= {(-1)^{|f||\xi|}} \duality{\xi\otimes R_f^\top(\one)}{\Delta(u)} \\
&= {(-1)^{|f||\xi|}} \duality{\xi\cdot R_f^\top(\one)}{u}
,\end{align*}
where the element $\one\in\jm$ denotes the morphism
of left $\cR$-modules $\cD(\cM)\ni u\mapsto{u}(1)\in\cR$
associated with the constant function $1\in\cR=C^\infty(\cM)$.
Hence, we have $R_f^\top(\xi)={(-1)^{|f||\xi|}} \xi\cdot R_f^\top(\one)$,
for all $f\in C^\infty(\cM)$ and $\xi\in\jm$.
Therefore, it follows that
\[ {\cY}_f(\varsigma)={(-1)^{|f||\varsigma|}} \varsigma\cdot {\cY}_f(\one),
\qquad\forall f\in C^\infty(\cM),\;\forall
\varsigma\in\sections{\hat{S}(T_\cM^\vee)} .\]
This shows that ${\cY}_f$ is indeed an $\cR$-linear differential operator
\emph{of order zero} on the algebra $\sections{\hat{S}(T_\cM^\vee)}$.

For every $X\in\sections{T_\cM}\subset\cD(\cM)$,
the operator $R_X$ is a coderivation of the (left)
$\cR$-coalgebra $\cD(\cM)$
and, consequently, $R_X^\top$ is a derivation of the (left)
$\cR$-algebra $\jm$. It thus follows that ${\cY}_X$
is a derivation of the $\cR$-algebra $\sections{\hat{S}(T_\cM^\vee)}$.
Since
%, for any $x,y\in\cD(\cM)$,
\begin{equation}
\label{eq:xy}
\cY_{vw}=\cY_v\circ\cY_w, \qquad\forall v,w\in\cD(\cM),
\end{equation}
it follows immediately that, for all $v\in\cD(\cM)$,
the associated operator $\cY_v$ acts on the algebra
$\sections{\hat{S}(T_\cM^\vee)}$ as an
$\cR$-linear differential operator. Furthermore,
since $\pbw$ preserves the natural filtrations,
it is not difficult to show that the operator $\cY_v$
is continuous w.r.t. the $I$-adic topology,
and by Remark~\ref{rmk:VerticalDiffOp}, it can be identified
with an element in $\enveloping{\cF}$. 

\textsl{Step 2 --- We prove that}
\begin{equation}\label{eq:ROME}
\cY_v=\tauDpolyF(v)
\end{equation}

More precisely, we will show that $\cY_v$ is a solution
of the initial value problem in Proposition~\ref{prop:IVP_TauD} specializing in the case $p=1$.
%We prove this statement by using Lemma~\ref{lem:PSU} to show
%that $\cY_x$ satisfies the initial value conditions.

First, by Step~1, we know that $\cY_v$ is a formal vertical
differential operator on $\Tformal\cM$, and thus
\[ \cY_v\in\OO^0\big(\cM,\hat{S}(T^\vee_\cM)\otimes S(T_\cM)\big) .\]
According to Lemma~\ref{lem:IVP_TauD}, we have $\etendu{\perturbed{h}}(\cY_v)=0$.

According to \eqref{eq:Gconnection}, for all $X\in\sections{T_\cM}$
and $\varsigma\in\GG\big(\hat{S}(T^\vee_\cM)\big)$, we have
\[ \interior{X}\fedosov(\varsigma)
=\nabla^\lightning_X\varsigma=
\pbwb^\top\circ\nablag_X\circ(\pbwb^\top)\inv(\varsigma) .\]
Therefore, a direct computation yields that
%by Lemma~\ref{lem:GcntnRvCommute}, 
\[ \interior{X} \commutator{\fedosov}{\cY_v}
=\commutator{\interior{X}\fedosov}{\cY_v}=\pbwb^\top\circ
\commutator{\nablag_X}{R_v^\top}\circ(\pbwb^\top)\inv=0 \]
for all $X\in\sections{T_\cM}$.
Here the last equality follows from Lemma~\ref{lem:GcntnRvCommute}.
Hence we obtain \[ \aQF(\cY_v)=\commutator{\fedosov}{\cY_v}=0 .\]

Let $\{\check{x}_i\}_{i=1,\ldots,m+r}$ and $\{y_i\}_{i=1,\ldots,m+r}$
be the local frames for the vector bundles $T_\cM\dual$ and $\Tformal\cM\dual$,
respectively, induced by local coordinate functions $(x_1,\cdots, x_{m+r})$ on $\cM$.
See Section~\ref{sec:FedosovMfd} for an explanation of the notations.
We need the following multi-indices notations:
\begin{align*}
\check{x}^{\odot I} & = (\check{x}_1)^{\odot i_1} \odot\cdots\odot
(\check{x}_{m+r})^{\odot i_{m+r}} \in \sections{S^{|I|} T_\cM\dual} ,\\
\partial_x^{\odot J} & = \Big(\dfrac{\partial}{\partial x_1}\Big)^{\odot j_1}
\odot\cdots\odot \Big(\dfrac{\partial}{\partial x_{m+r}}\Big)^{\odot j_{m+r}}
\in \sections{S^{|J|}T_\cM}, \\
y^K & = (y_1)^{k_1} \cdots (y_{m+r})^{k_{m+r}} \in\sections{S^{|K|}\Tformal\cM\dual}, \\ 
\partial_y^{L} &= \dfrac{\partial^{l_1}}{\partial(y_1)^{l_1}}
\cdots\dfrac{\partial^{l_{m+r}}}{\partial(y_{m+r})^{l_{m+r}}}
\in\sections{S^{|L|}\Tformal\cM}.
\end{align*}
With these notations, we have 
\begin{align}
\partial_y^I(y^I) &= \epsilon_I \cdot I!, \\
\duality{\check{x}^{\odot I}}{\partial_x^{\odot I}}
&= (-1)^{|\check{x}^{\odot I}||\partial_x^{\odot I}|} \,
\epsilon_I \cdot I! = (-1)^{|\check{x}^{\odot I}|} \,
\epsilon_I \cdot I!, \label{eq:ShuffledPairing}
\end{align}
where $I! = i_1! \cdots i_{m+r}!$ and 
\[ \epsilon_I = (-1)^{\sum_{l=1}^{m+r-1}
i_l|y_l|(i_{l+1}|y_{l+1}| + \cdots + i_{m+r} |y_{m+r}|)} .\]
Now given any formal vertical differential operator $\phi$, locally, it can be expressed as 
\[ \phi = \sum_{I,J \in \NO^{m+r}} \phi_{I,J} \cdot y^I \otimes \partial_y^J ,\]
where $\phi_{I,J}$ are local functions on $\cM$. Since 
\[ \duality{y^I \partial_y^J(y^K) }{1} = \begin{cases}
\epsilon_K \cdot K!, & \text{if } I=0, J=K, \\
0, & \text{otherwise},
\end{cases} \]
we have 
\begin{align*}
\sigma_\natural(\phi)
& = \sum_{J \in \NO^{m+r}} \phi_{0,J} \pbw( \partial_x^{\odot J}) \\
& = \sum_{I,J,K \in \NO^{m+r}} \dfrac{\epsilon_K}{K!} \duality{ \phi_{I,J}
\cdot y^I \partial_y^J(y^K) }{1} \pbw( \partial_x^{\odot J}) \\
& = \sum_{J \in \NO^{m+r}} \dfrac{\epsilon_J}{J!}
\duality{ \phi(y^J)}{1} \pbw( \partial_x^{\odot J})
.\end{align*}
In particular, for the formal vertical differential operator
$\cY_{v}=\pbwb^\top\circ R_{v}^\top\circ(\pbwb^\top)\inv$, we have
\begin{align*}
\sigma_\natural(\cY_v) & = \sum_{J \in \NO^{m+r}} \dfrac{\epsilon_J}{J!}
\duality{ \cY_v (y^J)}{1} \pbw( \partial_x^{\odot J}) \\
&= \sum_{J \in \NO^{m+r}} (-1)^{|v||\check{x}^{\odot J}|} \,
\dfrac{\epsilon_J}{J!} \duality{ \check{x}^{\odot J}}{\pbw\inv\circ R_v\circ\pbw(1)}
\pbw(\partial_x^{\odot J}) \\
&= \sum_{J \in \NO^{m+r}} (-1)^{|v||\check{x}^{\odot J}|} \,
\dfrac{\epsilon_J}{J!} \duality{\check{x}^{\odot J}}{\pbw\inv(v)}
\pbw(\partial_x^{\odot J})
.\end{align*}
Assume $\pbw\inv(v)=\sum_{L\in\NO^{m+r}}\tilde{v}_L\cdot\partial_x^{\odot L}.$ Since
\begin{align*}
\duality{ \check{x}^{\odot J}}{\pbw\inv(v)}
&= \sum_{L\in\NO^{m+r}} \duality{ \check{x}^{\odot J}}
{\tilde{v}_L \cdot \partial_x^{\odot L}} \\
&= (-1)^{|\tilde v_J||\check{x}^{\odot J}| + |\check{x}^{\odot J}|} \,
\epsilon_J \cdot J! \cdot \tilde{v}_J \qquad\text{by Equation~\eqref{eq:ShuffledPairing}} \\
& =(-1)^{| v||\check{x}^{\odot J}| } \, \epsilon_J \cdot J! \cdot \tilde{v}_J,
\end{align*}
we have 
\begin{align*}
\sigma_\natural(\cY_v) & = \sum_{J \in \NO^{m+r}} \tilde{v}_J \cdot \pbw( \partial_x^{\odot J}) \\
& = \pbw\Big( \sum_{J \in \NO^{m+r}} \tilde{v}_J \cdot \partial_x^{\odot J} \Big) \\
& = \pbw \big( \pbw\inv(v)\big) = v.
\end{align*}
It thus follows from Proposition~\ref{prop:IVP_TauD} that $\cY_v=\tauDpolyF(v)$.
\end{proof}

As an immediate consequence of Lemma~\ref{lem:tauR}
and Equation~\eqref{eq:xy}, we have the following 

\begin{lemma}\label{lem:ProdDiffOp_tau}
The injection $\tauDpolyF$ in the contraction \eqref{eq:Metz}
preserves the multiplications: 
\[ \tauDpolyF(vw)=\tauDpolyF(v)\circ\tauDpolyF(w),
\qquad\forall v,w\in\cD(\cM) .\]
\end{lemma}

\subsubsection{Proof of Proposition~\ref{pro:Metz}}

Recall that the comultiplication of $\enveloping{\cF}$ is given by the formula
\begin{multline}\label{eq:CoprodUF}
\Delta(X_1\cdots X_n)=1\otimes(X_1\cdots X_n)
+\sum_{p=1}^{n-1}\sum_{\sigma\in\shuffle{p}{n-p}}
\epsilon^{\sigma}_{X_1,\dots,X_n}
(X_{\sigma_1}\cdots X_{\sigma(p)})\otimes
(X_{\sigma_{p+1}}\cdots X_{\sigma(n)}) \\
+(X_1\cdots X_n)\otimes 1,
\end{multline} 
where $X_1,\cdots,X_n\in\sections{\cF}$, and the Koszul sign
$\epsilon^{\sigma}_{X_1,\dots,X_n}=\pm 1$ is determined
by the equation $X_{\sigma(1)}\odot\cdots\odot X_{\sigma(n)}
=\epsilon^{\sigma}_{X_1,\dots,X_n}X_1\odot\cdots\odot X_n$
in $\sections{S(\cF)}$.

Since the injection $\tauDpolyF$ in~\eqref{Sudbury}
preserves the tensor products according to Proposition~\ref{prop:TauDiffOpTensor},
from Equation~\eqref{DOODLE}, Equation~\eqref{eq:CoprodUF}
and Lemma~\ref{lem:ProdDiffOp_tau}, it follows by a direct verification
that $\tauDpolyF$ preserves the coalgebra structures as well.

\begin{lemma}\label{lem:CoprodDiffOp_tau}
The diagram
\[ \begin{tikzcd}
\cD(\cM) \arrow[d, "\Delta"'] \ar[r,"\tauDpolyF"]
& \enveloping{\cF} \arrow[d, "\Delta"] \\
\cD(\cM) \otimes_\cR \cD(\cM) \ar[r,"\tauDpolyF"]
& \enveloping{\cF}\otimes_\cNR\enveloping{\cF}
\end{tikzcd} \]
commutes.
Here the symbols $\tauDpolyF$ refer to the injection appearing
in the contraction \eqref{Sudbury} with either $p=1$ or $p=2$. 
\end{lemma}

Note that, by Proposition~\ref{prop:TauDiffOpTensor},
the injection $\tauDpolyF:\cD(\cM)\otimes_\cR\cD(\cM)
\to\enveloping{\cF}\otimes_{\cNR}\enveloping{\cF}$ in~\eqref{Sudbury}
coincides with the tensor power of the injection in~\eqref{eq:Metz} 
\[ \tauDpolyF\otimes\tauDpolyF: \cD(\cM)\otimes_\cR\cD(\cM)
\to \enveloping{\cF}\otimes_{\cNR}\enveloping{\cF} .\]
This tensor power $\tauDpolyF\otimes\tauDpolyF$ is well-defined
due to Proposition~\ref{pro:Jolly}~(2).

We are now ready to prove Proposition~\ref{pro:Metz}.

\begin{proof}[Proof of Proposition~\ref{pro:Metz}]
By Lemma~\ref{lem:ProdDiffOp_tau} and Lemma~\ref{lem:CoprodDiffOp_tau},
we know that $\tauDpolyF:\sD(\cM)\to\enveloping{\cF}$ preserves
both the algebra structures and the coalgebra structures.
Furthermore, recall that both the source map and the target map
of the Hopf algebroid of differential operators are the natural inclusion
from the functions into the differential operators \cite{MR1815717}.
Thus, by Proposition~\ref{pro:Jolly}~(2), the map $\tauDpolyF$
preserves the source map and the target map. 
It remains to prove that $\tauDpolyF$ respects the counit maps.

Note that, according to Proposition~\ref{pro:Jolly},
the injection $\tauDpolyF$ preserves the natural filtrations
on $\cD(\cM)$ and $\enveloping{\cF}$ determined by the orders,
and the counit maps $\epsilon:\cD(\cM)\to C^\infty(\cM)$
and $\epsilon:\enveloping{\cF}\to C^\infty(\cN)$ are,
by Definition~\cite{MR1815717}, the projections to their $0$-th order components.
% From the contruction \eqref{!!!!!},
Thus it follows immediately that the diagram
\[ \begin{tikzcd}
\cD(\cM) \ar[d,"\epsilon"'] \ar[r,"\tauDpolyF"]
& \enveloping{\cF} \ar[d,"\epsilon"] \\
C^\infty(\cM) \ar[r,"\tauDpolyF"] & C^\infty(\cN) 
\end{tikzcd} \]
commutes.
The proof is complete.
\end{proof}

\subsection{The maps $\breve{\tau}_\sharp$ respect the formal groupoid structures}

In this section, we aim to establish the following 

\begin{proposition}\label{prop:TauFormalGpd}
The injection
\[ \tauDpolyF: \big(\jm, 0\big) \to \big(\jet{\cF}, \iL_{\fedosov}\big) \]
in the contraction \eqref{eq:FcontractionJets}
respects the dg formal groupoid structures of $\jm$ and $\jet{\cF}$. 
\end{proposition}

\subsubsection{Properties of the map $\breve{\tau}_\sharp$ for jets}

We need a few lemmas.

\begin{lemma}\label{pro:Orly}
Let $\tauDpolyF$ be the injection in the contractions \eqref{eq:Metz}.
The space $\enveloping{\cF}$ is the $C^\infty(\cN)$-span
of $\tauDpolyF\big(\cD(\cM)\big)$.
\end{lemma}

\begin{proof}
Let $\tau_\natural$ be the injection in the contraction \eqref{eq:FcontractionVFnoPert},
and $\tau_\natural\big(\sections{T_\cM}\big)\subset\sections{\cN;\cF}$
the image of $\tau_\natural$. We have
\[ \sections{\cN;\cF}=C^\infty(\cN)\cdot \tau_\natural\big(\TcM \big)
\cong\OO\big(\cM,\hat{S}(T^\vee_\cM)\otimes T_\cM\big) .\]

Consider the exhaustive and complete filtration
\[ \sP\supset\sI\supset\sI^2\supset\sI^3\supset\cdots \]
of the algebra $\sP=\GG\big(\hat{S}(T^\vee_\cM)\big)$
induced by its ideal $\sI=\GG\big(\hat{S}^{\geq 1}(T^\vee_\cM)\big)$.

Since $\breve{h}_\natural\circ\breve{\tau}_\natural=0$,
according to Lemma~\ref{rmk:Paris},
the image of the map $\breve{\tau}_\natural:\XX(\cM)\to\GG(\cN;\cF)$ 
is contained in the subspace 
$\OO^0\big(\cM,\hat{S}(T^\vee_\cM)\otimes T_\cM\big)
\cong\sP\cdot\tau_\natural\big(\TcM\big)$ of $\GG(\cN;\cF)$.

In any local chart of the type described in Remark~\ref{rmk:CoordFedosovMfd} we have
\[ \sigma_\natural\left(
\frac{\partial}{\partial y_j}-\breve{\tau}_\natural
\left(\frac{\partial}{\partial x_j}\right)
\right)=\frac{\partial}{\partial x_j}
-\frac{\partial}{\partial x_j}=0 \]
since 
$\sigma_\natural\circ\breve{\tau}_\natural=\id$.

Therefore, we have
\[ \frac{\partial}{\partial y_j}-\breve{\tau}_\natural
\left(\frac{\partial}{\partial x_j}\right)
\in \OO^0\big(\cM,\hat{S}(T^\vee_\cM)\otimes T_\cM\big)\cap\ker(\sigma_\natural)
=\sI\cdot\tau_\natural\big(\sections{T_\cM}\big) \]
and, hence, there exists a family of functions $A_{ij}\in\sI$ such that 
\[ \frac{\partial}{\partial y_j}
-\breve{\tau}_\natural\left(\frac{\partial}{\partial x_j}\right)
=\sum_i A_{ij}\frac{\partial}{\partial y_i} \]
or, equivalently, 
\begin{equation}\label{tecnico} 
\breve{\tau}_\natural\left(\frac{\partial}{\partial x_j}\right)
=\sum_i (I-A)_{ij}\frac{\partial}{\partial y_i}
.\end{equation}

%Furthermore, since $\sigma_\natural\circ\breve{\tau}_\natural=\id$ and
%$\ker(\sigma)\cap\sS=\sI$,
%\[ \ker(\sigma_\natural)\cap\sS\cdot\pi^*\big(\sections{T_\cM}\big)
%=\sI\cdot\pi^*\big(\sections{T_\cM}\big) ,\]
%we have \[ \frac{\partial}{\partial y_j}-\breve{\tau}_\natural
%\left(\frac{\partial}{\partial x_j}\right)\in
%\sI\cdot\pi^*\big(\sections{T_\cM}\big) \]
%in any local chart of the type described in Section~\ref{?????},
%and there exists a family of functions $A_{ij}\in\sI$ such that 
%\[ \frac{\partial}{\partial y_j}
%-\breve{\tau}_\natural\left(\frac{\partial}{\partial x_j}\right)
%=\sum_i A_{ij}\frac{\partial}{\partial y_i} .\]

We think of the function $A_{ij}$ as the entry of a square matrix $A$
located at the intersection of row $i$ and column $j$.
Then $A^n$ is a square matrix with entries in $\sI^n$,
the geometric series $\sum_{n=0}^\infty A^n$ converges 
to a square matrix $B$ with coefficients in $\sS$, and we have
\[ (I-A)B=(I-A)\sum_{n=0}^\infty A^n=I .\]

It follows from Equation~\eqref{tecnico} that 
\begin{align*}
\sum_j B_{jk}\cdot\breve{\tau}_\natural\left(\frac{\partial}{\partial x_j}\right) 
& =\sum_j B_{jk}\left(\sum_i{(I-A)}_{ij}\frac{\partial}{\partial y_i}\right)  =\sum_i \left(\sum_j{(I-A)}_{ij}B_{jk}\right)\frac{\partial}{\partial y_i} \\
&=\sum_i \delta_{ik}\frac{\partial}{\partial y_i}=\frac{\partial}{\partial y_k}
=\tau_\natural(\frac{\partial}{\partial x_k})
.\end{align*}
Hence, we have proved that $\tau_\natural\big(\TcM\big)\subset
\sS\cdot\breve{\tau}_\natural\big(\XX(\cM)\big)$.
Since $\sS$ is a subalgebra of $C^\infty(\cN)$, it follows immediately that
\[ \sections{\cN;\cF}=C^\infty(\cN)\cdot\tau_\natural\big(\TcM\big)
=C^\infty(\cN)\cdot\breve{\tau}_\natural\big(\XX(\cM)\big) .\]

%It follows that 
%\[ \begin{split}
%\sum_j {\left(\sum_{n=0}^\infty A^n\right)}_{jk} \cdot 
%\breve{\tau}_\natural\left(\frac{\partial}{\partial x_j}\right) 
%& = \sum_j {\left(\sum_{n=0}^\infty A^n\right)}_{jk}
%\left(\sum_i {(I-A)}_{ij} \frac{\partial}{\partial y_i}\right) \\ 
%& = \sum_i \left(\sum_j{(I-A)}_{ij}{\left(\sum_{n=0}^\infty A^n\right)}_{jk}\right)
%\frac{\partial}{\partial y_i} \\ 
%& = \sum_i {\left((I-A)\left(\sum_{n=0}^\infty A^n\right)\right)}_{ik}
%\frac{\partial}{\partial y_i} \\ 
%& = \sum_i \delta_{ik} \frac{\partial}{\partial y_i} \\ 
%& = \frac{\partial}{\partial y_k}
%\end{split} \]

Since $\tauDpolyF:\cD(\cM)\to\enveloping{\cF}$ is a morphism of algebras
and $\tauDpolyF(X)\cdot g=\tauDpolyF(X)(g) + (-1)^{|X||g|} g\cdot\tauDpolyF(X)$
for all homogeneous $X\in\XX(\cM)$ and $g\in C^\infty(\cN)$, we infer that 
\[ \enveloping{\cF}
%=C^\infty(\cN)\cdot\tau_\natural\big(\sections{T_\cM}\big)
=C^\infty(\cN)\cdot\breve{\tau}_\natural\big(\cD(\cM)\big) .\]
Therefore, $\enveloping{\cF}$ is the $C^\infty(\cN)$-span
of $\tauDpolyF\big(\cD(\cM)\big)$.
%
%Therefore, $\Dpolyf{1}$ is the $C^\infty(\cN)$-span of
%$\tauDpolyF\big(\Dpolym{1}\big)$.
%Since $\tauDpolyF$ respects the tensor product of polydifferential operators 
%according to Proposition~\ref{pro:sigmainitialD}, it follows that
%$\Dpolyf{p}$ is the $C^\infty(\cN)$-span of $\tauDpolyF\big(\Dpolym{p}\big)$
%for all $p\geq 1$.
%
%FILL IN FOR C POLY
\end{proof}

The next proposition gives an alternative characterization
of the injections $\breve{\tau}_\natural$ in~\eqref{eq:FcJetsTensor}
and~\eqref{eq:Ccontraction1} as the solutions of initial value problems.

\begin{proposition}\label{prop:IVP_TauC}
Let $\breve{\tau}_\natural$ be the injection in the contraction \eqref{eq:FcJetsTensor}. 
Given $x\in \jm^{\cotimes p}$ and $y\in \jet{\cF}^{\cotimes p}$, we have
\[ \etendu{\perturbed{\tau}}(x)=y
%\qquad\text{if and only if}\qquad
%\begin{cases}
%\etendu{h}(y) = 0 \\
%\liederivative{\fedosov}(y) = 0 \\
%\etendu{\sigma}(y) = x
%\end{cases}
\qquad\text{if and only if}\qquad
\begin{cases}
\etendu{h}(y) = 0 \\
\liederivative{\fedosov}(y) = 0 \\
\etendu{\sigma}(y) = x
\end{cases} \]
The same statement holds for $x\in\Cpolym{-p}$ and $y\in\Cpolyf{-p}$.
\end{proposition}

The proof of Proposition~\ref{prop:IVP_TauC} is virtually identical
to the proof of Proposition~\ref{Mandalay},
one can prove it by the following lemma. 
We omit the proof here.

\begin{lemma}
\label{lem:IVP_TauC}
For every $p\geq 0$,
there is a pair of natural identifications
\begin{gather}
\label{eq:FJtensorIso}
\jet{\cF}^{\cotimes p} \xto{\cong}\OO\big(\cM,\hat{S}(T^\vee_\cM)
\otimes(\hat{S} (T_\cM\dual))^{\cotimes p}\big)
\\
\label{eq:FCpolyIso}
\Cpolyf{-p}\xto{\cong}\OO\big(\cM,\hat{S}(T^\vee_\cM)
\otimes(\hat{S} (T_\cM\dual))^{\cotimes p}\big)[p]
.\end{gather}

Given $y\in \jet{\cF}^{\cotimes p}$, 
the following assertions are equivalent:
\begin{itemize}
\item $\etendu{h}(y)=0$
\item $y\in \OO^0\big(\cM,\hat{S}(T^\vee_\cM)
\otimes(\hat{S} (T_\cM\dual))^{\cotimes p}\big)$
\item $\etendu{\perturbed{h}}(y)=0$
\end{itemize}
Similarly, 
$y\in\Cpolyf{-p}$, 
the following assertions are equivalent:
\begin{itemize}
\item $\etendu{h}(y)=0$
\item $y\in \OO^0\big(\cM,\hat{S}(T^\vee_\cM)
\otimes(\hat{S} (T_\cM\dual))^{\cotimes p}\big)[p]$
\item $\etendu{\perturbed{h}}(y)=0$
\end{itemize}
Note that $\etendu{\perturbed{h}}=
\sum_{k=0}^\infty(\etendu{h}\liederivative{\varrho})^k\etendu{h}=
(\id-\etendu{h}\liederivative{\varrho})^{-1}
\etendu{h}$.
\end{lemma}

\subsubsection{Proof of Proposition~\ref{prop:TauFormalGpd}}

In what follows, by abuse of notation, by $1$, we always denote both
$1\in C^\infty(\cM)\subset \cD(\cM)$ and
$1\in C^\infty(\cN)\subset\enveloping{\cF}$.
Similarly, by $\one$, we always denote both
the element $\one\in\jm$ and
$\one\in\jet{\cF}$ which are the morphism
of left $C^\infty(\cM)$-modules
$\cD(\cM)\ni u\mapsto u(1)\in C^\infty(\cM)$
and left $C^\infty(\cN)$-modules
$\enveloping{\cF}\ni u\mapsto u(1)\in C^\infty(\cN)$, respectively.
It is clear that
\[ \duality{\one}{u}=\epsilon(u),\quad u\in\cD(\cM) ,\]
where $\epsilon:\cD(\cM)\to C^\infty(\cM)$ is the counit
map. We have a similar identity for $\one\in\jet{\cF}$ as well.

%associated with the constant function $1\in \cR=C^\infty( \cM)$.
\begin{lemma}\label{lem:TauFormalGpd}
We have the following identities.
\begin{align}
\duality{\tauDpolyF(\xi)}{\tauDpolyF(u)} & =\tauDpolyF(\duality{\xi}{u}),
\quad\forall\xi\in\jm,\ u\in \cD(\cM) ; \label{eq:wuxi81} \\
\tauDpolyF ( u (f)) & =\tauDpolyF (u) \big( \tauDpolyF (f)\big), \quad \forall f
\in C^\infty (\cM), \ u\in \cD(\cM); \label{eq:wuxi82} \\
\tauDpolyF (f) \cdot \tauDpolyF (\xi) & =\tauDpolyF (f \cdot \xi), \quad \forall f
\in C^\infty (\cM), \ \xi\in \jm; 
\label{eq:wuxi83} \\
\tauDpolyF (1) & =1, \quad \quad \tauDpolyF (\one)=\one. \label{eq:wuxi84}
\end{align}
\end{lemma}
\begin{proof}
(1) We first show that 
\begin{equation}\label{eq:SigmaPreservePair}
\sigmaDpolyF\big(\duality{\tilde\xi \, }{\, \tilde u}\big) = \duality{\sigmaDpolyF(\tilde\xi) \, }{ \, 
\sigmaDpolyF(\tilde u)}, \qquad \forall \tilde\xi \in \jet{\cF}, \tilde u \in \enveloping{\cF}.
\end{equation}
We use the same multi-indices symbols as in the proof of Lemma~\ref{lem:tauR}:
$\check{x}^{\odot I}$, $\partial_x^{\odot J}$, $y^K$, $\partial_y^L$.
Given any $\tilde\xi \in \jet{\cF}$ and $\tilde u \in \enveloping{\cF}$,
we have local expressions
\begin{align*}
\tilde\xi & = \sum_{I,J \in \NO^{m+r}} \tilde\xi_{I,J} \cdot y^I \otimes y^J, \\
\tilde u & = \sum_{K,L \in \NO^{m+r}} \tilde u_{K,L} \cdot y^K \otimes \partial_y^L,
\end{align*}
where $\tilde\xi_{I,J}$ and $\tilde u_{K,L}$ are local functions on $\cM$. Then
\begin{align*}
\sigmaDpolyF\big(\duality{\tilde\xi \, }{\, \tilde u}\big)
&= \sum_{J \in \NO^{m+r}} (-1)^{|\tilde u_{0,J}||y^J|} \,
\tilde\xi_{0,J}\cdot\tilde u_{0,J}\cdot\duality{y^J}{\partial_y^J} \\
&= \sum_{J \in \NO^{m+r}} (-1)^{|\tilde u_{0,J}||\check{x}^{\odot J}|} \,
\tilde\xi_{0,J}\cdot\tilde u_{0,J}\cdot
\duality{(\pbw^\top)\inv(\check{x}^{\odot J})}{\pbw(\partial_x^{\odot J})} \\
&= \sum_{J,L \in \NO^{m+r}}
\duality{\tilde\xi_{0,J}\cdot(\pbw^\top)\inv(\check{x}^{\odot J}) \, }{\, \tilde u_{0,L}
\cdot\pbw(\partial_x^{\odot L})} \\
&= \duality{\sigmaDpolyF(\tilde\xi) \, }{ \, \sigmaDpolyF(\tilde u)}.
\end{align*}
This proves Equation~\eqref{eq:SigmaPreservePair}.

By Proposition~\ref{prop:IVP_TauD} and Proposition~\ref{prop:IVP_TauC},
it is clear that $\iL_{\fedosov}\big(\duality{\tauDpolyF(\xi)}{\tauDpolyF(u)}\big)=0$.
Furthermore, by Lemma~\ref{lem:IVP_TauD} and Lemma~\ref{lem:IVP_TauC},
we have $h_\natural\big(\duality{\tauDpolyF(\xi)}{\tauDpolyF(u)}\big)=0$. 
Moreover, by Equation~\eqref{eq:SigmaPreservePair},
we have $\sigmaDpolyF\big(\duality{\tauDpolyF(\xi)}
{\tauDpolyF(u)}\big)=\duality{\xi}{u}$.
Thus, by Proposition~\ref{prop:IVP_TauD} with $p=0$,
% Equation~\eqref{eq:SigmaPreservePair} implies
it follows that $\tauDpolyF(\duality{\xi}{u})
=\duality{\tauDpolyF(\xi)}{\tauDpolyF(u)}$, as desired.

(2) Note that \[ u(f)=\epsilon(u\cdot f) ,\]
where $\epsilon$ is the counit of $\cD(\cM)$,
and $f$ on the right-hand-side is considered
as a differential operator of order zero.
By Proposition~\ref{pro:Metz}, 
\[ \tauDpolyF\big(u(f)\big)
=\epsilon\big(\tauDpolyF(u)\cdot\tauDpolyF(f)\big)
=\tauDpolyF(u)\big(\tauDpolyF(f)\big) .\]
This proves Equation~\eqref{eq:wuxi82}.

%When being restricted to the
%differential operators of order~$0$ and order~$1$, the map
%$\tauDpolyF : \sD(\cM)_{\leq 1} \to \enveloping{\cF}_{\leq 1}$
%coincides with the
%morphisms $\tauTpolyF$ in \eqref{eq:Debrecen} (with a shift)
%being specialized to the case for $p=0, \ 1$.
%Hence, according to Lemma~\ref{lem:Paris},
%Equation~\eqref{eq:wuxi82} holds for
%for any $f \in C^\infty (\cM)$ and $u\in \sD(\cM)_{\leq 1}$.
%Hence we conclude that it holds for all $u\in \sD(\cM)$
%since $\tauDpolyF : \sD(\cM) \to
%\enveloping{\cF}$ is morphism of algebras by Proposition~\ref{pro:Metz}.

(3) According to Proposition~\ref{pro:Metz},
for any $f\in C^\infty (\cM)$, $\xi \in \jm$ and
$u \in \cD(\cM)$, we have
\begin{align*}
\duality{ \tauDpolyF (f) \cdot \tauDpolyF (\xi)}{\tauDpolyF (u)}
 & =\duality{ \tauDpolyF (\xi)}{\tauDpolyF (f) \cdot\tauDpolyF (u)}	
=\duality{ \tauDpolyF (\xi)}{\tauDpolyF (f\cdot u)} \\
& =\tauDpolyF(\duality{ \xi}{f\cdot u})
=\tauDpolyF(\duality{f\cdot \xi}{u})
=\duality{\tauDpolyF (f\cdot \xi)}{\tauDpolyF(u)}
\end{align*}
According to Lemma~\ref{pro:Orly}, $\enveloping{\cF}$
is the $C^\infty (\cN)$-span of $\tauDpolyF \cD(\cM)$,
it thus follows that
\[ \tauDpolyF (f) \cdot \tauDpolyF (\xi)=\tauDpolyF (f\cdot \xi). \]

(4) Note that $\tauDpolyF (1)=1$ essentially follows from
the construction. To prove the second identity,
for any $u \in \cD(\cM)$, we have
\[ \duality{\tauDpolyF(\one)}{\tauDpolyF (u)}
=\tauDpolyF(\duality{\one}{u})
%=\tauDpolyF (u(1))
=\tauDpolyF \circ \epsilon (u)
=\epsilon\circ \tauDpolyF(u)
=\duality{\one}{\tauDpolyF(u)} \]
Thus, by Lemma~\ref{pro:Orly}, it follows that $\tauDpolyF(\one)=\one$.
\end{proof}

\begin{lemma}
\label{lem:Lyon}
The following diagrams
\[ \begin{tikzcd}
\jm \ar[r,"\tauDpolyF"] \ar[d, bend right = 60, "\varepsilon"'] &
\jet{\cF} \ar[d, bend left = 60, "\varepsilon"] \\
C^\infty(\cM) \ar[r, "\tauDpolyF"] \ar[u, shift left, "\alpha"]
\ar[u, shift right, "\beta"'] & C^\infty(\cN)
\ar[u, shift left, "\alpha"] \ar[u, shift right, "\beta"']
\end{tikzcd} \]
commute.
\end{lemma}
\begin{proof}
(1) We prove that $\varepsilon\circ\tauDpolyF=\tauDpolyF\circ\varepsilon$.

%Since $\tauDpolyF$ in~\eqref{eq:Metz} maps $1\in\sD(\cM)$
%to $1\in\enveloping{\cF}$,
Using Equation~\eqref{eq:wuxi84}
for any $\xi\in\jm$, we have
%\begin{multline*}
\[ \varepsilon\circ\tauDpolyF(\xi)=\duality{\tauDpolyF(\xi)}{1}=
\duality{\tauDpolyF(\xi)}{\tauDpolyF(1)}=\tauDpolyF(\duality{\xi}{1})
%\iI_{1} (\tauDpolyF (\xi))=\iI_{ \tauDpolyF (1)}(\tauDpolyF (\xi))
%=\tauDpolyF \iI_{ 1}( \xi) =\tauDpolyF \xi (1) =
=\tauDpolyF\circ\varepsilon(\xi) .\]
%\end{multline*}

(2) We prove that $\tauDpolyF\circ\alpha=\alpha\circ\tauDpolyF$.

By Proposition~\ref{pro:Metz} and Lemma~\ref{lem:TauFormalGpd},
for any $f\in C^\infty(\cM)$ and $u\in\cD(\cM)$, we have
\[ \duality{\tauDpolyF\big(\alpha(f)\big)}{\tauDpolyF(u)}
=\tauDpolyF\big(f\cdot\epsilon(u)\big)
=\tauDpolyF(f)\cdot\epsilon\big(\tauDpolyF(u)\big) 
=\duality{\alpha\big(\tauDpolyF(f)\big)}{\tauDpolyF(u)} .\]
Thus, it follows from Lemma~\ref{pro:Orly}
that $\tauDpolyF\circ\alpha=\alpha\circ\tauDpolyF$.

(3) We prove that $\tauDpolyF\circ\beta=\beta\circ\tauDpolyF$.

By Proposition~\ref{pro:Metz} and Lemma~\ref{lem:TauFormalGpd},
for any $f\in C^\infty(\cM)$ and $u\in\cD(\cM)$, we have
\begin{align*}
\duality{(\tauDpolyF\big(\beta(f)\big)}{\tauDpolyF(u)}
& =\tauDpolyF\big(\duality{\beta(f)}{u}\big)=(-1)^{|u||f|} \,
\tauDpolyF\big(\epsilon(u \cdot f)\big) \\
& =(-1)^{|u||f|} \, \epsilon\big(\tauDpolyF(u)\cdot\tauDpolyF(f)\big)
=\duality{\beta\big(\tauDpolyF(f)\big)}{\tauDpolyF(u)}
.\end{align*} 
Thus, it follows from Lemma~\ref{pro:Orly}
that $\tauDpolyF\circ\beta=\beta\circ\tauDpolyF$.
\end{proof}

\begin{lemma}
\label{lem:nablag}
For all $u\in\cD(\cM)$ and $\xi\in\jm$,
\begin{equation}\label{eq:nablag}
 \nablag_{\tauDpolyF(u)}\big(\tauDpolyF(\xi)\big)
=\tauDpolyF\big(\nablag_{u}\xi\big).
\end{equation}
\end{lemma}
\begin{proof}
Since $\tauDpolyF$ preserves the algebra structures of $\cD(\cM)$ and $\enveloping{\cF}$, it suffices to verify Equation~\eqref{eq:nablag} for the case $u=X \in \sections{T_\cM}$. 

By Proposition~\ref{pro:Metz} and Lemma~\ref{lem:TauFormalGpd}, for any $v\in \cD(\cM)$, we have
\[ \begin{split}
 \duality{\nablag_{\tauDpolyF(X)}\big(\tauDpolyF(\xi)\big)}{\tauDpolyF(v)} 
=& (\tauDpolyF(X))\big(\duality{\tauDpolyF(\xi)}{\tauDpolyF(v)} \big) - (-1)^{|X||\xi|}
\duality{\tauDpolyF(\xi)}{\tauDpolyF(X) \cdot \tauDpolyF(v)} \\
=& \tauDpolyF\big(X(\duality{\xi}{v})\big) -(-1)^{|X||\xi|}
\tauDpolyF(\duality{\xi}{X \cdot v}) \\
=& \tauDpolyF\big(\duality{\nablag_{X}\xi}{v} \big) = \duality{\tauDpolyF\big(\nablag_{X}\xi\big)}{\tauDpolyF(v)}.
\end{split} \]
Since $\enveloping{\cF}$
is the $C^\infty(\cN)$-span of $\tauDpolyF\sD(\cM)$
according to Lemma~\ref{pro:Orly}, it thus follows that
$\nablag_{\tauDpolyF (X)} \big(\tauDpolyF (\xi) \big)=\tauDpolyF \big(\nablag_{X} \xi\big)$.
\end{proof}

%\ping{Introduce the notation $\cdota$ and $\cdotb$}
%
%\begin{lemma}
%For any $u\in\sD(\cM)$ and $\xi\in\jm$,
%\begin{gather}
%(\tauDpolyF (u))\cdota (\tauDpolyF (\xi) )=\tauDpolyF \big( u\cdota \xi\big) \label{eq:cdota} \\
%\red{(\tauDpolyF (u))\cdotb (\tauDpolyF (\xi) )=\tauDpolyF \big( u\cdotb \xi\big) \qquad \text{NO NEED}}
%\label{eq:cdotb}
%\end{gather}
%\end{lemma}
%
%\begin{proof}
%For any $f\in C^\infty(\cM)$,
%we have
%$(\tauDpolyF f)\cdota(\tauDpolyF\xi)=(\tauDpolyF f)\cdot(\tauDpolyF\xi)
%=\tauDpolyF(f\cdot\xi)=\tauDpolyF\big(f\cdota\xi\big)$
%according to \eqref{eq:wuxi83}.
%Thus Equation~\eqref{eq:cdota}
%follows immedediately from Lemma~\ref{lem:nablag}
%and the fact that $\tauDpolyF:\sD(\cM)\to\enveloping{\cF}$
%is an algebra morphism.
%
%To prove Equation~\eqref{eq:cdotb}, for any $x\in\sD(\cM)$, we have
%\begin{multline*}
%\duality{(\tauDpolyF(u))\cdotb(\tauDpolyF(\xi))}{\tauDpolyF(x)}
%=\duality{\tauDpolyF(\xi)}{\tauDpolyF(x)\tauDpolyF(u)}
%=\duality{\tauDpolyF\xi}{\tauDpolyF(x u)} \\
%=\tauDpolyF\duality{\xi}{x u}=\tauDpolyF\duality{u\cdotb\xi}{x}
%=\duality{\tauDpolyF(u\cdotb\xi)}{\tauDpolyF x}
%.\end{multline*}
%Thus Equation~\eqref{eq:cdotb} follows,
%since $\enveloping{\cF}$
%is the $C^\infty(\cN)$-span of $\tauDpolyF\sD(\cM)$
%according to Lemma~\ref{pro:Orly}.
%\end{proof}

As earlier, let $\cR = C^\infty(\cM)$ and $\cNR = C^\infty(\cN)$.
Now we are ready to compelete the proof of Proposition~\ref{prop:TauFormalGpd}.

\begin{proof}[Proof of Proposition~\ref{prop:TauFormalGpd}]
From Lemma~\ref{lem:Lyon}, we already know that $\tauDpolyF$ respects
the source, target and unit morphisms.
Now we prove that $\tauDpolyF$ is a morphism of algebras.
According to Proposition~\ref{pro:Metz}, the map
$\tauDpolyF:\cD(\cM)\to\enveloping{\cF}$
respects the coalgebra structures. 
%That is,
%for any $x\in\sD(\cM)$,
%\[ \Delta(\tauDpolyF(x))=(\tauDpolyF\otimes\tauDpolyF)\Delta(x) .\]
Hence, for any $u \in \cD(\cM)$ and $\xi,\eta\in\jm$,
\begin{align*}
\duality{\tauDpolyF(\xi)\tauDpolyF(\eta)}{\tauDpolyF(u)}
& =\duality{\tauDpolyF(\xi)\otimes\tauDpolyF(\eta)}{\Delta(\tauDpolyF(u))} \\
& =\duality{(\tauDpolyF\otimes\tauDpolyF)(\xi\otimes\eta)}{(\tauDpolyF\otimes\tauDpolyF)\Delta(u)} \\
& =\tauDpolyF\duality{\xi\otimes\eta}{\Delta(u)}
=\duality{\tauDpolyF(\xi\eta)}{\tauDpolyF(u)}
.\end{align*}
It thus follows that
\[ \tauDpolyF(\xi)\tauDpolyF(\eta)=\tauDpolyF(\xi\eta) ,\]
since $\enveloping{\cF}$ is the $\cNR$-span of $\tauDpolyF\cD(\cM)$
according to Lemma~\ref{pro:Orly}. Hence, the injection
$\tauDpolyF:\jm\to\jet{\cF}$ is a morphism of algebras.

Next, we show that $\tauDpolyF$ is a morphism of coalgebras.

First of all, we claim that
$\enveloping{\cF}\otimes_{\cNR}^{rl}\enveloping{\cF}$
is the $\cNR$-span of
$\tauDpolyF\cD(\cM)\otimes_{\cR}^{rl}\tauDpolyF\cD(\cM)$.

By Lemma~\ref{pro:Orly}, any element in $\enveloping{\cF}$
can be expressed as $\sum_i f_i\cdot\tauDpolyF(u_i)$
for some $f_i\in\cNR$ and $u_i\in\cD(\cM)$, and thus any element
in $\enveloping{\cF}\otimes_{\cNR}^{rl}\enveloping{\cF}$ can be written as
\[ \sum_j\tilde{v}_j\otimes^{rl}
\Big(\sum_i f_{i,j}\cdot\tauDpolyF(u_{i,j})\Big)
=\sum_{i,j} (\tilde{v}_j\cdot f_{i,j})\otimes^{rl}\tauDpolyF(u_{i,j})
\in\enveloping{\cF}\otimes_{\cR}^{rl}\tauDpolyF\cD(\cM) ,\]
where $\tilde{v}_j\in\enveloping{\cF}$, $f_{i,j}\in\cNR$
and $u_{i,j}\in\cD(\cM)$. Thus, we know that
$\enveloping{\cF}\otimes_{\cNR}^{rl}\enveloping{\cF}$
is the $\cNR$-span of $\enveloping{\cF}\otimes_{\cR}^{rl}\tauDpolyF\sD(\cM)$.
Then, applying Lemma~\ref{pro:Orly} to the left $\enveloping{\cF}$,
we prove the claim: $\enveloping{\cF}\otimes_{\cNR}^{rl}\enveloping{\cF}$
is the $\cNR$-span of $\tauDpolyF\cD(\cM)\otimes_{\cR}^{rl}\tauDpolyF\cD(\cM)$.

Since $\tauDpolyF:\cD(\cM)\to\enveloping{\cF}$ respects the multiplications,
for any $\xi\in\jm$ and $u,v\in\cD(\cM)$, we have
\begin{align*}
\duality{\Delta(\tauDpolyF(\xi))}{\tauDpolyF(u)\otimes^{rl}\tauDpolyF(v)}
&= \duality{\tauDpolyF(\xi)}{\tauDpolyF(u)\tauDpolyF(v)} 
=\duality{\tauDpolyF(\xi)}{\tauDpolyF(u v)} \\
&= \tauDpolyF\big(\duality{\xi}{uv}\big) 
=\tauDpolyF\big(\duality{\Delta\xi}{u\otimes^{rl} v}\big) \\
&= \duality{(\tauDpolyF\otimes\tauDpolyF)\big(\Delta(\xi)\big)}
{\tauDpolyF(u)\otimes^{rl}\tauDpolyF(v)}
.\end{align*}
The last equality follows from Lemma~\ref{lem:TauFormalGpd},
Proposition~\ref{pro:Metz}, and Equation~\eqref{eq:HKG}: 
\[ \duality{\xi'\otimes^{\beta\alpha}\eta' \, }{\, u \otimes^{rl} v}
= (-1)^{|u||\eta'|}\duality{\xi'}{u\cdot\duality{\eta'}{v}} .\]

Since $\enveloping{\cF}\otimes_{\cNR}^{rl}\enveloping{\cF}$
is the $\cNR$-span of $\tauDpolyF\cD(\cM)\otimes_{\cR}^{rl}\tauDpolyF\cD(\cM)$,
it follows that
\[ \Delta(\tauDpolyF(\xi))=(\tauDpolyF\otimes\tauDpolyF)\Delta(\xi) .\]
Hence $\tauDpolyF:\jm\to\jet{\cF}$ respects the coproducts.

Finally, we need to prove that $\tauDpolyF$ respects the antipodes. 
By Equation \eqref{eq:antipode}, Lemma~\ref{lem:TauFormalGpd}
and Lemma~\ref{lem:nablag}, for any $\xi\in\jm$ and $u\in \cD(\cM)$, we have
\begin{align*}
\duality{S\tauDpolyF(\xi)}{\tauDpolyF(u)}
& = (-1)^{|\xi||u|} \duality{\nablag_{\tauDpolyF(u)} \tauDpolyF(\xi)}{1}
=\duality{\tauDpolyF(\nablag_u \xi )}{\tauDpolyF(1)} \\
& =\tauDpolyF\big(\duality{\nablag_u \xi}{1} \big)
=\tauDpolyF\big(\duality{S\xi}{u} \big)
=\duality{\tauDpolyF \big( S(\xi)\big)}{\tauDpolyF(u)}.
\end{align*}
Hence, it follows from` Lemma~\ref{pro:Orly} that
\[ S\tauDpolyF(\xi)=\tauDpolyF S(\xi). \]
Thus $\tauDpolyF:\jm\to\jet{\cF}$ respects the antipodes.

This concludes the proof.
\end{proof}

Now let us recall the following main theorem in~\cite{paper-zero}.

\begin{theorem}\cite{paper-zero}
\label{thm:paper-zero}
Let $\cL_i\to \cM_i$, $i=1, 2$, be dg Lie algebroids.
Assume that we have the following
\begin{itemize}
\item there is a morphism of dg formal groupoids
\[ \breve{\tau}: \jet{\cL_1} \to \jet{\cL_2} \]
\item there is a morphism of dg bialgebroids
\[ \breve{\tau}: \enveloping{\cL_1} \to \enveloping{\cL_2} \]
\item there is a morphism of dg algebras
\[ \breve{\tau}: C^\infty(\cM_1)\to C^\infty(\cM_2) \]
\end{itemize}
such that
\[ \breve{\tau}\pairing{\xi}{u}=\pairing{\breve{\tau}\xi}{\breve{\tau}u},
\qquad\forall\xi\in\jet{\cL_1}\quad\text{and}\quad u\in\enveloping{\cL_1} \]
Then $\breve{\tau}$ preserves the Gerstenhaber brackets,
cup products and the calculus operations $\iI$, $\iL$ and $B$,
an in particular induces a morphism of calculus from 
$\calculus_H(\cL_1,\cQ_1)$ to $\calculus_H(\cL_2,\cQ_2)$.
\end{theorem}

Since the injections $\tauDpolyF$ in~\eqref{eq:Metz}
and~\eqref{eq:FcontractionJets} are degree-preserving maps
which respect (i) the dg Hopf algebroid structures
(by Proposition~\ref{pro:Metz}), (ii) the dg formal groupoid structures
(by Proposition~\ref{prop:TauFormalGpd}) and (iii) the pairing
(by Equation~\eqref{eq:wuxi81}),
according to Theorem~\ref{thm:paper-zero}, we have the following 

\begin{corollary}\label{cor:TauPreserveCalH}
The injections $\tauDpolyF$ in Corollary~\ref{cor:repositoryD}
preserve the Gerstenhaber brackets, cup products
and the calculus operations $\iI$, $\iL$ and $B$.
\end{corollary}

\subsection{Main results}

We are ready to prove the main result of this section.

\begin{theorem}\label{thm:Germain}
For any graded manifold $\cM$, let $\cF\to\cN$ be
its Fedosov dg Lie algebroid corresponding
to a torsion-free affine connection $\nabla$ on $\cM$. 
\begin{enumerate}
\item The pair of injections $\tauTpolyF$ in Corollary~\ref{pro:contractionTO}
respect all the algebraic operations defining the calculi
$\calculus_C(\cM)$ to $\calculus_C (\cF,\fedosov)$;
\item The pair of injections $\tauDpolyF$ in Corollary~\ref{cor:repositoryD}
respect all the algebraic operations defining the calculi 
$\calculus_H(\cM)$ and $\calculus_H(\cF,\fedosov)$.
Consequently, it induces, on the cohomology level,
a morphism of calculi from $\calculus_H(\cM)$ to $\calculus_H(\cF,\fedosov)$.
\end{enumerate}
\end{theorem}

\begin{lemma}\label{lem:Paris}
The pair of injections $\tauTpolyF$ in~\eqref{Fortune} and~\eqref{Pohang}
induces a morphism of dg Lie--Rinehart algebras
from $\big(C^\infty(\cM),\Gamma(\cM,T_\cM),0\big)$
to $\big(C^\infty(\cN),\Gamma(\cN,\cF),\tQF\big)$.
\end{lemma}
\begin{proof}
%By construction \eqref{eq:Metz}, we have
By Proposition~\ref{pro:Jolly}, we have
\begin{equation}\label{eq:SFO}
\tauDpolyF : \cD(\cM)_{\leq 1} \to \enveloping{\cF}_{\leq 1}.
\end{equation}
Note that we have the decompositions
\[ \cD(\cM)_{\leq 1}\cong C^\infty(\cM)\oplus \Gamma (\cM; T_\cM)
\qquad\text{and}\qquad
\enveloping{\cF}_{\leq 1}\cong C^\infty(\cN)\oplus \Gamma (\cN; \cF).\]
Again according to Proposition~\ref{pro:Jolly}, under such decompositions,
$\tauDpolyF$ in~\eqref{eq:SFO} indeed coincides with the injections
$\tauTpolyF$ in~\eqref{Fortune} and~\eqref{Pohang}. 
Since, by Proposition~\ref{pro:Metz},
the injection $\tauDpolyF:\cD(\cM)\to\enveloping{\cF}$
is morphism of algebras, it follows immediately
that the injections $\tauTpolyF$ in~\eqref{Fortune} and~\eqref{Pohang}
constitutes a morphism of dg Lie--Rinehart algebras.
\end{proof}

As an immediate consequence, we have the following

\begin{corollary}\label{cor:SFO}
The map $\tauTpolyF$ in the contraction \eqref{eq:Tcontraction}
is a morphism of differential graded Lie algebras.
\end{corollary}

%\begin{proposition}
%The map $\tauDpolyF$ in the contraction \eqref{eq:Dcontraction1}
%in Corollary~\ref{cor:repositoryD}
%is a morphism of differential graded Lie algebras.
%\end{proposition}
\begin{proof}[Proof of Theorem \ref{thm:Germain}]
The first assertion follows from Corollary~\ref{cor:SFO}
and Proposition~\ref{pro:Paris}. The second assertion
is an immediate consequence of Corollary~\ref{cor:TauPreserveCalH}.
This concludes the proof.
%
%For Assertion (2), note that
%The Gerstenhaber brackets on
%$\totDpolyM{\bullet}$ and $\totDpolyF{\bullet}$ are completely determined
%by the algebra an coalgebra strutures on $\sD(\cM)$ and
%$\enveloping{\cF}$, respectively, according to~\ref{!!!!!}. Since the morphism
%$\tauDpolyF$ in \eqref{eq:Dcontraction1} respects the
%tensor products according to Proposition~\ref{pro:sigmainitialD},
%from Proposition~\ref{pro:Metz}, it thus follows immediately
%that $\tauDpolyF$ in \eqref{eq:Dcontraction1}
%%appearing in Proposition~\ref{repository}
%is a morphism of graded Lie algebras.
%%The map $\tauDpolyF$ clearly respects the differentials since it is a cochain map.
%From Proposition \ref{pro:sigmainitialD}, we already know that
%$\tauDpolyF$ respects the rest of the calculus structure. 
% This concludes the proof.
\end{proof}	

Note that one also have the following results for the tensor fields:

\begin{proposition}\label{prop:TauTensor}
Let $X\in\sections{T_\cM}$, $\cY\in\sections{(T_\cM)^{\otimes k}
\otimes(T_\cM\dual)^{\otimes l}}$ and $\cY'\in\sections{(T_\cM)^{\otimes k'}
\otimes(T_\cM\dual)^{\otimes l'}}$. The injections in~\eqref{eq:Tkl}
respects the tensor product and Lie derivatives:
\begin{align}
\tauTensorF{k+k'}{l+l'}(\cY\otimes\cY')
&= \tauTensorF{k}{l}(\cY)\otimes\tauTensorF{k'}{l'}(\cY'), \label{eq:TauKLTensor} \\ 
\tauTensorF{k}{l}\big(\iL_X(\cY)\big)
&= \iL_{\tauTpolyF(X)}\big(\tauTensorF{k}{l}(\cY)\big). \label{eq:TauTensor_Lie}
\end{align}
\end{proposition}

\begin{proof}
The proof of Equation~\eqref{eq:TauKLTensor} is virtually identical
to the proof of Proposition~\ref{prop:TauDiffOpTensor}. 

Due to Equation~\eqref{eq:TauKLTensor},
the proof of Equation~\eqref{eq:TauTensor_Lie} can be reduced
to prove it in the three special cases:
(i) $k=l=0$; (ii) $k=1$ and $l=0$; (iii) $k=0$ and $l=1$.
Case~(i) and Case~(iii) follow from Proposition~\ref{pro:Paris}~(3)
specializing to the cases $n=0$ and $n=1$,
and Case~(ii) follows from Lemma~\ref{lem:Paris}.
This completes the proof.
\end{proof}

\section{Fedosov contractions for dg manifolds}
\label{sect:5}

Now we are ready to move to the case of a \emph{dg} manifold $(\cM,Q)$,
which is the main object of study in the present paper.
Let $\cF\to\cN$ be the Fedosov dg Lie algebroid
over the Fedosov dg manifold $(\cN,\fedosov)$
associated with the graded manifold $\cM$
and a torsion free affine connection $\nabla$ on $\cM$.
As an immediate consequence of Corollary~\ref{cor:SFO},
% Proposition~\ref{Mandalay} and Lemma~\ref{thm:Germain},
we have the following

\begin{proposition}
$(\cN,\fedosov+\tauTpolyF(Q))$ is a dg manifold.
\end{proposition}

By $\cNQ$, we denote the dg manifold $(\cN,\fedosov+\tauTpolyF(Q))$.
We write $\cFQ$ to denote the dg manifold structure on the graded manifold $\cF$
characterized by the following property: $\cFQ\to\cNQ$ is a dg vector space
such that the induced operator on $\sections{\cN;\cF}$ determined
by the homological vector fields of $\cFQ$ and $\cNQ$
is $\iL_{\fedosov+\tauTpolyF(Q)}$.

\begin{lemma}\label{lem:Florence}
The dg vector bundle $\cFQ\to\cNQ$ is a dg Lie subalgebroid
of the tangent dg Lie algebroid $T_{\cNQ}\to\cNQ$.
\end{lemma}

In other words, $\cFQ$ is a dg foliation of the dg manifold
$\cNQ=(\cN,\fedosov+\tauTpolyF(Q))$.
Such a dg Lie algebroid $\cFQ\to\cNQ$ is called a \emph{Fedosov dg Lie algebroid}
associated with the dg manifold $(\cM,Q)$. 
Using the same constructions in Section~\ref{sec:CalCFedosov}
and Section~\ref{sec:CalHFedosov}
with replacing the differential $\fedosov$ by $\fedosov+\tauTpolyF(Q)$,
we obtain two calculi $\calculus_C(\cF,\fedosov+\tauTpolyF(Q))$
and $\calculus_H(\cF,\fedosov+\tauTpolyF(Q))$.
They can be considered essentially as the calculi
associated with the dg Lie algebroid $\cFQ\to\cNQ$.

\subsection{Algebraic lemmas}

We need a few algebraic lemmas in order to study
the dg version of Fedosov contractions. 

Let $C^{\bullet,\bullet}$ be a double complex with the differentials 
\[ \dD:C^{p,q}\to C^{p+1,q}
\qquad\text{and}\qquad
\varrho:C^{p,q}\to C^{p,q+1} \]
such that
\begin{equation}\label{eq:BiCxContractionAssumption}
C^{p,q} =0, \qquad \forall \, p<0.
\end{equation}
Here, we require that $(\dD+\varrho)^2 =0$ for the double complex
$(C^{\bullet,\bullet},\dD,\varrho)$. 
Assume that
\begin{equation}\label{eq:ContractionGeneral}
\begin{tikzcd}[cramped]
\big(B^\bullet,d\big) \arrow[r, "\tau", shift left]
& \big(\tot_\oplus^\bullet C, \, \dD\big)
\arrow[l, "\sigma", shift left]
\arrow[loop, "h", out=7, in=-7, looseness=3]
\end{tikzcd}
\end{equation}
is a contraction satisfying the following condition:
\[ h\big(C^{p,q}\big) \subset C^{p-1,q}, \qquad\forall p,q \in\ZZ .\]
Consider
\[ \tot^n_\oplus C = \bigoplus_{\substack{p+q=n \\ p,q\in\ZZ}} C^{p,q} .\]

\begin{lemma}\label{lem:BiCxFiltration}
The filtration 
\[ F_m(\tot^n_\oplus C) = \bigoplus_{\substack{p+q=n, \, p,q\in\ZZ
\\ 0\leq p\leq m}} C^{p,q} \]
is a complete exhaustive filtration of the direct-sum total complex
$(\tot^\bullet_\oplus C,\dD+\varrho)$ such that 
\[ (\varrho h)\big(F_m(\tot^n_\oplus C) \big)\subset F_{m-1}(\tot^n_\oplus C),
\qquad\forall\, m,n \in\ZZ .\]
In particular, $\varrho$ is a small perturbation of the contraction
\eqref{eq:ContractionGeneral} in the sense of Appendix~\ref{Vilnius}. 
\end{lemma}

\begin{proof}
It is clear that the filtration is exhaustive:
$\bigcup_{m} F_m (\tot^n_\oplus C) = \tot^n_\oplus C$.
Furthermore, by the assumption \eqref{eq:BiCxContractionAssumption},
we have $F_m(\tot^n_\oplus C)=\{0\}$ for all $m<0$,
and thus the filtration is complete:
$\displaystyle\varprojlim_{m\to -\infty}
\dfrac{\tot^n_\oplus C}{F_m(\tot^n_\oplus C)}$. 
The smallness follows from a straightforward computation.
\end{proof}

As a consequence, we have a perturbed contraction: 
\begin{equation}\label{eq:ContractionPerturbed}
\begin{tikzcd}[cramped]
\big(B^\bullet,d_\varrho\big)
\arrow[r, "\tau_\varrho ", shift left] &
\big(\tot_\oplus^\bullet C, \, \dD +\varrho\big)
\arrow[l, "\sigma_\varrho ", shift left]
\arrow[loop, "h_\varrho ", out=7, in=-7, looseness=3]
\end{tikzcd}
\end{equation}
%by Lemma~\ref{Riga}.

\begin{lemma}\label{lem:BiCxPerturbation}
If the contraction \eqref{eq:ContractionGeneral} satisfies the extra condition
\[ \tau(B^q)\subset C^{0,q}, \qquad\forall q\in\ZZ ,\]
then
\[ \tau_\varrho=\tau \qquad\text{and}\qquad d_\varrho=d+\sigma\varrho\tau .\]
\end{lemma}
\begin{proof}
It follows from Lemma~\ref{Riga} and the observation that $(h\varrho)\tau=0$.
\end{proof}

The situation described in Lemma~\ref{lem:BiCxPerturbation}
can be visualized by the following diagram:
\begin{equation}\label{eq:AlgLemBiCx}
    \begin{tikzcd}%[column sep=small, row sep = small]
\vdots & & \vdots & \vdots & \vdots & \\
B^{q+1} \ar[u,"d"] \ar[rr,dashed,bend left=20,"\tau" description] & 0 \ar[r]
& \ar[l,dashed,bend left,"h" description] C^{0,q+1} \ar[u,"\varrho"] \ar[r,"\dD"]
& \ar[l,dashed,bend left,"h" description] C^{1,q+1} \ar[u,"\varrho"] \ar[r,"\dD"]
& \ar[l,dashed,bend left,"h" description]
C^{2,q+1} \ar[u,"\varrho"] \ar[r,"\dD"] & \ar[l,dashed,bend left,"h" description] \cdots \\
B^q \ar[u,"d"] \ar[rr,dashed,bend left=20,"\tau" description] & 0 \ar[r]
& \ar[l,dashed,bend left,"h" description] C^{0,q} \ar[u,"\varrho"] \ar[r,"\dD"]
& \ar[l,dashed,bend left,"h" description] C^{1,q} \ar[u,"\varrho"] \ar[r,"\dD"]
& \ar[l,dashed,bend left,"h" description] C^{2,q} \ar[u,"\varrho"] \ar[r,"\dD"]
& \ar[l,dashed,bend left,"h" description] \cdots \\
\vdots \ar[u,"d"] & & \vdots \ar[u,"\varrho"] & \vdots \ar[u,"\varrho"]
& \vdots \ar[u,"\varrho"] & 
\end{tikzcd}
\end{equation}

\subsection{Fedosov contractions for dg manifolds}
\label{sec:FedosovDGmfd}

%In order to obtain the explicit Fedosov contractions for dg manifolds, we need a couple of algebraic lemmas about contractions and double complexes. 

As in the case of graded manifolds, we consider five types
of Fedosov contractions for a dg manifold: tensor fields,
polyvector fields, differential forms, polydifferential operators and polyjets. 

\subsubsection{Fedosov contractions for tensor fields on a dg manifold}

%Let $(\cM,Q)$ be a dg manifold. By the injection map in the contraction \eqref{eq:Tcontraction0}, we have 
%$$
%\breve{\tau}_\natural(Q) \in \sT^{0,1,1}.
%$$

The space of tensor fields of type $(k,l)$, $k\geq 0$, $l\geq 0$,
on $\cN$ tangent to $\cF$ is equipped with the bigrading:
\[ \sections{\cN;\cF^{\otimes k}\otimes(\cF^\vee)^{\otimes l}}^{r,q}
= \Omega^r(\cM,\hat{S} T_\cM\dual\otimes
T_\cM^{\otimes k}\otimes(T_\cM\dual)^{\otimes l})^{q+r} .\]
Its direct-sum total space is 
\[ \tot_\oplus^n \big(\Gamma(\cN;\cF^{\otimes k}\otimes(\cF^\vee)^{\otimes l})\big)
=\bigoplus_{q+r=n}\Omega^r(\cM,\hat{S}T_\cM\dual\otimes T_\cM^{\otimes k}
\otimes(T_\cM\dual)^{\otimes l})^{q+r} .\]

\begin{proposition}\label{origin}
Given a dg manifold $(\cM,Q)$ and a torsion-free affine connection $\nabla$ on $\cM$,
let $\cF\to\cN$ be the Fedosov dg Lie algebroid corresponding to the graded manifold $\cM$
as in Lemma~\ref{lem:Rome}
and let $\fedosov$ be the corresponding Fedosov homological vector field on $\cN$.
Then there are contractions
at the level of tensor fields of any type $(k,l)$, $k \geq 0$, $l \geq 0$:
\[ \begin{tikzcd}[cramped]
\Big(\big(\sections{\cM;T_\cM^{\otimes k}
\otimes(T_\cM^\vee)^{\otimes l}}\big)^\bullet,\iL_Q\Big)
\arrow[r, "\tauTensorQ{k}{l}", shift left] &
\Big(\tot_\oplus^\bullet\big(\Gamma(\cN;\cF^{\otimes k}\otimes
(\cF^\vee)^{\otimes l})\big), \iL_{\fedosov+\tauTpolyF(Q)}\Big)
\arrow[l, "\sigmaTensorQ{k}{l}", shift left]
\arrow[loop, "\hTensorQ{k}{l}", out=3, in=-3, looseness=3]
\end{tikzcd} \]
%\end{enumerate}
\end{proposition}
\begin{proof}
Recall that
\[ \Gamma(\cN;\cF^{\otimes k}\otimes(\cF^\vee)^{\otimes l})
=\Omega(\cM,\hat{S}T_\cM\dual\otimes T_\cM^{\otimes k}\otimes (T_\cM\dual)^{\otimes l}) .\]
Applying Lemma~\ref{lem:BiCxFiltration} and Lemma~\ref{lem:BiCxPerturbation} to the case
\begin{gather*}
B^q=\Gamma(\cM; T_\cM^{\otimes k}\otimes
(T_\cM^\vee)^{\otimes l})^q, \qquad C^{r,q}
=\Omega^r(\cM,\hat{S}T_\cM\dual\otimes T_\cM^{\otimes k}
\otimes (T_\cM\dual)^{\otimes l})^{q+r}, \\
D=\iL_{\fedosov}, \quad\varrho=\iL_{\tauTpolyF(Q)},
\quad h=\hTensorF{k}{l}, \quad\tau=\tauTensorF{k}{l},
\quad\sigma=\sigmaTensorF{k}{l}, \quad d=0
\end{gather*}
--- see Proposition~\ref{repository} for the notations --- we obtain a contraction:
\[ \begin{tikzcd}[cramped]
\Big(\big(\sections{\cM;T_\cM^{\otimes k}
\otimes(T_\cM^\vee)^{\otimes l}}\big)^\bullet, d_\varrho \Big)
\arrow[r, "\tauTensorQ{k}{l}", shift left] &
\Big(\tot_\oplus^\bullet\big(\Gamma(\cN;\cF^{\otimes k}\otimes
(\cF^\vee)^{\otimes l})\big), \iL_{\fedosov+\tauTpolyF(Q)}\Big)
\arrow[l, "\sigmaTensorQ{k}{l}", shift left]
\arrow[loop, "\hTensorQ{k}{l}", out=3, in=-3, looseness=3]
\end{tikzcd} \]

It remains to show that $d_\varrho = \iL_Q$. Now we compute: 
\begin{align*}
d_\varrho(\omegaa) & = (d + \sigma \varrho \tau)(\omegaa) \\
& = \sigmaTensorF{k}{l} \iL_{\tauTpolyF(Q)} \tauTensorF{k}{l}(\omegaa) \\
& = \sigmaTensorF{k}{l} \tauTensorF{k}{l}(\iL_{Q} ( \omegaa)) \qquad 
\text{(by Proposition~\ref{prop:TauTensor})} \\
& = \iL_{Q} (\omegaa),
\end{align*}
for any $\omegaa \in \Gamma(\cN;\cF^{\otimes k}\otimes
(\cF^\vee)^{\otimes l})$.
\end{proof}

\begin{remark}
The injections $\tauTensorQ{k}{l}$ in Proposition~\ref{repository}
and $\tauTensorQ{k}{l}$ in Proposition~\ref{origin} are actually identical.
However, the homological perturbation modifies the maps
$\sigmaTensorF{k}{l}$ and $\hTensorF{k}{l}$ in Proposition~\ref{repository}
and returns the new maps $\sigmaTensorQ{k}{l}$ and $\hTensorQ{k}{l}$
appeared in Proposition~\ref{origin}.
\end{remark}

%\begin{remark}
%Writing $\big(\sections{\cN;\cF^{\otimes k}\otimes(\cF^\vee)^{\otimes l}}\big)^n$
%in Proposition~\ref{origin} is an abuse of notation.
%In fact, to ensure the convergence of the series involved in the homological perturbation,
%$\big(\sections{\cN;\cF^{\otimes k}\otimes(\cF^\vee)^{\otimes l}}\big)^n$ must be replaced by its subspace
%\[ \big(\bigoplus_{r\in\ZZ_{\geq 0}}\sections{\Lambda^r(T\dual_{\cM})\otimes\SM\otimes T_{\cM}^{\otimes k}\otimes
%(T_{\cM}\dual)^{\otimes l}}\big)^n .\]
%\end{remark}

\subsubsection{Fedosov contractions for $\sT_{\poly}$ and $\sA_{\poly}$ of a dg manifold}

By perturbing the contractions in Corollary~\ref{pro:contractionTO}
by $\schouten{\tauTpolyF(Q)}{\argument}$ and $\iL_{\tauTpolyF(Q)}$,
respectively, we obtain the following

\begin{theorem}\label{thm:mainT}
Given a dg manifold $(\cM,Q)$ and a torsion-free affine connection $\nabla$ on
$\cM$, let $\cF\to\cN$ be the Fedosov dg Lie algebroid corresponding to the
graded manifold $\cM$ as in Lemma~\ref{lem:Rome}
and $\fedosov$ the corresponding Fedosov homological vector field on $\cN$.
Then we have contractions
\[ \begin{tikzcd}[cramped]
\Big(\totTpolyM{\bullet},\schouten{Q}{\argument}\Big)
\arrow[r, "\tauTpolyQ", shift left]
& \Big(\totTpolyF{\bullet},\schouten{\fedosov+\tauTpolyF(Q)}{\argument}\Big)
\arrow[l, "\sigmaTpolyQ", shift left]
\arrow[loop, "\hTpolyQ", out=5, in=-5, looseness=3]
\end{tikzcd} \]
and
\[ \begin{tikzcd}[cramped]
\Big(\totApolyM{\bullet},\iL_{Q}\Big)
\arrow[r, "\tauTpolyQ", shift left] &
\Big(\totcApolyF{\bullet},\iL_{\fedosov+\tauTpolyF(Q)}\Big)
\arrow[l, "\sigmaTpolyQ", shift left]
\arrow[loop, "\hTpolyQ", out=5, in=-5, looseness=3]
\end{tikzcd} \]
such that, the pair of injections $\tauTpolyQ$ respect the operations
$\wedge$, $\schouten{\argument}{\argument}$,
$\iI$, $\iL$ and $d$.
\end{theorem}

\begin{proof}
We follow the notations in Section~\ref{sec:Subcontraction_TA}.
Apply Lemma~\ref{lem:BiCxFiltration} and Lemma~\ref{lem:BiCxPerturbation}
to the contractions \eqref{eq:Tcontraction} and \eqref{eq:Ocontraction} with 
\begin{gather*}
B^s=\bigoplus_{p+q=s}\sT^{0,q,p}, \qquad C^{r,s}=\bigoplus_{p+q = s}
\prescript{\cF}{}{\sT}^{r,q,p}, \\
D=\tQF, \quad\varrho=\iL_{\tauTpolyF(Q)} =\schouten{\tauTpolyF(Q)}{\argument},
\quad h=\hTpolyF, \quad\tau=\tauTpolyF, \quad\sigma=\sigmaTpolyF, \quad d=0,
\end{gather*}
and 
\begin{gather*}
B^s=\prod_{-p+q=s}\sA^{0,q,-p}, \qquad
C^{r,s}=\prod_{-p+q = s}\prescript{\cF}{}{\sA}^{r,q,-p}, \\
D=\tQF, \quad\varrho=\iL_{\tauTpolyF(Q)}, \quad h=\hTpolyF,
\quad\tau=\tauTpolyF, \quad\sigma=\sigmaTpolyF, \quad d=0,
\end{gather*}
respectively. We have double complexes here since 
\[ [D,\varrho] = [\tQF,\iL_{\tauTpolyF(Q)}] = \iL_{[\fedosov,\tauTpolyF(Q)]} =0 .\]

Similarly to Proposition~\ref{origin}, one can show that
$d_\varrho=\iL_Q$ by Proposition~\ref{prop:TauTensor}.

Since the pair of injection maps $\tauDpolyQ$
remains the same after the perturbation
according to Lemma~\ref{lem:BiCxPerturbation},
it preserves the calculus structures by Theorem~\ref{thm:Germain}.
This completes the proof.
\end{proof}

As an immediate consequence, we have
\begin{corollary}
The pair of maps $\tauTpolyQ$ (and hence $\sigmaTpolyQ$) below:
\[ \begin{tikzcd}[cramped]
\cohomology{} \Big(\totTpolyM{\bullet},\schouten{Q}{\argument}\Big)
\arrow[r, "\tauTpolyQ", shift left]
& \cohomology{} \Big(\totTpolyF{\bullet},\schouten{\fedosov+\tauTpolyF(Q)}{\argument}\Big)
\arrow[l, "\sigmaTpolyQ", shift left]
%\arrow[loop, "\hTpolyQ", out=7, in=-7, looseness=3]
\end{tikzcd} \]
and
\[ \begin{tikzcd}[cramped]
\cohomology{} \Big(\totApolyM{\bullet},\iL_{Q}\Big)
\arrow[r, "\tauTpolyQ", shift left]
& \cohomology{} \Big(\totcApolyF{\bullet},\iL_{\fedosov+\tauTpolyF(Q)}\Big)
\arrow[l, "\sigmaTpolyQ", shift left]
%\arrow[loop, "\hTpolyQ", out=7, in=-7, looseness=3]
\end{tikzcd} \]
is an isomorphism of Tamarkin--Tsygan calculi between
$\calculus_C(\cM,Q)$ and $\calculus_C(\cF, \fedosov+\tauTpolyF(Q) )$.
\end{corollary}

\subsubsection{Fedosov contractions for $\sD_{\poly}$ and $\sC_{\poly}$ of a dg manifold}

Let $m\in\Dpolym{2}$ be the shifted multiplication
defined in Equation~\eqref{eq:ShiftedMultiplication}.
According to Lemma~\ref{lem:ProdDiffOp_tau}
and Proposition~\ref{pro:Jolly},
% Corollary~\ref{cor:TauPreserveCalH},
we have
\begin{equation}
\big(\tauDpolyF(m)\big)(f, g) = (-1)^{|f|} fg,
\qquad \forall f,g \in C^\infty(\cN).
\end{equation} 
By $\hochschild$ and $\hochschildb$, respectively,
we denote the Hochschild cohomology differential and
Hochschild homology differential of the Fedosov Lie algebroid
$\cF\to \cN$ as in Section~\ref{sec:CalHdgAbd}.
Thus, $\hochschild=\gerstenhaber{\tauDpolyF(m)}{\argument}$
and $\hochschildb=\iL_{\tauDpolyF(m)}$.
By abuse of notation, we still denote $\tauDpolyF(m)$ by $m$.

We perturb the contractions in Corollary~\ref{cor:repositoryD} by
$\iL_{\tauTpolyF(Q)}+\hochschild
=\gerstenhaber{{\tauTpolyF(Q)}}{\argument}
+\gerstenhaber{m}{\argument}$
and $\iL_{\tauDpolyF(Q)+m}=\iL_{\tauDpolyF(Q)}+\hochschildb$,
respectively, and thus we have the following 

\begin{theorem}\label{thm:mainD}
Given a dg manifold $(\cM,Q)$ and a torsion-free affine connection $\nabla$ on
$\cM$, let $\cF\to\cN$ be the Fedosov dg Lie algebroid corresponding to the
graded manifold $\cM$ as in Lemma~\ref{lem:Rome}
and $\fedosov$ the corresponding Fedosov homological vector field on $\cN$.
Then we have contractions
\[ \begin{tikzcd}[cramped]
\Big( \totDpolyM{\bullet} , \gerstenhaber{Q}{\argument}+\hochschild \Big)
\arrow[r, "\tauDpolyQ", shift left] &
\Big( \totDpolyF{\bullet} , \gerstenhaber{\fedosov+\tauTpolyF(Q)}{\argument}+\hochschild \Big)
\arrow[l, "\sigmaDpolyQ", shift left]
\arrow[loop, "\hDpolyQ", out=5, in=-5, looseness=3]
\end{tikzcd} \]
and
\[ \begin{tikzcd}[cramped]
\Big(\totCpolyM{\bullet},\schoutenc{Q}+\hochschildb \Big)
\arrow[r, "\tauDpolyQ", shift left] &
\Big(\totcCpolyF{\bullet},
\schoutenc{\fedosov+\tauTpolyF(Q)}+\hochschildb \Big)
\arrow[l, "\sigmaTpolyQ", shift left]
\arrow[loop, "\hTpolyQ", out=5, in=-5, looseness=3]
\end{tikzcd} \]
such that, the pair of injections $\tauDpolyQ$
respect the operations $\cup$, $\gerstenhaber{\argument}{\argument}$,
$\iI$, $\iL$ and $B$.
%
%Hence, on the level of cohomology,
%both $\tauTpolyQ$ and $\sigmaTpolyQ$
%induce an isomorphism of
%Tamarkin--Tsygan calculi between
%$\calculus_H(\cM,Q)$
%and $\calculus_H(\cFQ)$.
\end{theorem}
\begin{proof}
We follow the notations in Section~\ref{sec:Subcontraction_DC}. 
Apply Lemma~\ref{lem:BiCxFiltration} and Lemma~\ref{lem:BiCxPerturbation}
to the contractions \eqref{eq:Dcontraction1} and \eqref{eq:Ccontraction1} with 
\begin{gather*}
B^s=\bigoplus_{p+q=s}\sD^{0,q,p},
\qquad C^{r,s}=\bigoplus_{p+q = s}
\prescript{\cF}{}{\sD}^{r,q,p}, \\
D=\gerstenhaber{\fedosov}{\argument},
\quad\varrho=\gerstenhaber{\tauTpolyF(Q)}{\argument}+\hochschild,
\quad h=\hDpolyF, \quad\tau=\tauDpolyF,
\quad\sigma=\sigmaDpolyF, \quad d=0,
\end{gather*}
and 
\begin{gather*}
B^s = \prod_{-p+q = s} \sC^{0,q,-p},
\qquad C^{r,s} = \prod_{-p+q = s}
\prescript{\cF}{}{\sC}^{r,q,-p}, \\
D = \schoutenc{\fedosov}, \quad\varrho=\schoutenc{\tauTpolyF(Q)}+\hochschildb,
\quad h=\hDpolyF, \quad\tau=\tauDpolyF, \quad\sigma=\sigmaDpolyF, \quad d=0,
\end{gather*}
respectively. 

Below we verify the following assertions for $\Cpoly{}$:
(i) $(C^{\bullet,\bullet},D,\varrho)$ is a double complex
and (ii) $d_\varrho=\schoutenc{Q}+\hochschildb$.
%for $\Cpoly{}$ here.
The case of $\Dpoly{}$ is similar.

In the case of $\Cpoly{}$, since 
\begin{align*}
[D,\varrho] &= [\schoutenc{\fedosov},\schoutenc{\tauTpolyF(Q)}+\schoutenc{m}] \\
&= [\schoutenc{\fedosov},\schoutenc{\tauTpolyF(Q)}]
+[\schoutenc{\fedosov},\schoutenc{m}] \\
&= \schoutenc{\gerstenhaber{\fedosov}{\tauTpolyF(Q)}}
+\schoutenc{\gerstenhaber{\fedosov}{m}} \qquad
(\text{Regard } \fedosov,\tauDpolyF(Q), \text{ and } m\text{ as elements in } \Dpoly{}(\cN).) \\
&= 0,
\end{align*}
the triple $(C^{\bullet,\bullet},D,\varrho)$ is a double complex. 

Furthermore, for any $\omegaa\in\Cpolym{}$,
\begin{align*}
d_\varrho(\omegaa) &= (d+\sigma\varrho\tau)(\omegaa)
\qquad\text{(by Lemma~\ref{lem:BiCxPerturbation})} \\
&= \sigmaDpolyF\schoutenc{\tauTpolyF(Q)}\tauDpolyF(\omegaa)
+\sigmaDpolyF\iL_{\tauDpolyF(m)}\tauDpolyF(\omegaa) \\
&= \sigmaDpolyF\tauDpolyF(\iL_{Q}(\omegaa))
+\sigmaDpolyF\tauDpolyF(\iL_m(\omegaa))
\qquad\text{(by Corollary~\ref{cor:TauPreserveCalH})} \\
&= \iL_{Q}(\omegaa)+\hochschildb(\omegaa).
\end{align*}

Since the pair of injection maps $\tauDpolyQ$ remains the same
after the perturbation according to Lemma~\ref{lem:BiCxPerturbation},
it preserves all the calculus algebraic operations
by Corollary~\ref{cor:TauPreserveCalH}. This completes the proof.
\end{proof}

As an immediate consequence, we have
\begin{corollary}
The pair of maps $\tauTpolyQ$ (and hence $\sigmaTpolyQ$) below:
\[ \begin{tikzcd}[cramped]
\cohomology{} \Big(\totDpolyM{\bullet},
\schouten{Q}{\argument} + \hochschild \Big)
\arrow[r, "\tauTpolyQ", shift left] & \cohomology{}
\Big(\totDpolyF{\bullet},\schouten{\fedosov+\tauTpolyF(Q)}{\argument}+\hochschild\Big)
\arrow[l, "\sigmaTpolyQ", shift left]
%\arrow[loop, "\hTpolyQ", out=7, in=-7, looseness=3]
\end{tikzcd} \]
and
\[ \begin{tikzcd}[cramped]
\cohomology{} \Big(\totCpolyM{\bullet},\iL_{Q}+ \hochschildb \Big)
\arrow[r, "\tauTpolyQ", shift left] &
\cohomology{} \Big(\totcCpolyF{\bullet},\iL_{\fedosov+\tauTpolyF(Q)} + \hochschildb\Big)
\arrow[l, "\sigmaTpolyQ", shift left]
%\arrow[loop, "\hTpolyQ", out=7, in=-7, looseness=3]
\end{tikzcd} \]
is an isomorphism of Tamarkin--Tsygan calculi between
$\calculus_H(\cM,Q)$
and $\calculus_H(\cF,\fedosov+\tauTpolyF(Q))$.
\end{corollary}

\subsection{Compatibility with the dg formal groupoid structures}

%According to the general theory \cite{MR3456700, Mathieu}\footnote{Kowalzig \cite{MR3456700}
%proved this for nongraded formal groupoids, in fact, for
%more general Hopf algebroids, but the theory extends to
%$\ZZ$-graded situation in a straightforward manner.},
%for a dg Lie algebroid $(\cL,Q)$, the calculus structure
%maps for $\calculus_H(\cL,\cQ)$
%can be expressed explicitly
%in terms of the structure maps of the dg
%formal groupoid $\cJ(\cL)$ of infinity jets
%--- see \cite{MR3456700} or Appendix~\ref{sec:FormalGpdandCalculi}.

% Thus, it suffices to show that the injection maps $\tauDpolyF$
%preserve the formal groupoid structures. 

Similar to Section~\ref{sec:Hopf}, one can consider the universal enveloping algebra
$\enveloping{\cFQ} := \big(\enveloping{\cF},\iL_{\fedosov+\tauDpolyF(Q)}\big)$
and the jet space $\jet{\cFQ}:=\big(\jet{\cF},\iL_{\fedosov+\tauDpolyF(Q)}\big)$
of the Fedosov dg Lie algebroid $\cFQ\to\cNQ$. 
In this section, we aim to obtain an $\cFQ$-version
of Propositions~\ref{pro:Metz} and~\ref{prop:TauFormalGpd}. 

Recall that we have the following identifications:
\begin{align*}
\enveloping{\cF} & \cong \OO\big(\cM,\hat{S}(T^\vee_\cM)\otimes S( T_\cM)\big), \\
\jet{\cF} & \cong \OO\big(\cM,\hat{S}(T^\vee_\cM)\otimes \hat{S}( T_\cM\dual)\big).
\end{align*}
Both spaces $\enveloping{\cF}$ and $\jet{\cF}$ are equipped with the bigrading:
\begin{align*}
\enveloping{\cF}^{r,s} & \cong \Big(\OO^r\big(\cM,\hat{S}(T^\vee_\cM)
\otimes S( T_\cM)\big)\Big)^{r+s}, \\
\jet{\cF}^{r,s} & \cong \Big(\OO^r\big(\cM,\hat{S}(T^\vee_\cM)
\otimes\hat{S}( T_\cM\dual)\big)\Big)^{r+s}.
\end{align*}
With this bigrading, we have the direct-sum total space
$\tot_\oplus^\bullet\big(\enveloping{\cF}\big)$
and $\tot_\oplus^\bullet\big(\jet{\cF}\big)$.
See Remark~\ref{rmk:SumTotIsSmaller} for the difference
between the universal enveloping algebra $\enveloping{\cF}$
(or respectively, the jet space $\jet{\cF}$)
and its direct-sum total space.

%Now we perturb the contractions \eqref{eq:Metz}
%and \eqref{eq:FcontractionJets} by $\iL_{\tauDpolyF(Q)}$.

Similar to Section~\ref{sec:Subcontraction_TA},
we consider the subcontractions of~\eqref{eq:Metz} and~\eqref{eq:FcontractionJets}:
\begin{equation}\label{eq:FcontractionU_sum}
\begin{tikzcd}[cramped]
\Big( \big(\cD(\cM) \big)^\bullet, 0 \Big)
\arrow[r, " \tauDpolyF", shift left] &
\Big(\tot_\oplus^\bullet \big(\enveloping{\cF}\big) , \dQF \Big)
\arrow[l, " \sigmaDpolyF", shift left] \arrow[loop, "\hDpolyF",out=7,in=-7,looseness = 3]
\end{tikzcd}
\end{equation}
and
\begin{equation}\label{eq:FcontractionJet_sum}
\begin{tikzcd}[cramped]
\Big( \big(\jm \big)^\bullet ,0\Big)
\arrow[r, "\breve{\tau}_\natural", shift left] &
\Big(\tot_\oplus^\bullet \big(\jet{\cF}\big) ,\iL_{\fedosov}\Big)
\arrow[l, "\sigma_\natural", shift left]
\arrow[loop, "\breve{h}_\natural", out=7, in=-7, looseness=3]
\end{tikzcd}.
\end{equation}

\begin{lemma}
Given a dg manifold $(\cM,\cQ)$ and a torsion-free affine connection 
$\nabla$ on $\cM$, let $\cF \to \cN$ be the Fedosov dg Lie algebroid
associated with $\cM$ and $\fedosov$ the corresponding Fedosov homological
vector field on $\cN$. We have the following contractions: 
\begin{equation}
\label{eq:FcontractionU_dg}
\begin{tikzcd}[cramped]
\Big( \big(\cD(\cM)\big)^\bullet, \iL_Q \Big)
\arrow[r, " \tauDpolyF", shift left] &
\Big(\tot_\oplus^\bullet \big(\enveloping{\cF}\big) , \iL_{\fedosov+\tauTpolyF(Q)} \Big)
\arrow[l, " \sigmaDpolyQ", shift left] \arrow[loop, "\hDpolyQ",out=7,in=-7,looseness = 3]
\end{tikzcd}
\end{equation}
and
\begin{equation}\label{eq:FcontractionJets_dg}
\begin{tikzcd}[cramped]
\Big( \big(\jm \big)^\bullet, \iL_Q \Big)
\arrow[r, "\breve{\tau}_\natural", shift left] &
\Big(\tot_\oplus^\bullet \big(\jet{\cF}\big) ,\iL_{\fedosov+\tauTpolyF(Q)}\Big)
\arrow[l, "\sigmaDpolyQ", shift left]
\arrow[loop, "\hDpolyQ", out=7, in=-7, looseness=3]
\end{tikzcd}.
\end{equation} 
\end{lemma}
\begin{proof}
This lemma can be demonstrated using the proof of Theorem~\ref{thm:mainD},
disregarding the Hochschild (co)boundary operators. 
Explicitly, we apply Lemma~\ref{lem:BiCxFiltration}
and Lemma~\ref{lem:BiCxPerturbation} to the contractions \eqref{eq:FcontractionU_sum}
and \eqref{eq:FcontractionJet_sum} with 
\begin{gather*}
B^s=\big(\cD(\cM)\big)^s, \qquad C^{r,s}=\enveloping{\cF}^{r,s} ,\\
D=\iL_{\fedosov}, \quad\varrho=\iL_{\tauDpolyF(Q)},
\quad h=\hDpolyF, \quad\tau=\tauDpolyF, \quad\sigma=\sigmaDpolyF, \quad d =0,
\end{gather*}
and 
\begin{gather*}
B^s=\big(\jm\big)^s, \qquad C^{r,s}=\jet{\cF}^{r,s}, \\
D=\schoutenc{\fedosov}, \quad\varrho=\schoutenc{\tauTpolyF(Q)},
\quad h=\hDpolyF, \quad\tau=\tauDpolyF, \quad\sigma=\sigmaDpolyF, \quad d=0.
\end{gather*}
One can verify $d_\varrho=\iL_Q$ by Corollary~\ref{cor:TauPreserveCalH}.
\end{proof}

Since the injections in the contractions \eqref{eq:FcontractionU_dg}
and \eqref{eq:FcontractionJets_dg} are same as the ones
in~\eqref{eq:Metz} and~\eqref{eq:FcontractionJets},
we establish the following theorem by Proposition~\ref{pro:Metz}
and Proposition~\ref{prop:TauFormalGpd},
which we believe that it should be of independent interest.

%Specializing the contraction \eqref{eq:OrlyC}
%to the particular case that $p=1$ and applying the degree shifting funtor
%$\pshift$, we obtain the following contraction:
%\ping{NOT correct;need to see the best way to get this contraction}
%
%\begin{equation}
%\label{eq:Dijon}
%\begin{tikzcd}[cramped]
%\Big(\jm, \schoutenc{Q}\Big)
%\arrow[r, "\tauDpolyF", shift left] &
%\Big(\jfQ, \schoutenc{\fedosov+\tauTpolyF(Q)} \Big)
%\arrow[l, "\sigmaDpolyF", shift left] \arrow[loop, "\hDpolyF", out=7, in=-7, looseness=3]
%\end{tikzcd}
%\end{equation}

%For a dg manifold
%$(\cM, \cQ)$, the Tamarkin--Tsygan calculus
%structure $\calculus_H(\cM, \cQ)$ is completely determined by
%the dg formal groupoid structure on $ \Big(\jm, \schoutenc{Q}\Big)$,
%while the Tamarkin--Tsygan calculus
%structure $\calculus_H(\cFQ,\fedosov)$ of a
%Fedosov dg 	Lie algebroid $\cFQ$ is completely determined by
%the dg formal groupoid $ \Big(\jfQ, \schoutenc{\fedosov+\tauTpolyF(Q)} \Big)$.
%The following result is the key factor responsable
%for Theorem \ref{thm:Germain} we established in previous sections, which
%we believe that it should be of independent interest.

\begin{theorem}
\label{thm:Dijon}
For a dg manifold $(\cM,\cQ)$, let $\cFQ\to\cNQ$ be its Fedosov dg Lie algebroid.
In the contraction \eqref{eq:FcontractionU_dg}, the injection 
\[ \tauDpolyF: \big(\cD(\cM),\iL_Q\big) \to
\big(\enveloping{\cF},\iL_{\fedosov+\tauDpolyF(Q)}\big) \]
is a morphism of dg Hop algebroids.
Likewise, in the contraction \eqref{eq:FcontractionJets_dg}, the injection 
\[ \tauDpolyF: \big(\jm,\iL_Q\big) \to
\big(\jet{\cF},\iL_{\fedosov+\tauDpolyF(Q)}\big) \]
is a morphism of dg formal groupoids.
\end{theorem}

\appendix

\section{Formal groupoid of jets}
\label{sec:FormalGpdJet}

Let $\cL \to \cM$ be a graded Lie algebroid, and
$\Delta:\enveloping{\cL}\to\enveloping{\cL}\otimes_\cR\enveloping{\cL}$
be the comultiplication of $\enveloping{\cL}$ which is defined via a formula
parallel to Equation~\eqref{eq:CoprodUF}. As before, we adopt the Sweedler notation
\[ \Delta(u)=\sum_{(u)}u_{(1)}\otimes u_{(2)} \]
to denote the cocommutative comultiplication of $\enveloping{\cL}$. The jet space 
\begin{equation}\label{eq:jet2}
\jet{\cL}=\Hom_\cR(\enveloping{\cL},\cR)
\end{equation}
is equipped with the following graded formal groupoid structure 
\cite{MR2361096,MR2817646,MR2738352,paper-zero}\footnote{The
non-graded situation was considered in~\cite{MR2361096,MR2817646,MR2738352},
but the approach works for graded Lie algebroid as well.}:

(1) \textbf{(multiplication)} It is clear that $\jet{\cL}$ is a graded
commutative associative $\cR$-algebra, with multiplication 
\begin{equation}\label{eq:mjet}
m: \jet{\cL}\times\jet{\cL} \to \jet{\cL},
\qquad m(\xi,\eta)=\xi*\eta,
\end{equation}
being induced by the comultiplication \eqref{eq:SCE}
on $\enveloping{\cL}$:
\[ \duality{\xi*\eta}{u} = \duality{\xi\otimes\eta}{\Delta(u)}
= \sum_{(u)} {(-1)^{|\eta||u_{(1)}|}} \duality{\xi}{u_{(1)}}
\duality{\eta}{u_{(2)}}, \quad\forall\xi,\eta\in\jet{\cL},
u\in\enveloping{\cL} .\]

(2) \textbf{(source and target maps)}
The source map is the algebra morphism
\begin{equation}
\alpha: \cR\to\jet{\cL}, \quad \duality{\alpha(f)}{u}:=f\epsilon(u)
\end{equation}
while the target map is the algebra anti-homomorphism
\begin{equation}
\beta: \cR\to\jet{\cL}, \quad \duality{\beta(f)}{u}:=
{(-1)^{|f||u|}}\, \epsilon(u\cdot f).
\end{equation}
Here $\epsilon:\enveloping{\cL}\to\cR$ is the counit map of $\enveloping{\cL}$.

%There are two algebra morphisms, source and target maps
%$\alpha, \ \beta: \cR\to \jet{\cL}$ defined by
%\[ \duality{\alpha (f)}{u}:=f\epsilon (u)=f \dot u(1), \ \ 
% \duality{\beta (f)}{u}:= \epsilon (u\cdot f)=u (f) \]

(3) \textbf{(comultiplication)} 
There are two (different) commuting left $\cR$-module structures on
$\jet{\cL}$ defined by the composition of the multiplication with
the source and target maps $\alpha, \ \beta: \cR\to\jet{\cL}$, respectively:
\begin{align*}
f \cdot_\alpha \xi &:= \alpha(f) \xi, \\
f \cdot_\beta \xi &:= \beta(f) \xi,
\end{align*}
for $f\in\cR$, $\xi\in\jet{\cL}$. 
The default left $\cR$-module structure on $\jet{\cL}$
as the dual $\cR$-module $\enveloping{\cL}$ according
to~\eqref{eq:jet2} coincides with the one by the composition of
the multiplication with the source map $\alpha:\cR\to\jet{\cL}$.

Since $\cR$ is graded commutative,
one has four types of tensor products. 
By default, denote by $\jet{\cL}\cotimes_\cR\jet{\cL}$, 
%the complete tesnor product with rese
the quotient of $\jet{\cL}\cotimes_\KK\jet{\cL}$
by the ideal generated by 
$\{\alpha(f)\xi\otimes\eta - {(-1)^{|f||\xi|}} \,
\xi\otimes\alpha(f)\eta \mid \xi,\eta\in\jet{\cL}, \ f\in\cR \}$.
Denote by $\jet{\cL}\cotimes^{\beta \alpha}_\cR\jet{\cL}$,
the quotient of $\jet{\cL}\cotimes_\KK \jet{\cL}$
by the ideal generated by
$\{\beta(f)\xi\otimes\eta - {(-1)^{|f||\xi|}} \,
\xi\otimes\alpha(f)\eta \mid \xi,\eta\in\jet{\cL}, \ f\in\cR \}$. 

Note that, besides being a left $\cR$-module, 
$\enveloping{\cL}$ is also a right $\cR$-module
with the action being the multiplication by $\cR$ from the right. 
We use $\enveloping{\cL}\otimes^{rl}_\cR\enveloping{\cL}$
to denote the quotient of $\enveloping{\cL}\otimes_\KK\enveloping{\cL}$
by the ideal generated by
$\{ u\cdot f\otimes v - u\otimes f\cdot v
\mid u,v\in\enveloping{\cL}, \ f\in\cR \}$.
One can prove that \cite[Lemma~3.16]{MR2817646}
\begin{equation}\label{eq:HKG1}
\jet{\cL}\cotimes^{\beta \alpha}_\cR \jet{\cL}\cong
\Hom_\cR (\enveloping{\cL}\otimes^{rl}_\cR \enveloping{\cL}, \ \cR), 
\end{equation}
where the isomorphism is given by
\begin{equation}\label{eq:HKG}
\xi\otimes\eta\mapsto \Big( u\otimes v\mapsto (-1)^{|u||\eta|}
\duality{\xi}{u\cdot\duality{\eta}{v}} \Big)
,\quad\forall\xi,\eta\in\jet{\cL}; u,v\in\enveloping{\cL}
.\end{equation}
The product $\enveloping{\cL}\otimes_\KK\enveloping{\cL}\to\enveloping{\cL}$
on $\enveloping{\cL}$ descends to a map
$m:\enveloping{\cL}\otimes^{rl}_\cR\enveloping{\cL}\to\enveloping{\cL}$.
The comultiplication on $\jet{\cL}$ is simply
the $\cR$-dual of this map by using the isomorphism \eqref{eq:HKG1}. 
More precisely, the coproduct 
\begin{equation}\label{eq:coproduct1}
\Delta:\jet{\cL}\to\jet{\cL}\cotimes^{\beta \alpha}_\cR\jet{\cL}
\end{equation}
is given by
\begin{equation}\label{eq:coproduct}
\Delta(\xi)(u\otimes^{rl}v):=\duality{\xi}{u\cdot v}
,\qquad\forall\xi\in\jet{\cL}; u,v\in\enveloping{\cL}
.\end{equation}
We often write $\Delta(\xi)=\xi^{(1)}\otimes\xi^{(2)}$.

(4) \textbf{(counit)} The counit for the above comultiplication is given by
\begin{equation}\label{eq:counit}
\varepsilon:\jet{\cL}\to\cR, \quad 
\varepsilon(\xi)=\duality{\xi}{1}
,\qquad\forall\xi\in\jet{\cL}
.\end{equation}

(5) \textbf{(antipode)} Recall that the Grothendieck connection
on $\jet{\cL}$ is the flat $\cL$-connection \cite{MR1913813} defined by
\begin{equation}\label{eq:Grothendieck}
\duality{\nabla_X^G \xi}{u}=\rho_X \duality{\xi}{u}
-(-1)^{\degree{X}\degree{\xi}}\duality{\xi}{X\cdot u}
\end{equation}
for any homogeneous $X\in\XX(\cM)$ and $\xi\in\jet{\cL}$,
and $u\in\enveloping{\cL}$.
By flatness, it induces a left $\enveloping{\cL}$-action on $\jet{\cL}$ 
\[ \nabla^G: \enveloping{\cL} \times \jet{\cL} \to \jet{\cL},
\qquad (u,\xi) \mapsto \nabla^G_u \xi .\]
The morphism $S: \jet{\cL}\to\jet{\cL}$ defined by
\begin{equation}\label{eq:antipode}
\duality{S\xi}{u}={(-1)^{|\xi||u|}} \duality{\nabla^G_u \xi}{1},
\quad\quad\forall\xi\in\jet{\cL},\ u\in\enveloping{\cL} 
\end{equation}
is the antipode.

\begin{theorem}[{\cite{paper-zero}, \cite[Theorem~3.17]{MR2817646},
and \cite[Appendix~A]{MR2738352}}]
For any graded Lie algebroid $\cL$, the jet space $\jet{\cL}$,
equipped with the operations (1) to (5), is a graded formal groupoid.
% i.e. left bialgebroid with involutive antipode.
\end{theorem}

\section{Homological perturbation}\label{Vilnius}

A \emph{contraction} consists of 
a pair of cochain complexes $(W,d_W)$ and $(V,d_V)$,
a pair of chain maps $\sigma:W\to V$ and $\tau:V\to W$,
and an endomorphism $h:W\to W[-1]$ of the graded module $W$
satisfying the following five relations:
\begin{equation}\label{Lasne}
\begin{gathered} \sigma\tau=\id_V, \qquad
\id_W-\tau\sigma=h d_W+d_W h, \\
\sigma h=0, \qquad h\tau=0, \qquad h^2=0
.\end{gathered}
\end{equation}
We symbolize such a contraction by a diagram
\begin{equation}\label{Incourt}
\begin{tikzcd}
(V,d_V) \arrow[r, "\tau", shift left] &
(W,d_W) \arrow[l, "\sigma", shift left]
\arrow["h", loop right]
\end{tikzcd}
.\end{equation}

It follows immediately from the relations \eqref{Lasne} that $hd_W h=h$.
Therefore, the endomorphisms $hd_W$, $d_W h$, and $hd_W+d_W h$ of $W$ are projection operators.
%\[ (-hd_W)(-hd_W)=-hd_W \qquad\text{and}\qquad (-d_W h)(-d_W h)=-d_W h .\]
Likewise, $\tau\sigma=\id_W-(hd_W+d_W h)$ is a projection as well.
Since $\id=\tau\sigma+hd_W+d_W h$ and the composition of any two of the projections
$\tau\sigma$, $hd_W$, and $d_W h$ is zero, we have the direct sum decomposition
\begin{equation}\label{Clabecq} W = \im(\tau\sigma)\oplus\im(hd_W)\oplus\im(d_W h) .\end{equation}
The injective chain map $\tau$ identifies the cochain complex $(V,d_V)$ to the subcomplex
$\im(\tau\sigma)$ of $(W,d_W)$. The complementary subcomplex $\im(hd_W)\oplus\im(d_W h)$ of $(W,d_W)$
is precisely the kernel of the surjective chain map $\sigma$.

\begin{remark}
One can easily show that any pair of endomorphisms $d_W:W\to W[1]$ and $h:W\to W[-1]$
of a graded module $W$ satisfying the three relations $d_W^2=0$; $h^2=0$; and $hd_W h=h$
determines a subcomplex $V:=\ker(hd_W+d_W h)$
of the cochain complex $(W,d_W)$
and a contraction
\[ \begin{tikzcd}
(V,d_V) \arrow[r, "\tau", shift left] &
(W,d_W) \arrow[l, "\sigma", shift left]
\arrow["h", loop right]
\end{tikzcd} .\]
Here the injection $\tau$ is the inclusion 
$V\into W$ and the surjection $\sigma$ is 
determined by the relation 
$\tau\sigma=\id_W-(hd_W+d_W h)$. See \cite[Appendix~A]{MR4665716}.
\end{remark}

A filtration
\[ \cdots \subseteq F_{m-1}W \subseteq F_{m}W \subseteq F_{m+1}W \subseteq \cdots \]
on a vector space $W$ is said to be \emph{exhaustive} if $W=\bigcup_m F_m W$
and \emph{complete} if $ W= \displaystyle\varprojlim_{m \to -\infty} \dfrac{W}{F_m W}$.

A \emph{perturbation} of a cochain complex
\[ \begin{tikzcd} \cdots \arrow{r} & W^{n-1} \arrow{r}{d_W}
& W^{n} \arrow{r}{d_W} & W^{n+1} \arrow{r} & \cdots \end{tikzcd} \]
is an endomorphism $\varrho:W\to W[1]$ of the graded module $W$,
which satisfies \[ (d_W-\varrho)^2=0 \]
so that $d_W-\varrho$ is a new differential on $N$.

A \emph{small perturbation} of a contraction \eqref{Incourt}
is a perturbation of the cochain complex $(W,d_W)$ such that
there exists an exhaustive complete filtration $F_\bullet W$ on $W$
with the property \[ (h \varrho )(F_{m} W)\subset F_{m-1} W .\]

If $\varrho$ is a small perturbation of the contraction \eqref{Incourt}, then

\begin{itemize}
\item the series $\sum_{k=0}^\infty(\varrho h)^k$
and $\sum_{k=0}^\infty(h\varrho)^k$ converge
and define a pair of endomorphisms of $W$;
\item the endomorphisms $\id_W-\varrho h$
and $\id_W-h\varrho$ of $W$ are both invertible;
\item we have $(\id_W-\varrho h)^{-1}=\sum_{k=0}^\infty(\varrho h)^k$
and $(\id_W-h\varrho)^{-1}=\sum_{k=0}^\infty(h\varrho)^k$;
\end{itemize}
Consider the endomorphism $\machin{\varrho}:W\to W[1]$ defined by
\[ \machin{\varrho}=(\id_W-\varrho h)^{-1}\varrho
=\sum_{k=0}^\infty(\varrho h)^k\varrho
=\sum_{k=0}^\infty\varrho(h\varrho)^k=\varrho(\id_W-h\varrho)^{-1} .\]
Rewriting the equation above as
\[ (\id_W-\varrho h)\machin{\varrho}=\varrho=\machin{\varrho}(\id_W-h\varrho) \]
and then
\[ \varrho(\id+h\machin{\varrho})=\machin{\varrho}=(\machin{\varrho}h+\id)\varrho ,\]
we see that we could also have described $\breve{\varrho}$
as the fixed point of either of the operators $x\mapsto\varrho(\id+hx)$
and $y\mapsto(yh+\id)\varrho$ acting on $\Hom(W,W[1])$.

We note that \[ (\id_W-\varrho h)^{-1}=\sum_{k=0}^\infty
(\varrho h)^k=\id_W+\machin{\varrho}h \]
and \[ (\id_W-h\varrho)^{-1}=\sum_{k=0}^\infty
(h\varrho)^k=\id_W+h\machin{\varrho} .\]

We refer the reader to~\cite[\S1]{MR1109665} for a brief history of the following lemma.

\begin{lemma}[Homological Perturbation \cite{MR0220273}]
\label{Riga}
Let \begin{equation}\label{Helecine}
\cR=\left[ \begin{tikzcd}
(V,d_V) \arrow[r, "\tau", shift left] &
(W,d_W) \arrow[l, "\sigma", shift left]
\arrow["h", loop right]
\end{tikzcd} \right] \end{equation} be a contraction
and let $\varrho$ be a small perturbation of the contraction $\cR$. 
Then, the operator $\sigma\machin{\varrho}\tau$
is a perturbation of the cochain complex $(V,d_V)$ and the operators
\[ \sigma+\sigma\machin{\varrho}h=\sigma(\id+\breve{\varrho}h),
\qquad \tau+h\machin{\varrho}\tau=(\id+h\breve{\varrho})\tau,
\qquad\text{and}\qquad h+h\machin{\varrho}h \]
determine a contraction 
\begin{equation}\label{Ramilies}
\cR^\varrho=\left[ \begin{tikzcd}[column sep=huge]
(V,d_V-\sigma\machin{\varrho}\tau)
\arrow[r, "(\id+h\breve{\varrho})\tau", shift left] &
(W,d_W-\varrho)
\arrow[l, "\sigma(\id+\breve{\varrho}h)", shift left]
\arrow["h+h\machin{\varrho}h", loop right]
\end{tikzcd} \right] .\end{equation}
%constitutes a new filtered contraction.
\end{lemma}

\section*{Acknowledgments}

We gratefully acknowledge the hospitality of the following institutions, which facilitated the completion of this project:
Pennsylvania State University (Liao),
National Center for Theoretical Sciences (Liao, Stiénon, Xu),
National Tsing Hua University (Stiénon),
Institut des Hautes Études Scientifiques (Liao),
Institut Henri Poincaré (Liao, Stiénon, Xu),
and Institut Mittag-Leffler (Stiénon, Xu).

We also thank
Ruggero Bandiera,
Sophie Chemla,
Vasily Dolgushev,
Ezra Getzler,
Niels Kowalzig,
Camille Laurent-Gengoux,
Pavol Ševera,
Boris Tsygan,
Kai Wang,
and
Thomas Willwacher
for insightful discussions and valuable comments.

%\section*{Declarations}
%
%\subsection*{Funding}
%
%Xu received partial financial support from
%the National Science Foundation under award DMS-2302447
%and from the Simons Foundation under award MP-TSM-00002272. 
%Sti\'enon and Xu were supported by the Swedish Research Council
%under grant no.\ 2021-06594 while in residence
%at Institut Mittag-Leffler in Djursholm, Sweden
%during the spring trimester of 2025.
%%
%Furthermore, Several institutions provided partial support
%which allowed the completion of this work:
%Pennsylvania State University (Liao),
%National Center for Theoretical Science (Liao, Sti\'enon, Xu),
%National Tsing Hua University (Sti\'enon),
%Institut des Hautes \'{E}tudes Scientifiques (Liao),
%Institut Henri Poincar\'e (Liao, Sti\'enon, Xu),
%and Institut Mittag-Leffler (Sti\'enon, Xu).
%
%\subsection*{Competing Interests}
%
%The authors have no relevant financial or non-financial
%competing interests to declare that are relevant
%to the content of this article.
%
%\subsection*{Data Availability}
%
%All data generated or analysed during this 
%study are included in this published article.

\printbibliography

\end{document}